\documentclass[12pt]{amsart}
\usepackage{amsmath,amssymb,amscd}
\usepackage[usenames,dvipsnames]{color}
\usepackage{graphicx}
\usepackage{longtable}

\newtheorem{propo}{Proposition}[section]
\newtheorem{corol}[propo]{Corollary}
\newtheorem{theor}[propo]{Theorem}
\newtheorem{lemma}[propo]{Lemma}

\theoremstyle{definition}
\newtheorem{defin}[propo]{Definition}

\theoremstyle{remark}
\newtheorem{remar}[propo]{Remark}

\newcommand{\algo}[6]
{\vspace{11pt}\addtocounter{propo}{1}
\noindent{\bf Algorithm \arabic{section}.\arabic{propo}.}
{\bf #1}(#2)\\ {\it #3}.\\
{\bf Input:} #4\\
{\bf Output:} #5\\
\newcounter{#1}
\begin{list}{\textbf{\arabic{#1}.}}{\usecounter{#1}}
#6\end{list}\vspace{3pt}}

\numberwithin{equation}{section}

\newcommand{\al }{\alpha }

\newcommand{\cC }{\mathcal{C}}

\newcommand{\cD }{\Gamma}

\newcommand{\Ac }{\mathcal{A}}
\newcommand{\Cm }{C}
\newcommand{\cm }{c}

\newcommand{\id }{\mathrm{id}}

\newcommand{\NN }{\mathbb{N}}
\newcommand{\QQ }{\mathbb{Q}}
\newcommand{\RR }{\mathbb{R}}
\newcommand{\ZZ }{\mathbb{Z}}
\newcommand{\rer }[1]{(R\re)^{#1}}
\newcommand{\rfl }{\rho }

\newcommand{\Ob }{\mathrm{Ob}}

\newcommand{\re }{^\mathrm{re}}
\newcommand{\rsC }{\mathcal{R}}
\newcommand{\s }{\sigma }

\newcommand{\Wg }{\mathcal{W}}
\DeclareMathOperator{\Aut}{Aut}

\DeclareMathOperator{\Hom}{Hom}
\DeclareMathOperator{\hg}{ht}
\DeclareMathOperator{\Sym}{Sym}

\newcommand{\coscorf}{coscorf Cartan scheme }
\newcommand{\coscorfn}{coscorf Cartan scheme}
\newcommand{\coscorfs}{coscorf Cartan schemes }
\newcommand{\Coscorfs}{Coscorf Cartan schemes }
\newcommand{\coscorfsn}{coscorf Cartan schemes}

\newcommand{\rootsubset}{root subset }
\newcommand{\rootsubsetn}{root subset}
\newcommand{\rootsubsets}{root subsets }
\newcommand{\rootsubsetsn}{root subsets}

\newcommand{\frto}{\eta}

\title{Finite Weyl groupoids}

\author{M.~Cuntz}
\address{Michael Cuntz,
Fachbereich Mathematik,
Universit\"at Kaiserslau\-tern,
Postfach 3049,
D-67653 Kaiserslautern, Germany}
\email{cuntz@mathematik.uni-kl.de}

\author{I.~Heckenberger}
\thanks{I.H. is
supported by the German Research Foundation (DFG) in the Heisenberg program}
\address{Istv\'an Heckenberger, Philipps-Universit\"at Marburg,
Fachbereich Mathematik und Informatik,
Hans-Meerwein-Stra\ss e,
D-35032 Marburg, Germany}
\email{heckenberger@mathematik.uni-marburg.de}

\begin{document}

\begin{abstract}
Using previous results concerning the rank two and rank three cases, all
connected simply connected Cartan schemes for which the real roots form a
finite irreducible root system of arbitrary rank are determined.
As a consequence one obtains the list of all crystallographic arrangements,
a large subclass of the class of simplicial hyperplane arrangements.
Supposing that the rank is at least three, the classification yields Cartan
schemes of type $A$ and $B$, an infinite family of series involving the types
$C$ and $D$, and $74$ sporadic examples.
\end{abstract}

\maketitle

\section{Introduction}

In the 1970s, simplicial arrangements became a popular subject of study after
Deligne \cite{MR0422673} proved that the complement of a complexified finite
simplicial real hyperplane arrangement is an Eilenberg-MacLane space of type
$K(\pi ,1)$ for some group $\pi $. It is known that the cohomology ring of
such a space coincides with the group cohomology $H^*(\pi ,\ZZ )$.
Previously, Brieskorn \cite{MR0293615} had identified the fundamental groups of
complements of complexified Coxeter arrangements as pure Artin braid groups.
In 1980, the result of Brieskorn was extended to all real simplicial
arrangements by Orlik and Solomon \cite{MR558866} based on algebraic
constructions of lattices which are related to Bac{\l}awskis work on geometric
lattices \cite{MR0387086}.

Simplicial arrangements in the real projective plane were introduced by Melchior
\cite{MR0004476} in 1941. Their classification remained so far an open problem.
Gr\"unbaum \cite{p-G-09} provides a conjecturally complete list which contains three
infinite series and a large number of exceptional examples. In higher dimensional
spaces only a few examples of simplicial arrangements are known.

Let $\Ac $ be a simplicial arrangement of finitely many real hyperplanes in a
Euclidean space $V$ and let $R$ be a
set of nonzero covectors such that $\Ac =\{\alpha ^\perp \,|\,\alpha \in R\}$.
Assume that $\RR\alpha \cap R=\{\pm\alpha\}$ for all $\alpha\in R$.
The pair $(\Ac ,R)$ is called crystallographic, see \cite[Def.\ 2.3]{p-C10},
if for any chamber $K$ the elements of $R$
are integer linear combinations of the covectors defining the walls of $K$.
For example, crystallographic Coxeter groups give rise to crystallographic
arrangements in this sense, but there are many other. In this paper we solve the
natural problem of classifying crystallographic arrangements by considering Cartan
schemes, their root systems, and their Weyl groupoids.

Weyl groupoids have been introduced by the second author in
\cite{p-H-06} to obtain finiteness properties of Nichols algebras of
diagonal type. The Weyl groupoid provides simplification, generalization, and unification
of related finiteness results by Rosso \cite{MR1632802}
and Andruskiewitsch and Schneider \cite{MR1780094}. A very general definition of the
Weyl groupoid of a Nichols algebra and the necessary structural results have
been obtained by Andruskiewitsch, Schneider, and the second author in
\cite{p-AHS}. An axiomatic approach to Weyl groupoids via Cartan schemes and root systems
was developed in a series of papers by Yamane and the authors
\cite{p-HY-08,p-CH09a,p-CH09b,p-CH09c}.
For historical and practical reasons, the emphasis was put on connected simply connected
Cartan schemes such that the real roots form a finite irreducible root system.
Such Cartan schemes will be called coscorf Cartan schemes.

The connection to simplicial arrangements was established successively
in \cite{p-CH09c,p-HW-10,p-C10}.
If the real roots of a connected Cartan scheme form a finite irreducible root system, then
they define a simplicial arrangement similarly as in the construction in
\cite[Sect.\,1.15]{b-Hum} for Coxeter groups.
This was first observed in \cite{p-CH09c} in the case of rank three and then proved
in full generality in \cite{p-HW-10}.
The final step was achieved in \cite{p-CH09c} where it was shown that crystallographic
arrangements can be described axiomatically as coscorf Cartan schemes.
Therefore it is natural to try a classification of simplicial arrangements via Cartan schemes.

Coscorf Cartan schemes of rank at most three have been classified by the authors,
see \cite{p-CH09b} and \cite{p-CH09c}.
The classification of rank two is surprisingly nice:
There is a natural bijection between the set of coscorf Cartan schemes of rank two
with $2n$ objects and the triangulations of a convex $n$-gon by non-intersecting
diagonals \cite{p-CH09d}. In contrast, up to equivalence
there are only finitely many coscorf Cartan schemes of rank three.
In the present paper we treat the general case. Our main result is the following.

\begin{theor}
There are exactly three families of connected simply connected Cartan schemes for which
the real roots form a finite irreducible root system:
\begin{enumerate}
\item The family of Cartan schemes of rank two parametrized by triangulations of a
convex $n$-gon by non-intersecting diagonals.
\item For each rank $r>2$, the standard Cartan schemes of type $A_r$, $B_r$, $C_r$
and $D_r$, and a series of $r-1$ further Cartan schemes described explicitly in
Thm.\ \ref{thm:main_rank_gt8}.
\item A family consisting of $74$ further ``sporadic'' Cartan schemes (including those
of type $F_4$, $E_6$, $E_7$ and $E_8$).
\end{enumerate}
\end{theor}

\begin{remar}
We classify connected simply connected Cartan
schemes $\cC$ up to equivalence in the sense of \cite[Def.\ 2.1]{p-CH09a},
such that $\rsC\re(\cC)$ is a finite root system of type $\cC$. To obtain a
classification of connected Cartan schemes such that $\rsC\re(\cC)$ is a
finite root system, one additionally has to classify quotients
(inverse of coverings) of simply connected Cartan schemes, which
amounts to classify all subgroups of the automorphism group of the minimal
quotient of $\cC$, see Def.\ \ref{auto}.
\end{remar}

As mentioned above, the result is known for Cartan schemes of rank at most three.
We split the proof of the remaining part of the theorem in two cases depending on the rank.

We obtain the classification in rank $r$, $3<r\le 8$ by an algorithm
which enumerates all root systems of \coscorfsn.
We use the knowledge of rank three as a starting point and then
inductively classify \coscorfs of rank $r$ using the classification
of coscorf Cartan schemes of rank $r-1$.
The structure of the algorithm is similar to the one given in \cite{p-CH09c}, but
additional algorithmic ideas and further improvements of the
computational techniques are needed to make the implementation practicable.

The classification in rank greater than eight mainly relies on the analysis
of the Dynkin diagrams corresponding to the Cartan matrices of
the Cartan schemes. The simplicity of the arguments suggests a similar approach for the
Cartan schemes of lower rank. However, this is misleading, since in lower rank there are
many sporadic examples making the argumentations much more difficult. In particular,
single Cartan matrices of non-standard sporadic Cartan schemes contain only little
information about the roots at the corresponding object.

This paper is organized as follows.
In Section \ref{gen_res} we repeat the definitions of Cartan schemes, Weyl groupoids,
root systems and crystallographic arrangements.
Section \ref{sec:rank_gt_eight} is divided into two subsections: In the first one we
determine the Dynkin diagrams of finite Weyl groupoids of rank greater than eight.
In the second subsection we give an explicit description of the root systems of all
\coscorfs of rank greater than eight which are not standard, i.e.\ which
have at least two different Cartan matrices.
In Section \ref{sec:rank_le_eight} we describe an algorithm which enumerates all
coscorf Cartan schemes of rank at most eight.
In the appendix we collect invariants of coscorf Cartan schemes.
Finally we give a list of root systems that contains for each sporadic example
the roots of precisely one object. We explain in which sense this object is canonical.

\section{Weyl groupoids and crystallographic arrangements}\label{gen_res}

\subsection{Cartan schemes and root systems}

We briefly recall the notions of Cartan schemes, Weyl groupoids
and root systems, following \cite{p-CH09a,p-CH09b}.
The foundations of the general theory have been developed in
\cite{p-HY-08} using a somewhat different terminology.

\begin{defin}
Let $I$ be a non-empty finite set and
$\{\alpha_i\,|\,i\in I\}$ the standard basis of $\ZZ ^I$.
By \cite[\S 1.1]{b-Kac90} a {\it generalized Cartan matrix}
$\Cm =(\cm _{ij})_{i,j\in I}$
is a matrix in $\ZZ ^{I\times I}$ such that
\begin{enumerate}
\item[(M1)] $\cm _{ii}=2$ and $\cm _{jk}\le 0$ for all $i,j,k\in I$ with
  $j\not=k$,
\item[(M2)] if $i,j\in I$ and $\cm _{ij}=0$, then $\cm _{ji}=0$.
\end{enumerate}
\end{defin}

\begin{defin}
Let $A$ be a non-empty set, $\rfl _i : A \to A$ a map for all $i\in I$,
and $\Cm ^a=(\cm ^a_{jk})_{j,k \in I}$ a generalized Cartan matrix
in $\ZZ ^{I \times I}$ for all $a\in A$. The quadruple
\[ \cC = \cC (I,A,(\rfl _i)_{i \in I}, (\Cm ^a)_{a \in A})\]
is called a \textit{Cartan scheme} if
\begin{enumerate}
\item[(C1)] $\rfl _i^2 = \id$ for all $i \in I$,
\item[(C2)] $\cm ^a_{ij} = \cm ^{\rfl _i(a)}_{ij}$ for all $a\in A$ and
  $i,j\in I$.
\end{enumerate}
\end{defin}

\begin{defin}
Let $\cC = \cC (I,A,(\rfl _i)_{i \in I}, (\Cm ^a)_{a \in A})$ be a
Cartan scheme. For all $i \in I$ and $a \in A$ define $\s _i^a \in
\Aut(\ZZ ^I)$ by
\begin{align}
\s _i^a (\alpha_j) = \alpha_j - \cm _{ij}^a \alpha_i \qquad
\text{for all $j \in I$.}
\label{eq:sia}
\end{align}
The \textit{Weyl groupoid of} $\cC $
is the category $\Wg (\cC )$ such that $\Ob (\Wg (\cC ))=A$ and
the morphisms are compositions of maps
$\s _i^a$ with $i\in I$ and $a\in A$,
where $\s _i^a$ is considered as an element in $\Hom (a,\rfl _i(a))$.
The category $\Wg (\cC )$ is a groupoid in the sense that all morphisms are
isomorphisms.
\end{defin}
As above, for notational convenience we will often neglect upper indices referring to
elements of $A$ if they are uniquely determined by the context. For example,
the morphism $\s _{i_1}^{\rfl _{i_2}\cdots \rfl _{i_k}(a)}
\cdots \s_{i_{k-1}}^{\rfl _{i_k(a)}}\s _{i_k}^a\in \Hom (a,b)$,
where $k\in \NN $, $i_1,\dots,i_k\in I$, and
$b=\rfl _{i_1}\cdots \rfl _{i_k}(a)$,
will be denoted by $\s _{i_1}\cdots \s _{i_k}^a$ or by
$\id ^b\s _{i_1}\cdots \s_{i_k}$.
The cardinality of $I$ is termed the \textit{rank of} $\Wg (\cC )$.
\begin{defin}
A Cartan scheme is called \textit{connected} if its Weyl grou\-poid
is connected, that is, if for all $a,b\in A$ there exists $w\in \Hom (a,b)$.
The Cartan scheme is called \textit{simply connected},
if $\Hom (a,a)=\{\id ^a\}$ for all $a\in A$.

Two Cartan schemes $\cC =\cC (I,A,(\rfl _i)_{i\in I},(\Cm ^a)_{a\in A})$
and $\cC '=\cC '(I',A',$
$(\rfl '_i)_{i\in I'},({\Cm '}^a)_{a\in A'})$
are termed \textit{equivalent}, if there are bijections $\varphi _0:I\to I'$
and $\varphi _1:A\to A'$ such that
\begin{align}\label{eq:equivCS}
\varphi _1(\rfl _i(a))=\rfl '_{\varphi _0(i)}(\varphi _1(a)),
\qquad
\cm ^{\varphi _1(a)}_{\varphi _0(i) \varphi _0(j)}=\cm ^a_{i j}
\end{align}
for all $i,j\in I$ and $a\in A$. We write then $\cC\cong \cC'$.
\end{defin}

Let $\cC $ be a Cartan scheme. For all $a\in A$ let
\[ \rer a=\{ \id ^a \s _{i_1}\cdots \s_{i_k}(\alpha_j)\,|\,
k\in \NN _0,\,i_1,\dots,i_k,j\in I\}\subseteq \ZZ ^I.\]
The elements of the set $\rer a$ are called \textit{real roots} (at $a$).
The pair $(\cC ,(\rer a)_{a\in A})$ is denoted by $\rsC \re (\cC )$.
A real root $\alpha\in \rer a$, where $a\in A$, is called positive
(resp.\ negative) if $\alpha\in \NN _0^I$ (resp.\ $\alpha\in -\NN _0^I$).
In contrast to real roots associated to a single generalized Cartan matrix,
$\rer a$ may contain elements which are neither positive nor negative. A good
general theory, which is relevant for example for the study of Nichols
algebras, can be obtained if $\rer a$ satisfies additional properties.

\begin{defin}
Let $\cC =\cC (I,A,(\rfl _i)_{i\in I},(\Cm ^a)_{a\in A})$ be a Cartan
scheme. For all $a\in A$ let $R^a\subseteq \ZZ ^I$, and define
$m_{i,j}^a= |R^a \cap (\NN_0 \alpha_i + \NN_0 \alpha_j)|$ for all $i,j\in
I$ and $a\in A$. We say that
\[ \rsC = \rsC (\cC , (R^a)_{a\in A}) \]
is a \textit{root system of type} $\cC $, if it satisfies the following
axioms.
\begin{enumerate}
\item[(R1)]
$R^a=R^a_+\cup - R^a_+$, where $R^a_+=R^a\cap \NN_0^I$, for all
$a\in A$.
\item[(R2)]
$R^a\cap \ZZ\alpha_i=\{\alpha_i,-\alpha_i\}$ for all $i\in I$, $a\in A$.
\item[(R3)]
$\s _i^a(R^a) = R^{\rfl _i(a)}$ for all $i\in I$, $a\in A$.
\item[(R4)]
If $i,j\in I$ and $a\in A$ such that $i\not=j$ and $m_{i,j}^a$ is
finite, then
$(\rfl _i\rfl _j)^{m_{i,j}^a}(a)=a$.
\end{enumerate}
\end{defin}

The axioms (R2) and (R3) are always fulfilled for $\rsC \re $.
The root system $\rsC $ is called \textit{finite} if for all $a\in A$ the
set $R^a$ is finite. By \cite[Prop.\,2.12]{p-CH09a},
if $\rsC $ is a finite root system
of type $\cC $, then $\rsC =\rsC \re $, and hence $\rsC \re $ is a root
system of type $\cC $ in that case.

In \cite[Def.\,4.3]{p-CH09a} the concept of an \textit{irreducible}
root system of type
$\cC $ was defined. By \cite[Prop.\,4.6]{p-CH09a}, if $\cC $ is a
Cartan scheme and $\rsC $ is a finite root system of type $\cC $, then $\rsC $
is irreducible if and only if
for all $a\in A$
the generalized Cartan matrix $C^a$ is indecomposable.
If $\cC $ is also connected, then it suffices
to require that there exists $a\in A$ such that $C^a$ is indecomposable.

\subsection{\Coscorfs and arrangements}

\begin{defin}
In this article, we will abbreviate a connected simply connected Cartan scheme
for which the real roots form a finite root system by a
{\it coscorf Cartan scheme} ({\bf co}nnected {\bf s}imply {\bf co}nnected,
{\bf r}eal roots, {\bf f}inite).
\end{defin}

Although we will not need it, we reproduce the definition of a crystallographic
arrangement \cite[Def.\ 2.3]{p-C10} because it demonstrates how large the class
of arrangements which we classify actually is.
\begin{defin}\label{cryarr}
Let $(\Ac,V)$ be a simplicial arrangement and $R\subseteq V$ a finite set
such that $\Ac = \{ \alpha^\perp \mid \alpha \in R\}$ and $\RR\alpha\cap R=\{\pm \alpha\}$
for all $\alpha \in R$. For a chamber $K$ of $\Ac$ let $B^K$ denote
the set of normal vectors in $R$ of the walls of $K$ pointing to the inside.
We call $(\Ac,R)$ a {\it crystallographic arrangement} if
\begin{itemize}
\item[(I)] \quad $R \subseteq \sum_{\alpha \in B^K} \ZZ \alpha$\quad for all chambers $K$.
\end{itemize}
\end{defin}
As mentioned in the introduction, by \cite[Thm.\ 1.1]{p-C10} all results on \coscorfs
directly apply to crystallographic arrangements:
\begin{theor}
There is a one-to-one correspondence between crystallographic arrangements 
and \coscorfs (up to equivalence on both sides).
\end{theor}

\begin{defin}\label{auto}
Let $\cC$ be a \coscorf and $a\in A$. Then we call
\[ \Aut(\cC,a):=\{w \in \Hom(a,b) \mid b\in A,\:\: R^a=R^b \} \]
the \textit{automorphism group of $\cC$ at $a$}. This is a finite subgroup
of $\Aut(\ZZ^r)$ because the number of all morphisms is finite.
Since $\cC$ is connected, $\Aut(\cC,a)\cong \Aut(\cC,b)$ for all $a,b\in A$.
We will therefore write $\Aut(\cC)$ if we are only interested in the isomorphism
class of the group.
\end{defin}
\begin{remar}
If $\cC$ is a \coscorf then it is simply connected. The automorphism group of
$\cC$ is the automorphism group of an object of the Cartan scheme obtained from $\cC$ by
identifying all objects with equal root systems (the ``smallest'' quotient,
see \cite[Def.\ 3.1]{p-CH09b} for the definition of coverings).
If we have $n$ objects in $\cC$ and $m$ different root systems, then
$m|\Aut(\cC,a)|=n$.
\end{remar}

\subsection{Diagrams}

\begin{defin}
Let $r,s\in\NN$ with $r\le s$.
We will say that a finite set $\Lambda\subseteq\ZZ^s$
is a \textit{root set of rank $r$} if there exists
a Cartan scheme $\cC$ of rank $r$ and an injective linear map $w : \ZZ^r\rightarrow\ZZ^s$
such that $w(\rer a)=\Lambda$ for some object $a$.
We call the set $\{w(\alpha_1),\ldots,w(\alpha_r)\}$ a {\it base} of $\Lambda$.
If $\rsC(\cC)$ is irreducible, then we call $\Lambda$ {\it irreducible}.

If $U\le \RR^r$ is a subspace and $\Lambda' = \rer a\cap U$ for some object $a$,
then we call $\Lambda'$ a {\it \rootsubsetn} of $\cC$.
Remark that if $\Lambda'\ne\emptyset$, then it is a root set by \cite[Cor.\ 2.5]{p-CH09c}
(Cor.\ \ref{cor:genposk}).
Sometimes, for $\beta_1,\ldots,\beta_k\in \rer a$ we will write
$\langle\beta_1,\ldots,\beta_k\rangle$ for the \rootsubset
$\rer a\cap\sum_{i=1}^k \RR\beta_i$.
\end{defin}

Remember the following fact (\cite[Cor.\ 2.9]{p-CH09c}):
\begin{lemma}\label{co:cij}
Let $\cC$ be a Cartan scheme and assume that $\rsC \re (\cC )$ is a finite root system.
Let $a\in A$, $k\in \ZZ $, and $i,j\in I$ such that $i\not=j$.
Then $\al _j+k\al _i\in \rer a$ if and only if $0\le k\le -c^a_{i j}$.
\end{lemma}

\begin{defin}\label{cartan_maps}
To a finite set $\Lambda\subseteq \ZZ^r$ we associate a matrix
$C_\Lambda=(c_{ij})_{1\le i,j \le r}$ given by
\[ c_{ij} = -\max \{ k\mid k\alpha_i+\alpha_j\in \Lambda \}, \quad c_{ii}=2 \]
for $1\le i,j \le r$ and $i\ne j$.
The matrix $C_\Lambda$ is a generalized Cartan matrix and it defines linear maps
$\sigma_i : \ZZ^r\rightarrow \ZZ^r$, $i=1,\ldots,r$ via
\[ \sigma_i(\alpha_j) = \alpha_j-c_{ij}\alpha_i \]
for all $j=1,\ldots,r$.
\end{defin}
\begin{remar}
For example, if $\cC$ is a \coscorf and $a\in A$, then $C^a=C_{\rer{a}_+}$ and
$\sigma_i=\sigma_i^a$ for $i=1,\ldots,r$ by Lemma \ref{co:cij}.
\end{remar}

\begin{defin}\label{def_standard}
As in \cite[Def.\ 3.1]{p-CH09a}, we call a \coscorf \textit{standard} if all its Cartan
matrices are equal. Note that up to coverings, $\Wg(\cC)$ is a Weyl group in this case.
\end{defin}

\begin{defin}
Let $\cC$ be a Cartan scheme and $a$ an object.
The {\it Dynkin diagram} $\cD^a$ at $a$ is a labeled directed graph given
by the Cartan matrix $C^a=(c^a_{ij})_{i,j}$ in the following way:
The vertices are the elements of $I$.
Vertices $i,j\in I$ with $i\ne j$ are connected by an arrow pointing to $i$
and labeled $-c^a_{ij}$ if and only if $c^a_{ij}\ne 0$.

When drawing the diagrams, if $c^a_{ij}=c^a_{ji}=-1$ then instead
of drawing two labeled arrows we just connect $i$ and $j$ by an edge.
If $c^a_{ij}\ne c^a_{ji}=-1$ then we only draw the arrow labeled $c^a_{ij}$.
\end{defin}

For an object $a$, we will write ``{\it $a$ is of Dynkin type $X$}'' if its Dynkin
diagram is of type $X$.

\section{Finite \coscorfs of rank $>8$}
\label{sec:rank_gt_eight}

In this section we use that all \coscorfs of rank $<9$ are as
in Section \ref{sec:rank_le_eight}, see Thm.\ \ref{th:class_upto8}.
In particular, we have a complete list of all their Dynkin diagrams
(p.\ \pageref{fig_dynkin}, Figure \ref{fig_dynkin}).

\subsection{The Dynkin diagrams}

Recall \cite[Cor.\ 2.5]{p-CH09c} which will be the key to
Prop.\ \ref{indec_sum} and Lemma \ref{remove_edge}:

\begin{corol}\label{cor:genposk}
Let $\cC$ be a Cartan scheme such that $\rsC\re(\cC)$ is a finite root system.
Let $a\in A$, $k\in \{1,\dots,|I|\}$, and let $\beta _1,\dots,\beta _k\in R^a_+$
be linearly independent elements.
Then $R^a_+\cap \sum_{i=1}^k \RR\beta_i\subseteq \sum_{i=1}^k \NN_0\beta_i$
if and only if there exist $b\in A$, $w\in \Hom(a,b)$, and a permutation
$\tau $ of $I$ such that $w(\beta _i)=\al _{\tau (i)}$ for all $i\in \{1,\dots,k\}$.
\end{corol}

\begin{propo}\label{indec_sum}
Let $\cC$ be a Cartan scheme such that $\rsC\re(\cC)$ is a finite root system.
Let $a\in A$. Then the following are equivalent.
\begin{enumerate}
\item\label{insu_1} $\sum_{i\in I} \alpha_i\in \rer a$,
\item\label{insu_2} $\Gamma^a$ is connected.
\end{enumerate}
\end{propo}
\begin{proof}
We proceed by induction on the rank $r$.
The claim is true by \cite[Prop.\ 3.7]{p-CH09d} for $r=2$ and by
\cite[Lemma 3.12(2)]{p-CH09c} for $r=3$
(alternatively one can verify this by inspecting the data in \cite{p-CH09c}).

Let $r>3$. The implication (\ref{insu_1}) $\Rightarrow$ (\ref{insu_2}) is \cite[Prop.\ 4.6]{p-CH09a}.
Hence we have to show that (\ref{insu_2}) implies (\ref{insu_1}).
Without loss of generality, for each $i>1$ there exists $j<i$ such that $\alpha_i+\alpha_j$
is a root in $\rer a$. In particular $\alpha_1+\alpha_2\in \rer a$.
Let
\[ \beta_1=\alpha_1+\alpha_2,\quad\beta_2=\alpha_3,\:\beta_3=\alpha_4,
\ldots,\:\beta_{r-1}=\alpha_r.\]
By Cor.\ \ref{cor:genposk} there exist $b\in A$, $w\in\Hom(a,b)$ such that
$w(\beta_i)\in\{\alpha_1,\ldots,\alpha_r\}$ for all $i=1,\ldots,r-1$.
By the above assumption, if $\alpha_1+\alpha_i\in \rer a$ or $\alpha_2+\alpha_i\in \rer a$
for some $i>2$
then $\{\alpha_1$, $\alpha_2$, $\alpha_i\}$ is the base of an irreducible root set
and by induction $\alpha_1+\alpha_2+\alpha_i\in \rer a$.
Thus for each $i>1$ there exists $j<i$ such that $\beta_i+\beta_j\in \rer a$.
By induction, $\sum_{i=1}^{r-1}w(\beta_i)\in \rer b$ and hence
$\sum_{i=1}^r \alpha_i=\sum_{i=1}^{r-1}\beta_i\in \rer a$.
\end{proof}

In the sequel let $\cC$ be an irreducible \coscorf of rank $r>8$, $a$ be an object of $\cC$
and $\Gamma$ the Dynkin diagram of $C^a$.

\begin{lemma}\label{not_E8}
The diagram $\Gamma$ does not contain a diagram of type $E_8$.
\end{lemma}
\begin{proof}
Assume that $\Gamma$ contains a subdiagram of type $E_8$.
Then by the list of diagrams with $8$ vertices in Figure \ref{fig_dynkin},
$\Gamma$ contains a subdiagram of affine type $\tilde E_8$,
otherwise one gets a forbidden subdiagram of rank $8$. So assume without loss of generality
that $r=9$ and that the diagram at $a$ is $\tilde E_8$.

Label the vertices of the diagram $E_8$ by $1,\ldots,8$ and the new vertex
by $9$ (as in Fig.\ \ref{fig_E9}).
\begin{figure}[h]
\setlength{\unitlength}{3947sp}
\begingroup\makeatletter\ifx\SetFigFont\undefined
\gdef\SetFigFont#1#2#3#4#5{%
  \reset@font\fontsize{#1}{#2pt}%
  \fontfamily{#3}\fontseries{#4}\fontshape{#5}%
  \selectfont}%
\fi\endgroup
\begin{picture}(4348,628)(1111,-305)
\thicklines
{\put(1801,-211){\line( 1, 0){600}}
\put(2401,-211){\line( 1, 0){600}}}
{\put(3001,-211){\line( 1, 0){600}}}
{\put(1201,-211){\line( 1, 0){600}}}
{\put(3601,-211){\line( 1, 0){600}}}
{\put(4201,-211){\line( 1, 0){600}}}
{\put(2401,239){\line( 0,-1){450}}}
{\multiput(4801,-211)(120.0,0.0){6}{\makebox(1.6667,11.6667){\SetFigFont{5}{6}{\rmdefault}{\mddefault}{\updefault}$\bullet$}}}
\put(1726,-436){\makebox(0,0)[lb]{\smash{{\SetFigFont{12}{14.4}{\rmdefault}{\mddefault}{\updefault}{$2$}}}}}
\put(2326,-436){\makebox(0,0)[lb]{\smash{{\SetFigFont{12}{14.4}{\rmdefault}{\mddefault}{\updefault}{$3$}}}}}
\put(2926,-436){\makebox(0,0)[lb]{\smash{{\SetFigFont{12}{14.4}{\rmdefault}{\mddefault}{\updefault}{$4$}}}}}
\put(3526,-436){\makebox(0,0)[lb]{\smash{{\SetFigFont{12}{14.4}{\rmdefault}{\mddefault}{\updefault}{$5$}}}}}
\put(1126,-436){\makebox(0,0)[lb]{\smash{{\SetFigFont{12}{14.4}{\rmdefault}{\mddefault}{\updefault}{$1$}}}}}
\put(2476,164){\makebox(0,0)[lb]{\smash{{\SetFigFont{12}{14.4}{\rmdefault}{\mddefault}{\updefault}{$8$}}}}}
\put(4126,-436){\makebox(0,0)[lb]{\smash{{\SetFigFont{12}{14.4}{\rmdefault}{\mddefault}{\updefault}{$6$}}}}}
\put(4726,-436){\makebox(0,0)[lb]{\smash{{\SetFigFont{12}{14.4}{\rmdefault}{\mddefault}{\updefault}{$7$}}}}}
\put(5326,-436){\makebox(0,0)[lb]{\smash{{\SetFigFont{12}{14.4}{\rmdefault}{\mddefault}{\updefault}{$9$}}}}}
{\put(1801,-211){\circle*{100}}}
{\put(2401,-211){\circle*{100}}}
{\put(3001,-211){\circle*{100}}}
{\put(3601,-211){\circle*{100}}}
{\put(4801,-211){\circle*{100}}}
{\put(1201,-211){\circle*{100}}}
{\put(2401,239){\circle*{100}}}
{\put(4201,-211){\circle*{100}}}
{\put(5401,-211){\circle*{100}}}
\end{picture}
\caption{Dynkin diagram of type $\tilde E_8$}
\label{fig_E9}
\end{figure}
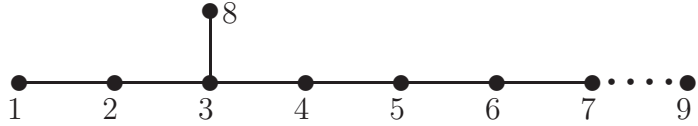
Since the \coscorf with diagram $E_8$ is standard,
$\sigma_1,\ldots,\sigma_8$ map to objects $b_1,\ldots,b_8$ where the subdiagram to the
labels $1,\ldots,8$ is of type $E_8$ as well. But by the same argument as
above, the $\cD^{b_i}$, $i=1,\ldots,8$ are all of type $\tilde E_8$.

Now consider the map $\sigma_9$; let $b_9 = \rho_9(a)$.
Since the vertices $1,2,\ldots,6$, $8$ are not connected
with the vertex $9$, they are not connected in $\Gamma^{b_9}$ as well.
Thus \cite[Lemma 4.5]{p-CH09a} implies that
the connections between the vertices $1,2,\ldots,6,8$ in $\Gamma^{b_9}$
are the same as in $\Gamma $ and that vertices $6$ and $7$ are connected.
By the reason given at the beginning of the proof it follows that
$\Gamma^{b_9}$ is of Dynkin type $\tilde{E}_8$.

Altogether, the Cartan scheme is standard of Dynkin type $\tilde E_8$, thus it is not a \coscorf
by the classification of finite crystallographic Coxeter groups and by \cite[Thm.\ 3.3]{p-CH09a}.
\end{proof}

\begin{defin}
Let $\Gamma$ be a Dynkin diagram and assume that it has vertices $i,j$ such that:
\begin{enumerate}
\item $i$ and $j$ are connected by an edge,
\item there is no vertex $k\notin\{i,j\}$ such that $(i,k)$ and $(j,k)$ are edges.
\end{enumerate}
Let $\Gamma'$ be the diagram obtained from $\Gamma$ by removing the edge $(i,j)$
and identifying the vertices $i$ and $j$ to a new vertex $\ell$, i.e.\ the edges of
$\Gamma'$ are
$\{ (k,m) \mid k,m \notin \{i,j\},\:\: (k,m) \mbox{ edge in }\Gamma \}$$\cup$$\{ (k,\ell) \mid k\notin \{i,j\},\:\: (k,i) \mbox{ or } (k,j) \mbox{ edge in }\Gamma \}$.

Then we call $\Gamma'$ the {\it contraction of} $\Gamma$ {\it along} $(i,j)$.
\end{defin}

The following lemma is a useful tool for the classification:
\begin{lemma}\label{remove_edge}
Let $\cC$ be an irreducible \coscorf of rank $r>8$ and assume that there are
pairwise different $i_1,\ldots,i_7\in I$
such that in $\Gamma^a$, $(i_\nu,i_\mu)$ for $\nu<\mu$ is connected if and only if
$\mu-\nu=1$, and such that $(i_\nu,i_{\nu+1})$ are edges (with labels $1$)
for all $\nu=1,\ldots,6$.

Then there exists an irreducible \coscorf
$\cC '=\cC '(I',A',$ $(\rfl '_i)_{i\in I'},({\Cm '}^a)_{a\in A'})$
and an object $a'\in A'$ such that
\begin{enumerate}
\item $I'=\{\ell\} \cup I\backslash\{i_3,i_4\}$,
\item $\Gamma'^{a'}$ is the contraction of $\Gamma^a$ along $(i_3,i_4)$,
\item\label{lem_teil2} for all $k\notin\{i_2,i_3,i_4,i_5\}$, $\Gamma'^{\rho'_k(a')}$ is
the contraction of $\Gamma^{\rho_k(a)}$ along $(i_3,i_4)$.
\end{enumerate}
\end{lemma}
\begin{proof}
Notice first that by Lemma \ref{not_E8} there is
an edge from $k$ to $i_3$ if and only if $k\in\{i_2,i_4\}$ and
that by Fig.\ \ref{fig_dynkin} there is
an edge from $k$ to $i_4$ if and only if $k\in\{i_3,i_5\}$.

By Cor.\ \ref{cor:genposk},
$\{\alpha_{i_3}+\alpha_{i_4},\alpha_j \mid j\in I, i_3\ne j\ne i_4\}$
is a base of a finite root set $\Lambda$ of rank $r-1$.
Let $\cC '=\cC '(I',A',$ $(\rfl '_i)_{i\in I'},({\Cm '}^a)_{a\in A'})$ be a
Cartan scheme, $\iota : \ZZ^{r-1}\to \ZZ^r$ a linear map and
$a'$ be an object of $\cC'$ such that $\Lambda = \iota(\rer {a'})$.
Remark that $\Wg(\cC')$ is a parabolic subgroupoid of $\Wg(\cC)$
(see \cite[Def.\ 2.3]{p-HW-10} for the precise definition of a parabolic subgroupoid of W(C)).
For the vertices of $\Gamma'^{a'}$ we use the same labels as for $\Gamma^a$;
the new vertex $\iota^{-1}(\alpha_{i_3}+\alpha_{i_4})$ is labeled $\ell$.

We prove that the Dynkin diagram $\Gamma'^{a'}$ is the contraction of $\Gamma^a$ along $(i_3,i_4)$.
The subdiagram to
$i_1,i_2,\ell,i_5,i_6,i_7$
is of type $A_6$:
Let $k\notin\{i_2,\ldots,i_5\}$ and assume that there is a connection
from $k$ to $\ell$ in $\Gamma'^{a'}$.
Then $\alpha_k+\alpha_{i_3}+\alpha_{i_4}$ is a root in $R^a$.
But by Prop.\ \ref{indec_sum}, either $\alpha_k+\alpha_{i_3}$ or
$\alpha_k+\alpha_{i_4}$ is a root, contradicting the fact that there is no edge
from $k$ to $i_3$ or $i_4$ in $\Gamma$.
Thus there is no connection from $k$ to $\ell$ in $\Gamma'^{a'}$.
Moreover, the edge $(i_5,\ell)$ is labeled by a one by Fig.\ \ref{fig_dynkin} and
the edge $(i_2,\ell)$ is labeled by a one because $r>8$ and the diagrams of type
$\Gamma_6^1$ and $\Gamma_6^2$ are not part of an irreducible Dynkin diagram with
$7$ vertices.
Of course, connections not involving $\ell$ are the same as in $\Gamma^a$
by Lemma \ref{co:cij}.

For $k\notin\{i_2,\ldots,i_5\}$ the diagram
\[ \begin{CD}
R'^{a'} @>{\iota}>> R^{a} \\
@V {\sigma_k} VV @V{\sigma_k} VV \\
R'^{\rho'_k(a')} @>{\iota}>> R^{\rho_k(a)}
\end{CD} \]
commutes because $\iota$ maps simple roots $\alpha_m$ with $m\notin\{i_3,i_4\}$
to simple roots and because there is no edge from $k$ to $i_3$ or $i_4$.
By the same argument as above we obtain (\ref{lem_teil2}).
\end{proof}

The following theorem classifies the possible Dynkin diagrams.

\begin{theor}\label{th:diagrams}
Let $\Gamma$ be the Dynkin diagram of an object $a$ in a \coscorf $\cC$ of rank $r>8$.
Then $\Gamma$ is of type $A$, $B$, $C$, $D$ or $D'$.
\end{theor}
\begin{proof}
We proceed by induction on $r$.
By Section \ref{sec:rank_le_eight}, for $r=8$ the diagrams are of type
$A_8,B_8,C_8,D_8$ or $E_8$.
Now let $r>8$.
By Lemma \ref{not_E8} and induction, each connected subdiagram of $\Gamma$ of rank $r-1$
is of type $A_{r-1}$, $B_{r-1}$, $C_{r-1}$, $D_{r-1}$ or $D'_{r-1}$.

If $\Gamma$ has a subdiagram of type $A_{r-1}$, then using
induction and Lemma \ref{not_E8}, one checks that $\Gamma$ is of type
$A_r$, $B_r$, $C_r$, $D_r$, $D'_r$ or $\tilde A_r$.
If $\Gamma$ is of type $\tilde A_r$ then by Lemma \ref{remove_edge},
removing an edge in the middle yields an irreducible root set of rank $r-1$
with a Dynkin diagram of type $\tilde A_{r-1}$ which is forbidden.

Similarly, if $\Gamma$ has a subdiagram of type $B_{r-1}$ or $C_{r-1}$, then $\Gamma$ is of type
$B_r$ resp.\ $C_r$ (notice that $r-1>7$).

If $\Gamma$ has a subdiagram of type $D_{r-1}$ or $D'_{r-1}$,
then $\Gamma$ is of type $D_r$, $D'_r$ or we are in one of two cases:\\
1.\ The diagram $\Gamma$ has the connections of a diagram of type $\tilde D_r$
(the affine diagram of type $D$) and possibly some more connections.
Choose the labels as in Fig.\ \ref{fig_case_1}.
\begin{figure}[h]
\setlength{\unitlength}{3547sp}
\begingroup\makeatletter\ifx\SetFigFont\undefined%
\gdef\SetFigFont#1#2#3#4#5{%
  \reset@font\fontsize{#1}{#2pt}%
  \fontfamily{#3}\fontseries{#4}\fontshape{#5}%
  \selectfont}%
\fi\endgroup%
\begin{picture}(4980,300)(1200,-300)
\thicklines
{\put(1801,-211){\line( 1, 0){600}}\put(2401,-211){\line( 1, 0){600}}}
{\put(1801,-211){\line(-5, 2){594.828}}}{\put(1201,-436){\line( 5, 2){594.828}}}
{\put(4801,-211){\line( 1, 0){600}}}{\put(5401,-211){\line( 5, 2){594.828}}}
{\put(5401,-211){\line( 5,-2){594.828}}}{\put(3001,-211){\line( 1, 0){600}}}
{\multiput(3601,-211)(200,0.00000){6}{\makebox(1.6667,11.6667){\SetFigFont{5}{6}{\rmdefault}{\mddefault}{\updefault}$\bullet$}}}
{\multiput(1201, 14)(0.0,-90.0){6}{\makebox(1.6667,11.6667){\SetFigFont{5}{6}{\rmdefault}{\mddefault}{\updefault}$\bullet$}}}
{\multiput(6001, 14)(0.0,-90.0){6}{\makebox(1.6667,11.6667){\SetFigFont{5}{6}{\rmdefault}{\mddefault}{\updefault}$\bullet$}}}
\put(580,0){\makebox(0,0)[lb]{\smash{{\SetFigFont{12}{14.4}{\rmdefault}{\mddefault}{\updefault}{$r-1$}}}}}
\put(980,-550){\makebox(0,0)[lb]{\smash{{\SetFigFont{12}{14.4}{\rmdefault}{\mddefault}{\updefault}{$1$}}}}}
\put(1726,-436){\makebox(0,0)[lb]{\smash{{\SetFigFont{12}{14.4}{\rmdefault}{\mddefault}{\updefault}{$2$}}}}}
\put(6126,0){\makebox(0,0)[lb]{\smash{{\SetFigFont{12}{14.4}{\rmdefault}{\mddefault}{\updefault}{$r$}}}}}
\put(6126,-511){\makebox(0,0)[lb]{\smash{{\SetFigFont{12}{14.4}{\rmdefault}{\mddefault}{\updefault}{$r-2$}}}}}
\put(2326,-436){\makebox(0,0)[lb]{\smash{{\SetFigFont{12}{14.4}{\rmdefault}{\mddefault}{\updefault}{$3$}}}}}
\put(2926,-436){\makebox(0,0)[lb]{\smash{{\SetFigFont{12}{14.4}{\rmdefault}{\mddefault}{\updefault}{$4$}}}}}
\put(3526,-436){\makebox(0,0)[lb]{\smash{{\SetFigFont{12}{14.4}{\rmdefault}{\mddefault}{\updefault}{$5$}}}}}
\put(5151,-436){\makebox(0,0)[lb]{\smash{{\SetFigFont{12}{14.4}{\rmdefault}{\mddefault}{\updefault}{$r-3$}}}}}
{\put(1801,-211){\circle*{100}}}{\put(2401,-211){\circle*{100}}}
{\put(3001,-211){\circle*{100}}}{\put(1201, 14){\circle*{100}}}
{\put(1201,-436){\circle*{100}}}{\put(3601,-211){\circle*{100}}}
{\put(5401,-211){\circle*{100}}}{\put(4801,-211){\circle*{100}}}
{\put(6001, 14){\circle*{100}}}{\put(6001,-436){\circle*{100}}}
\end{picture}
\caption{Case 1.}
\label{fig_case_1}
\end{figure}
Identifying the vertices $3$ and $4$ does not give a Dynkin diagram of a \coscorfn,
thus by Lemma \ref{remove_edge} this case is impossible (again, notice that $r-1>7$).\\
2.\ The subdiagrams to the labels $(r-1,2,\ldots,r-2,r)$ and $(1,\ldots,r-2,r)$ are both
of type $B$ or $C$. But then by Lemma \ref{remove_edge}, removing an edge in the middle
yields an irreducible root set of rank $r-1$ with a forbidden Dynkin
diagram.
\end{proof}

\begin{lemma}\label{lem5}
Let $\Gamma$ be the Dynkin diagram of an object $a$ in a \coscorf $\cC$
and $i\in I$.
\begin{enumerate}
\item $\{i,i_1,\ldots,i_k\}\subseteq I$ are connected in $\Gamma^a$
if and only if $\{i,i_1,\ldots,i_k\}$ are connected in $\Gamma^{\rho_i(a)}$.
\end{enumerate}
Let $j,k\in I$ with $|\{i,j,k\}|=3$.
\begin{enumerate} \setcounter{enumi}{1}
\item If $i$ is not connected to $j$ nor to $k$ then the connection between $j$ and $k$
(including labels) is the same in $\Gamma^a$ and $\Gamma^{\rho_i(a)}$.
\item If $i$ is connected to $j$ and $i$ is not connected to $k$ then $j$ is connected to $k$
in $\Gamma^a$ if and only if they are connected in $\Gamma^{\rho_i(a)}$.
\end{enumerate}
\end{lemma}
\begin{proof}
Use \cite[Lemma 4.5]{p-CH09a}, axiom (C2), \cite[Prop.\ 4.6]{p-CH09a}.
\end{proof}

Now Section \ref{sec:rank_le_eight} allows us to give more details about
the Cartan schemes.

\begin{figure}[h]
\setlength{\unitlength}{3300sp}
\begingroup\makeatletter\ifx\SetFigFont\undefined%
\gdef\SetFigFont#1#2#3#4#5{%
  \reset@font\fontsize{#1}{#2pt}%
  \fontfamily{#3}\fontseries{#4}\fontshape{#5}%
  \selectfont}%
\fi\endgroup%
\begin{picture}(5580,5178)(400,-4880)
\thicklines
{\put(1201,-661){\oval(600,900)[tl]}\put(1201,-661){\oval(600,900)[bl]}}
{\put(1201,-1561){\oval(600,900)[tl]}\put(1201,-1561){\oval(600,900)[bl]}}
{\put(1201,-2461){\oval(600,900)[tl]}\put(1201,-2461){\oval(600,900)[bl]}}
{\put(1201,-2911){\oval(1200,1800)[tl]}\put(1201,-2911){\oval(1200,1800)[bl]}}
{\put(1201,-4261){\oval(600,900)[tl]}\put(1201,-4261){\oval(600,900)[bl]}}
{\put(1801,-211){\line( 1, 0){600}}\put(2401,-211){\line( 1, 0){600}}}
{\put(4201,-211){\line( 1, 0){600}}}{\put(4801,-211){\line( 5, 2){594.828}}}
{\multiput(3001,-211)(200,0.00000){6}{\makebox(1.6667,11.6667){\SetFigFont{5}{6}{\rmdefault}{\mddefault}{\updefault}$\bullet$}}}
{\put(5401,-436){\vector( 0, 1){450}}}{\put(1801,-1111){\line( 1, 0){600}}
\put(2401,-1111){\line( 1, 0){600}}}{\put(4201,-1111){\line( 1, 0){600}}}
{\put(4801,-1111){\line( 5, 2){594.828}}}
{\multiput(3001,-1111)(200,0.00000){6}{\makebox(1.6667,11.6667){\SetFigFont{5}{6}{\rmdefault}{\mddefault}{\updefault}$\bullet$}}}
{\put(5401,-886){\line( 0,-1){450}}}{\put(1801,-2011){\line( 1, 0){600}}
\put(2401,-2011){\line( 1, 0){600}}}
{\multiput(3001,-2011)(200,0.00000){6}{\makebox(1.6667,11.6667){\SetFigFont{5}{6}{\rmdefault}{\mddefault}{\updefault}$\bullet$}}}
{\put(4201,-2011){\line( 1, 0){600}}}{\put(4801,-2011){\line( 5, 2){594.828}}}
{\put(4801,-2011){\line( 5,-2){594.828}}}{\put(5401,-1786){\line( 0,-1){450}}}
{\put(1801,-2911){\line( 1, 0){600}}\put(2401,-2911){\line( 1, 0){600}}}
{\put(4201,-2911){\line( 1, 0){600}}}{\put(4801,-2911){\line( 5, 2){594.828}}}
{\put(4801,-2911){\line( 5,-2){594.828}}}
{\multiput(3001,-2911)(200,0.00000){6}{\makebox(1.6667,11.6667){\SetFigFont{5}{6}{\rmdefault}{\mddefault}{\updefault}$\bullet$}}}
{\put(1801,-3811){\line( 1, 0){600}}\put(2401,-3811){\line( 1, 0){600}}}
{\multiput(3001,-3811)(200,0.00000){6}{\makebox(1.6667,11.6667){\SetFigFont{5}{6}{\rmdefault}{\mddefault}{\updefault}$\bullet$}}}
{\put(4201,-3811){\line( 1, 0){600}}}{\put(4801,-3811){\line( 5,-2){594.828}}}
{\put(5401,-3586){\line( 0,-1){450}}}{\put(1801,-4711){\line( 1, 0){600}}
\put(2401,-4711){\line( 1, 0){600}}}
{\multiput(3001,-4711)(200,0.00000){6}{\makebox(1.6667,11.6667){\SetFigFont{5}{6}{\rmdefault}{\mddefault}{\updefault}$\bullet$}}}
{\put(4201,-4711){\line( 1, 0){600}}}{\put(4801,-4711){\line( 5,-2){594.828}}}
{\put(5401,-4486){\vector( 0,-1){450}}}
\put(1351,-286){\makebox(0,0)[lb]{\smash{{\SetFigFont{12}{14.4}{\rmdefault}{\mddefault}{\updefault}{$C$}}}}}
\put(4651,-61){\makebox(0,0)[lb]{\smash{{\SetFigFont{12}{14.4}{\rmdefault}{\mddefault}{\updefault}{$r-2$}}}}}
\put(5476,-61){\makebox(0,0)[lb]{\smash{{\SetFigFont{12}{14.4}{\rmdefault}{\mddefault}{\updefault}{$r-1$}}}}}
\put(1726,-61){\makebox(0,0)[lb]{\smash{{\SetFigFont{12}{14.4}{\rmdefault}{\mddefault}{\updefault}{$1$}}}}}
\put(2326,-61){\makebox(0,0)[lb]{\smash{{\SetFigFont{12}{14.4}{\rmdefault}{\mddefault}{\updefault}{$2$}}}}}
\put(5476,-511){\makebox(0,0)[lb]{\smash{{\SetFigFont{12}{14.4}{\rmdefault}{\mddefault}{\updefault}{$r$}}}}}
\put(5476,-286){\makebox(0,0)[lb]{\smash{{\SetFigFont{12}{14.4}{\rmdefault}{\mddefault}{\updefault}{$2$}}}}}
\put(1351,-1186){\makebox(0,0)[lb]{\smash{{\SetFigFont{12}{14.4}{\rmdefault}{\mddefault}{\updefault}{$A$}}}}}
\put(4651,-961){\makebox(0,0)[lb]{\smash{{\SetFigFont{12}{14.4}{\rmdefault}{\mddefault}{\updefault}{$r-2$}}}}}
\put(5476,-961){\makebox(0,0)[lb]{\smash{{\SetFigFont{12}{14.4}{\rmdefault}{\mddefault}{\updefault}{$r-1$}}}}}
\put(5476,-1411){\makebox(0,0)[lb]{\smash{{\SetFigFont{12}{14.4}{\rmdefault}{\mddefault}{\updefault}{$r$}}}}}
\put(1726,-961){\makebox(0,0)[lb]{\smash{{\SetFigFont{12}{14.4}{\rmdefault}{\mddefault}{\updefault}{$1$}}}}}
\put(2326,-961){\makebox(0,0)[lb]{\smash{{\SetFigFont{12}{14.4}{\rmdefault}{\mddefault}{\updefault}{$2$}}}}}
\put(1351,-2086){\makebox(0,0)[lb]{\smash{{\SetFigFont{12}{14.4}{\rmdefault}{\mddefault}{\updefault}{$D'$}}}}}
\put(4651,-1861){\makebox(0,0)[lb]{\smash{{\SetFigFont{12}{14.4}{\rmdefault}{\mddefault}{\updefault}{$r-2$}}}}}
\put(5476,-1861){\makebox(0,0)[lb]{\smash{{\SetFigFont{12}{14.4}{\rmdefault}{\mddefault}{\updefault}{$r-1$}}}}}
\put(5476,-2311){\makebox(0,0)[lb]{\smash{{\SetFigFont{12}{14.4}{\rmdefault}{\mddefault}{\updefault}{$r$}}}}}
\put(1726,-1861){\makebox(0,0)[lb]{\smash{{\SetFigFont{12}{14.4}{\rmdefault}{\mddefault}{\updefault}{$1$}}}}}
\put(2326,-1861){\makebox(0,0)[lb]{\smash{{\SetFigFont{12}{14.4}{\rmdefault}{\mddefault}{\updefault}{$2$}}}}}
\put(1351,-2986){\makebox(0,0)[lb]{\smash{{\SetFigFont{12}{14.4}{\rmdefault}{\mddefault}{\updefault}{$D$}}}}}
\put(4651,-2761){\makebox(0,0)[lb]{\smash{{\SetFigFont{12}{14.4}{\rmdefault}{\mddefault}{\updefault}{$r-2$}}}}}
\put(5476,-2761){\makebox(0,0)[lb]{\smash{{\SetFigFont{12}{14.4}{\rmdefault}{\mddefault}{\updefault}{$r-1$}}}}}
\put(5476,-3211){\makebox(0,0)[lb]{\smash{{\SetFigFont{12}{14.4}{\rmdefault}{\mddefault}{\updefault}{$r$}}}}}
\put(1726,-2761){\makebox(0,0)[lb]{\smash{{\SetFigFont{12}{14.4}{\rmdefault}{\mddefault}{\updefault}{$1$}}}}}
\put(2326,-2761){\makebox(0,0)[lb]{\smash{{\SetFigFont{12}{14.4}{\rmdefault}{\mddefault}{\updefault}{$2$}}}}}
\put(1351,-3886){\makebox(0,0)[lb]{\smash{{\SetFigFont{12}{14.4}{\rmdefault}{\mddefault}{\updefault}{$A$}}}}}
\put(4651,-3661){\makebox(0,0)[lb]{\smash{{\SetFigFont{12}{14.4}{\rmdefault}{\mddefault}{\updefault}{$r-2$}}}}}
\put(5476,-4111){\makebox(0,0)[lb]{\smash{{\SetFigFont{12}{14.4}{\rmdefault}{\mddefault}{\updefault}{$r$}}}}}
\put(5476,-3661){\makebox(0,0)[lb]{\smash{{\SetFigFont{12}{14.4}{\rmdefault}{\mddefault}{\updefault}{$r-1$}}}}}
\put(1726,-3661){\makebox(0,0)[lb]{\smash{{\SetFigFont{12}{14.4}{\rmdefault}{\mddefault}{\updefault}{$1$}}}}}
\put(2326,-3661){\makebox(0,0)[lb]{\smash{{\SetFigFont{12}{14.4}{\rmdefault}{\mddefault}{\updefault}{$2$}}}}}
\put(1351,-4786){\makebox(0,0)[lb]{\smash{{\SetFigFont{12}{14.4}{\rmdefault}{\mddefault}{\updefault}{$C$}}}}}
\put(4651,-4561){\makebox(0,0)[lb]{\smash{{\SetFigFont{12}{14.4}{\rmdefault}{\mddefault}{\updefault}{$r-2$}}}}}
\put(5476,-4561){\makebox(0,0)[lb]{\smash{{\SetFigFont{12}{14.4}{\rmdefault}{\mddefault}{\updefault}{$r-1$}}}}}
\put(5476,-5011){\makebox(0,0)[lb]{\smash{{\SetFigFont{12}{14.4}{\rmdefault}{\mddefault}{\updefault}{$r$}}}}}
\put(1726,-4561){\makebox(0,0)[lb]{\smash{{\SetFigFont{12}{14.4}{\rmdefault}{\mddefault}{\updefault}{$1$}}}}}
\put(2326,-4561){\makebox(0,0)[lb]{\smash{{\SetFigFont{12}{14.4}{\rmdefault}{\mddefault}{\updefault}{$2$}}}}}
\put(5476,-4786){\makebox(0,0)[lb]{\smash{{\SetFigFont{12}{14.4}{\rmdefault}{\mddefault}{\updefault}{$2$}}}}}
\put(376,-736){\makebox(0,0)[lb]{\smash{{\SetFigFont{12}{14.4}{\rmdefault}{\mddefault}{\updefault}{$\sigma_{r-2}$}}}}}
\put(976,-2536){\makebox(0,0)[lb]{\smash{{\SetFigFont{12}{14.4}{\rmdefault}{\mddefault}{\updefault}{$\sigma_{r-2}$}}}}}
\put(376,-4336){\makebox(0,0)[lb]{\smash{{\SetFigFont{12}{14.4}{\rmdefault}{\mddefault}{\updefault}{$\sigma_{r-2}$}}}}}
\put(376,-1636){\makebox(0,0)[lb]{\smash{{\SetFigFont{12}{14.4}{\rmdefault}{\mddefault}{\updefault}{$\sigma_{r-1}$}}}}}
\put(350,-2986){\makebox(0,0)[lb]{\smash{{\SetFigFont{12}{14.4}{\rmdefault}{\mddefault}{\updefault}{$\sigma_r$}}}}}
{\put(4801,-211){\circle*{100}}}{\put(1801,-211){\circle*{100}}}
{\put(2401,-211){\circle*{100}}}{\put(4201,-211){\circle*{100}}}
{\put(5401, 14){\circle*{100}}}{\put(5401,-436){\circle*{100}}}
{\put(3001,-211){\circle*{100}}}{\put(1801,-1111){\circle*{100}}}
{\put(2401,-1111){\circle*{100}}}{\put(3001,-1111){\circle*{100}}}
{\put(4201,-1111){\circle*{100}}}{\put(4801,-1111){\circle*{100}}}
{\put(5401,-886){\circle*{100}}}{\put(5401,-1336){\circle*{100}}}
{\put(1801,-2011){\circle*{100}}}{\put(2401,-2011){\circle*{100}}}
{\put(3001,-2011){\circle*{100}}}{\put(4201,-2011){\circle*{100}}}
{\put(4801,-2011){\circle*{100}}}{\put(5401,-1786){\circle*{100}}}
{\put(5401,-2236){\circle*{100}}}{\put(4801,-2911){\circle*{100}}}
{\put(1801,-2911){\circle*{100}}}{\put(2401,-2911){\circle*{100}}}
{\put(3001,-2911){\circle*{100}}}{\put(4201,-2911){\circle*{100}}}
{\put(5401,-2686){\circle*{100}}}{\put(5401,-3136){\circle*{100}}}
{\put(1801,-3811){\circle*{100}}}{\put(2401,-3811){\circle*{100}}}
{\put(3001,-3811){\circle*{100}}}{\put(4201,-3811){\circle*{100}}}
{\put(4801,-3811){\circle*{100}}}{\put(5401,-3586){\circle*{100}}}
{\put(5401,-4036){\circle*{100}}}{\put(1801,-4711){\circle*{100}}}
{\put(2401,-4711){\circle*{100}}}{\put(3001,-4711){\circle*{100}}}
{\put(4201,-4711){\circle*{100}}}{\put(4801,-4711){\circle*{100}}}
{\put(5401,-4936){\circle*{100}}}{\put(5401,-4486){\circle*{100}}}
\end{picture}
\caption{Dynkin diagrams for the series}
\label{fig_DD}
\end{figure}
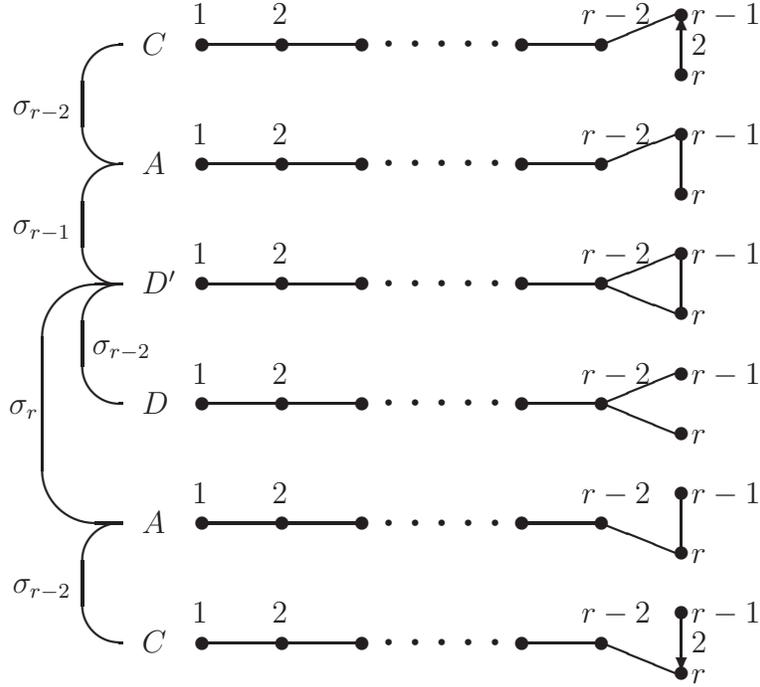

\begin{propo}\label{prop:diagrams}
Let $\Gamma$ be the Dynkin diagram of an object $a$ in a \coscorf $\cC$ of rank $r>8$.
\begin{enumerate}
\item\label{ct_1} If $\Gamma$ is of type $B$ then $\cC$ is a standard Cartan scheme.
\item\label{ct_2} If $\Gamma$ is of type $A$, $C$ or $D$, then either $\cC$ is a standard
Cartan scheme, or there is an object of Dynkin type $D'$.
\item\label{ct_3} Assume that there is an object of Dynkin type $D'$ in $\cC$.\\
If $\Gamma$ is of type $D'$ with labels as in Fig.\ \ref{fig_DD},
then the diagrams that appear in $\cC$ are the diagrams of Fig.\ \ref{fig_DD} with the same labels,
possibly without the diagrams of type $C$ or $D$.\\
If $\cC$ has an object $a$ with diagram of type $D$ resp.\ $C$ and if there is a $j\in I$ such
that $\rho_j(a)$ is not of Dynkin type $D$ resp.\ $C$, then $j=r-2$ and $\rho_j(a)$ is of Dynkin type $D'$ resp.\ $A$.\\
The simple reflections $\sigma_{r-1}$ and $\sigma_r$ always map an object of Dynkin type
$D'$ to an object of Dynkin type $A$ and vice versa (as in Fig.\ \ref{fig_DD}).
\end{enumerate}
\end{propo}
\begin{proof}
We proceed by induction on $r$ and prove (\ref{ct_1})-(\ref{ct_3}) simultaneously.
For $r=8$ and $\Gamma$ not of type $E_8$ all the above
claims hold by inspecting the resulting data of Section \ref{sec:rank_le_eight}.
Now let $r>8$.

If $\Gamma$ is of type $B$ then by induction hypothesis, Lemma \ref{remove_edge} and
Thm.\ \ref{th:diagrams}, the maps $\sigma_1,\ldots,\sigma_{r}$ map to objects of Dynkin type $B_r$,
thus $\cC$ is standard.

Assume that $\Gamma$ is of type $A$, $C$ or $D$ and that $\cC$ is not standard.
Then there is an object in $\cC$ with diagram $\Gamma$ and $j\in I$ such that applying
$\sigma_j$ leads to an object of different Dynkin type.
Choose the labels as in Fig.\ \ref{fig_DD}. Then by Lemma \ref{remove_edge},
removing the edge $(4,5)$ yields a diagram $\Gamma'$ of the same type belonging
to a Cartan scheme $\cC'$ of rank $r-1$. If $\cC'$ was standard, then the maps
$\sigma_1,\sigma_2,\sigma_7,\ldots,\sigma_r$ would preserve the diagram $\Gamma$;
but since this is also the case for $\sigma_3,\ldots,\sigma_6$ by Lemma \ref{lem5},
this would contradict the assumption that $\sigma_j$ maps to a different diagram.
Hence $\cC'$ is not standard.
Now if $\Gamma'$ is of type $A$, then by induction either $\sigma_{r-1}$ or $\sigma_{r}$
maps (in $\cC'$) to a diagram of type $D'$. But these maps are not affected by the deletion of
$(4,5)$, so $\sigma_{r-1}$ or $\sigma_r$ map $\Gamma$ to a diagram of type $D'$ in $\cC$.
If $\Gamma$ is of type $D$ then an easy calculation shows that
$j\in\{r-2,r-1,r\}$. But then using $\cC'$ we get that $j=r-2$ and that $\sigma_j$
maps to a diagram of type $D'$.

If $\Gamma$ is of type $C$, then by the same argument as for type $D$ we get to
an object of Dynkin type $A$.
We just proved that in this case an object of Dynkin type $D'$ also occurs in $\cC$.
Thus we have proved (\ref{ct_2}):
If $\Gamma$ is of type $A$, $C$ or $D$ and $\cC$ is not standard,
then there exists an object $b$ of Dynkin type $D'$.
The morphisms needed to get from $a$ to $b$ are as explained in (\ref{ct_3})
by Lemma \ref{remove_edge}.
\end{proof}

\subsection{The root systems}

Let $r\in \NN$.
Recall that we denote $\{\alpha_1,\ldots,\alpha_r\}$ the standard basis of $\ZZ^r$.
We use the following notation: For $1\le i,j\le r$, let
\[ \frto_{i,j}:=\begin{cases}
\sum_{k=i}^j\alpha_k & i\le j \\ 0 & i>j
\end{cases}. \]

\begin{defin}
Let $Z\subseteq\{1,\ldots,r-1\}$. Let $\Phi_{r,Z}$ denote the set of roots
\begin{eqnarray*}
&\frto_{i,j-1}, & 1\le i<j\le r,\\
&\frto_{i,r-2}+\alpha_r, & 1\le i< r,\\
&\frto_{i,r}+\frto_{j,r-2}, & 1\le i<j<r,\\
&\frto_{j,r}+\frto_{j,r-2}, & j\in Z.
\end{eqnarray*}

Let $Y\subseteq\{1,\ldots,r-1\}$. Let $\Psi_{r,Y}$ denote the set of roots
\begin{eqnarray*}
&\frto_{i,j}, & 1\le i\le j\le r,\\
&\frto_{i,r}+\frto_{j,r-1}, & 1\le i<j<r,\\
&\frto_{j,r}+\frto_{j,r-1}, & j\in Y.
\end{eqnarray*}

Further, denote $\Psi'_{r,Y}$ the set obtained from $\Psi_{r,Y}$ by exchanging
$\alpha_{r-1}$ and $\alpha_r$.
\end{defin}
\begin{remar}
The sets $\Phi_{r,\emptyset}$ resp.\ $\Psi_{r,\{1,\ldots,r-1\}}$ are the sets of positive roots
of the Weyl groups of type $D_r$ resp.\ $C_r$, compare \cite[VI.\ 4.6, 4.8]{b-BourLie4-6}.
\end{remar}

Recall that by Def.\ \ref{cartan_maps} we write $C_\Lambda$ for the generalized Cartan matrix
given by a set $\Lambda$.

\begin{propo}\label{DD_scheme1}
Let $Y,Z \subseteq \{1,\ldots,r-1\}$.
\begin{enumerate}
\item The Dynkin diagram of $C_{\Phi_{r,Z}}$ is of type $D'_r$ if $r-1\in Z$ and
of type $D_r$ if $r-1\notin Z$.
\item The Dynkin diagram of $C_{\Psi_{r,Y}}$ is of type $C_r$ if $r-1\in Y$ and
of type $A_r$ if $r-1\notin Y$.
\end{enumerate}
\end{propo}
\begin{proof}
This is clear by definition.
\end{proof}
\begin{propo}\label{DD_scheme2}
Let $Y,Z \subseteq \{1,\ldots,r-1\}$ with $0\le |Y|<|Z| <r$.
Then
\begin{center}
\begin{longtable}{rcll}
$\sigma_i(\pm\Phi_{r,Z})$ &$=$& $\pm\Phi_{r,(i\:\: i+1)(Z)}$ & \quad for $i=1,\ldots,r-2$,\\
$\sigma_i(\pm\Psi_{r,Y})$ &$=$& $\pm\Psi_{r,(i\:\: i+1)(Y)}$ & \quad for $i=1,\ldots,r-2$,\\
$\sigma_i(\pm\Psi'_{r,Y})$ &$=$& $\pm\Psi'_{r,(i\:\: i+1)(Y)}$ & \quad for $i=1,\ldots,r-2$,\\
$\sigma_{r-1}(\pm\Phi_{r,Z})$ &$=$& $\pm\Phi_{r,Z}$ & \quad if $r-1\notin Z$,\\
$\sigma_{r}(\pm\Phi_{r,Z})$ &$=$& $\pm\Phi_{r,Z}$ & \quad if $r-1\notin Z$,\\
$\sigma_{r-1}(\pm\Phi_{r,Z})$ &$=$& $\pm\Psi_{r,Z\backslash\{r-1\}}$ & \quad if $r-1\in Z$,\\
$\sigma_{r}(\pm\Phi_{r,Z})$ &$=$& $\pm\Psi'_{r,Z\backslash\{r-1\}}$ & \quad if $r-1\in Z$,\\
$\sigma_{r-1}(\pm\Psi_{r,Y})$ &$=$& $\pm\Phi_{r,Y\cup\{r-1\}}$ & \quad if $r-1\notin Y$,\\
$\sigma_{r}(\pm\Psi_{r,Y})$ &$=$& $\pm\Psi_{r,Y}$ & \quad if $r-1\notin Y$,\\
$\sigma_{r-1}(\pm\Psi_{r,Y})$ &$=$& $\pm\Psi_{r,Y}$ & \quad if $r-1\in Y$,\\
$\sigma_{r}(\pm\Psi_{r,Y})$ &$=$& $\pm\Psi_{r,Y}$ & \quad if $r-1\in Y$,\\
$\sigma_{r}(\pm\Psi'_{r,Y})$ &$=$& $\pm\Phi_{r,Y\cup\{r-1\}}$ & \quad if $r-1\notin Y$,\\
$\sigma_{r-1}(\pm\Psi'_{r,Y})$ &$=$& $\pm\Psi'_{r,Y}$ & \quad if $r-1\notin Y$,\\
$\sigma_{r}(\pm\Psi'_{r,Y})$ &$=$& $\pm\Psi'_{r,Y}$ & \quad if $r-1\in Y$,\\
$\sigma_{r-1}(\pm\Psi'_{r,Y})$ &$=$& $\pm\Psi'_{r,Y}$ & \quad if $r-1\in Y$,\\
\end{longtable}
\end{center}
where in ``$\sigma_i(\Lambda)$'' the map $\sigma_i$ is the map given by $\Lambda$ as
in Def.\ \ref{cartan_maps}, $(i\:\: i+1)$ is the transposition and
$\pm\Lambda = \Lambda\cup-\Lambda$.
\end{propo}
\begin{proof}
Let $\beta_j:=\frto_{j,r}+\frto_{j,r-2}$.
Then one computes (at $\Phi_{r,Z}$)
\[ \sigma_i(\beta_j) = \begin{cases}
\beta_j & i\notin\{j-1,j\} \\
\beta_{j-1} & i=j-1 \\
\beta_{j+1} & i=j
\end{cases} \]
for all $i=1,\ldots,r-2$ and $j=1,\ldots,r-1$.
So $\sigma_1,\ldots,\sigma_{r-2}$ act as transpositions on
$\beta_1,\ldots,\beta_{r-1}$. The situation is similar for $\Psi$ and $\Psi'$.
The other claims are an easy (although tiring) calculation.
\end{proof}

Prop.\ \ref{DD_scheme1}, Prop.\ \ref{DD_scheme2}, and Def.\ \ref{cartan_maps} immediately give:

\begin{corol}\label{phi_coscorf}
Let $Z\subseteq\{1,\ldots,r-1\}$. Then there exists a \coscorf $\cC$ such that
$\rer{a}_+=\Phi_{r,Z}$ for an object $a$.
\end{corol}

\begin{remar}
The Dynkin diagrams of the \coscorf of Cor.\ \ref{phi_coscorf} and their connections are
given by Fig.\ \ref{fig_DD}.
The connections $\sigma_{r-2}$ in the figure depend on $Z$ resp.\ $Y$. For example $\sigma_{r-2}$
maps an object $\Phi_{r,Z}$ of Dynkin type $D$ to an object of Dynkin type $D'$ if and only if $r-2 \in Z$;
if $\Psi_{r,Y}$ is of Dynkin type $C$ (as in the first diagram of Fig.\ \ref{fig_DD}) then
$\sigma_{r-2}$ maps to an object of Dynkin type $A$ if and only if $r-2 \notin Z$.
\end{remar}

\begin{propo}\label{Z1_Z2}
Let $Z_1,Z_2,Y_1,Y_2\subseteq\{1,\ldots,r-1\}$ with $|Z_1|=|Z_2|=|Y_1|+1=|Y_2|+1$.
Then there exists a \coscorf $\cC$ with objects $a,b,c,d$ such that
$\rer{a}_+=\Phi_{r,Z_1}$, $\rer{b}_+=\Phi_{r,Z_2}$, $\rer{c}_+=\Psi_{r,Y_1}$ and $\rer{d}_+=\Psi'_{r,Y_2}$.
\end{propo}
\begin{proof}
By Prop.\ \ref{DD_scheme2}, $\sigma_1,\ldots,\sigma_{r-2}$ act as transpositions
on $\{1,\ldots,r-1\}$ and generate the group $\Sym(\{1,\ldots,r-1\})$.
Thus for the given $Z_1$, $Z_2$ there exists a product of $\sigma_i$'s, $i=1,\ldots,r-2$
mapping $\Phi_{r,Z_1}$ to $\Phi_{r,Z_2}$.
The proof for the other assertions is similar.
\end{proof}

\begin{remar}
The \coscorf which has the root systems $\Phi_{r,\{1\}}$ and $\Psi_{r,\emptyset}$
has no object with Cartan matrix of type $C_r$.
The \coscorf which has the root systems $\Phi_{r,\{1,\ldots,r-1\}}$ and $\Psi_{r,\{1,\ldots,r-2\}}$
has no object with Cartan matrix of type $D_r$.
\end{remar}

\begin{defin}
Let $\cC$ be a \coscorf of rank $r$.
If there exists a $Z\subseteq\{1,\ldots,r-1\}$ such that $\Phi_{r,Z}=\rer{a}_+$
for some object $a$, then we say that $\cC$ {\it is of type} $D'(r,|Z|)$.
If there exists a $Y\subseteq\{1,\ldots,r-1\}$ such that $\Psi_{r,Y}=\rer{a}_+$
for some object $a$, then we say that $\cC$ {\it is of type} $D'(r,|Y|+1)$.

Notice that this is well-defined by Prop.\ \ref{Z1_Z2}.
Thus if $\cC$ is of type $D'(r,0)$ then it is standard of type $D$ and
if $\cC$ is of type $D'(r,r)$ then it is standard of type $C$.
\end{defin}

\begin{theor}
Let $Z\subseteq\{1,\ldots,r-1\}$, $\cC$ be the \coscorf with $\rer{a}_+=\Phi_{r,Z}$ for an object $a$,
and $m:=|Z|$. Then $\Aut(\cC)$ is isomorphic to a reflection group of type
$C_{m}\times D_{r-m}$, where $D_1:=C_0:=$trivial group, $D_2:=A_1\times A_1$, $D_3:=A_3$, $C_1:=A_1$.
\end{theor}
\begin{proof}
By Prop.\ \ref{Z1_Z2} we may assume $Z=\{1,\ldots,m\}$.
Let $\beta_j:=\frto_{j,r}+\frto_{j,r-2}$ where $j=1,\ldots,r-1$ as in Prop.\ \ref{DD_scheme2}.
Assume first that $1<m<r-1$. Then $\sigma^a_m$ is the only simple reflection
which maps to an object with a different root system.
The maps $\sigma^a_1,\ldots,\sigma^a_{m-1}$ permute $\beta_1,\ldots,\beta_m$ and
$\sigma^a_{m+1},\ldots,\sigma^a_{r}$ generate a reflection group of type $D_{r-m}$.
Write $W(X_m)$ for the reflection group of type $X_m$. Then we have at least
$S_m\times W(D_{r-m})\le \Aut(\cC)$.

By Prop.\ \ref{Z1_Z2} there is a morphism $w$ leading to an object $b$ with
$\rer{b}_+=\Psi_{r,\{r-m+1,\ldots,r-1\}}$.
At $b$, $\sigma^d_{r-m+1},\ldots,\sigma^d_r$ generate the group $W(C_m)$ and all these
morphisms map to objects with the same root system. Conjugating this group back to $a$,
we get $W(C_m)\times W(D_{r-m})\le \Aut(\cC)$. We obtain the same result for
$m=1$ and $m=r-1$ similarly.

It remains to check that there are no more morphisms. We achieve this by
counting all morphisms to a fixed object $a$, i.e.\ by determining $n:=|\Hom(\Wg(\cC),a)|$.
Since the Cartan scheme is connected and simply connected, $n$ is the number of objects.
Now $|W(C_m)\times W(D_{r-m})|=2^m m! 2^{r-m-1} (r-m-1)!$ and we have
$\binom{r-1}{m-1}+\binom{r}{m}$ different root sets. Thus we must prove
$n = 2^{r-1}(m+r)(r-1)!$.

We proceed by induction on $r$ and $m$ and choose the object $a$ with
$\rer{a}_+=\Phi_{r,\{1,\ldots,m\}}$.
For $r\le 2$ the formula is an easy verification.
For $m=0$ the set $\pm \rer{a}_+$ is a root system of type $D$, thus $n=|W(D_r)|=2^{r-1}r!$.
Now let $m>0$. Write $J=\{2,\ldots,r\}$ and $\Wg_J(\cC)$ for the parabolic subgroupoid
of rank $r-1$ to $J$.
One can check that $\Hom(\Wg(\cC),a)$ is the union of the following ``cosets'':
\begin{center}
\begin{tabular}{l l c l}
(1)& $\id ^a \sigma_i \sigma_{i-1} \ldots \sigma_1 \Wg_J(\cC)$ &\quad& $i=0,\ldots,m-1$, \\
(2)& $\id ^a \sigma_i \sigma_{i-1} \ldots \sigma_1 \Wg_J(\cC)$ &\quad& $i=m,\ldots,r-1$, \\
(2)& $\id ^a \sigma_i \sigma_{i+1} \ldots \sigma_r \sigma_{r-2} \sigma_{r-3} \ldots \sigma_1
      \Wg_J(\cC)$ &\quad& $i=r,\ldots,m+1$, \\
(1)& $\id ^a \sigma_i \sigma_{i+1} \ldots \sigma_{r-1} \sigma_{r} \sigma_{r-1} \ldots \sigma_1
      \Wg_J(\cC)$ &\quad& $i=m,\ldots,1$.
\end{tabular}
\end{center}
Hereby, the parabolic subgroupoids $\Wg_J(\cC)$ are of type $D'(r-1,m-1)$ resp.\ $D'(r-1,m)$
in the rows labeled (1) resp.\ (2). Remark that for $m=r-1$, $\Wg_J(\cC)$ is of type
$D'(r-1,r-1)$ in the rows labeled (2); this is standard of type $C$ and has $2^{r-1}(r-1)!$
morphisms.
Hence by induction $n = 2m 2^{r-2}(m+r-2)(r-2)! + 2(r-m) 2^{r-2}(m+r-1)(r-2)!
= 2^{r-1}(m+r)(r-1)!$.
\end{proof}

Our goal is now to prove that the above \coscorfs are the only ones with
an object of Dynkin type $D'$ in rank $\ge 9$.

\begin{propo}\label{DD_rootsD}
Let $r\ge 8$ and let $\cC$ be a \coscorf of rank $r$.
Let $a$ an object of $\cC$ and assume that $\Gamma^a$ is of type $D'_r$.
Then $\Phi_{r,\{r-1\}}\subseteq \rer{a}_+$.
\end{propo}
\begin{proof}
Choose the labels for the vertices of $\Gamma^a$ as in Fig.\ \ref{fig_DD}.
We proceed by induction on $r$. For $r=8$ the claim is true by Section \ref{sec:rank_le_eight},
so let $r>8$. Since the subdiagram to the labels $2,\ldots,r$ is of type $D'$, by induction we
have $M:=\Phi_{r,\{r-1\}}\cap \langle\alpha_2,\ldots,\alpha_r\rangle\subseteq \rer{a}_+$.
Let $\beta:=\frto_{1,r}+\frto_{2,r-2}$. One computes
\begin{equation}\label{M_s}
M\cup\Big\{\sigma_1(\gamma) \mid \gamma=\sum_{i=2}^r a_i\alpha_i\in M,\:\: a_2\ne 0 \Big\}
\cup \{\alpha_1,\beta\} = \Phi_{r,\{r-1\}}.
\end{equation}
But $\sigma_1$ maps to an object with the same Dynkin diagram, so
$M\subseteq \rer{\rho_1(a)}_+$.
Further $\beta=\sigma_2(\sigma_1(\alpha_2+2\frto_{3,r-2}+\alpha_{r-1}+\alpha_r))$.
Since $\sigma_2$ also maps to an object with the same Dynkin diagram,
with Equation \eqref{M_s} we obtain $\Phi_{r,\{r-1\}}\subseteq \rer{a}_+$.
\end{proof}

\begin{theor}\label{typ_DD}
Let $r\ge 8$, let $\cC$ be a \coscorf and $a$ an object of $\cC$.
Assume that $\Gamma^a$ is of type $D'_r$ or $D_r$ with labels for the vertices as
in Fig.\ \ref{fig_DD}. Then there exists a subset $Z\subseteq\{1,\ldots,r-1\}$
such that $\rer{a}_+=\Phi_{r,Z}$.
\end{theor}
\begin{proof}
Notice first that by Prop.\ \ref{DD_rootsD}, $\Phi_{r,\{r-1\}}\subseteq \rer{a}_+$ if $\Gamma^a$
is of type $D'_r$.
Further, we know by Cor.\ \ref{phi_coscorf} that $\Phi_{r,Z}$ is a
root set of rank $r$ for all $Z\subseteq\{1,\ldots,r-1\}$.

Now assume that $\Gamma^a$ is of type $D_r$ or $D'_r$ and let $\alpha\in \rer{a}_+$.
Denote $Z_0=\{1,\ldots,r-1\}$.
We prove by induction on the height $\hg(\alpha)$ of $\alpha$ that $\alpha\in \Phi_{r,Z_0}$.
If $\hg(\alpha)=1$ then $\alpha$ is simple and we are done.
So assume $\hg(\alpha)>1$.
Applying $\sigma_i^a$ for $i=1,\ldots,r-2$ leads to an object
of Dynkin type $D_r$ or $D'_r$. If the height of $\sigma_i^a(\alpha)$ is smaller
than $\hg(\alpha)$ for such an $i$, then by induction $\sigma_i^a(\alpha)$
is in $\Phi_{r,Z_0}$.
Let $\sigma_i$ be the reflection corresponding to $\Phi_{r,Z_0}$ as in Def.\ \ref{cartan_maps}.
Then $\sigma_i=\sigma_i^a$ since the Dynkin diagram of $\Phi_{r,Z_0}$ is of type $D'$.
But then $\alpha\in \sigma_i(\Phi_{r,Z_0})=\{-\alpha_i\}\cup\Phi_{r,Z_0}\backslash\{\alpha_i\}$.

Assume that $\hg(\sigma_i^a(\alpha))\ge \hg(\alpha)$ for
$i=1,\ldots,r-2$. Then writing $\alpha=\sum_{i=1}^r a_i\alpha_i$ we obtain
\begin{eqnarray*}
a_2-a_1&\ge& a_1,\\
a_3-a_2+a_1&\ge& a_2,\\
\cdots\\
a_{r-2}-a_{r-3}+a_{r-4}&\ge& a_{r-3},\\
(a_r+a_{r-1})-a_{r-2}+a_{r-3}&\ge& a_{r-2}.
\end{eqnarray*}
This means that
\begin{equation}\label{aichain}
a_r+a_{r-1}-a_{r-2}\ge a_{r-2}-a_{r-3}\ge\ldots\ge a_2-a_1\ge a_1 \ge 0.
\end{equation}
Now if $\Gamma$ is of type $D'$ resp.\ $D$ then we compute
$\beta=\sigma_r\sigma_{r-1}\sigma_r^a(\alpha)$ resp.\ $\beta=\sigma_{r-1}\sigma_r^a(\alpha)$.
Notice that in both cases $\beta\in \rer{b}_+$ for some object $b$ of type $D$ or $D'$.
Again, if $\hg(\beta)<\hg(\alpha)$ then we are done by induction.
Assuming the converse, in both cases we obtain
\[ 0 \le \hg(\beta)-\hg(\alpha) = 2 a_{r-2} - 2 a_{r-1} - 2 a_r. \]
With (\ref{aichain}) this gives $a_r+a_{r-1}-a_{r-2}=0$, thus
$a_k-a_{k-1}=0$ for $k=2,\ldots,r-2$ and $a_1=0$. But then
$a_{r-2}=\ldots=a_1=0$, which implies $a_r+a_{r-1}=0$ and hence $\alpha=0$
contradicting $\alpha\in \rer{a}_+$.
\end{proof}

Collecting the last results we obtain the main theorem of this section:

\begin{theor}\label{thm:main_rank_gt8}
Let $\cC$ be a \coscorf of rank $r>8$ and let
\[ \rsC_+:=\{ \rer{a}_+ \mid a \in\cC\}. \]
Then there are two possibilities:\\
(1) The Cartan scheme $\cC$ is standard ($|\rsC_+|=1$) of type $A$, $B$, $C$, $D$.\\
(2) Up to equivalence the root sets of $\cC$ are given by
\[ \rsC_+ = \{ \Phi_{r,Z},\Psi_{r,Y},\Psi'_{r,Y} \mid Z,Y\subseteq\{1,\ldots,r-1\},\:\:|Z|=s,|Y|=s-1\} \]
for some $s\in \{1,\ldots,r-1\}$.

In particular, if $\cC$ is not standard then it has
\[ \binom{r-1}{s-1}+\binom{r}{s} \]
different root sets and $2^{r-1}(m+r)(r-1)!$ objects.
\end{theor}

\section{Finite \coscorfs of rank $<9$}
\label{sec:rank_le_eight}

In this section we explain the classification of \coscorfs of rank less or equal to $8$.
The proof is performed using computer calculations based on the knowledge
of the case of rank two and three (\cite{p-CH09a}, \cite{p-CH09c}).
Our algorithm described below is sufficiently powerful: The
implementation in {\sc C++} terminates within a few hours on a usual computer.

\begin{theor}\ 
\label{th:class_upto8}
\begin{enumerate}
\item Let $\cC$ be a connected Cartan scheme of rank $r$, with $3<r<9$ and $I=\{1,\ldots,r\}$.
Assume that $\rsC \re (\cC )$ is a finite irreducible root system of type $\cC$.
If $\cC$ is not equivalent to a Cartan scheme as in Cor.\ \ref{phi_coscorf},
then there exists an object $a\in A$ and a linear map $\tau \in \Aut (\ZZ ^I)$ such that
$\tau (\al _i)\in \{\al _1,\ldots,\al _r\}$ for all $i\in I$ and
$\tau (R^a_+)$ is one of the sets listed in
Appendix\ \ref{ap:rs}. Moreover, $\tau (R^a_+)$ with this property is
uniquely determined.
\item Let $R$ be one of the $24$ subsets of $\ZZ ^r$, $3<r<9$ appearing in
Appendix\ \ref{ap:rs}. There exists up to equivalence a unique
\coscorf $\cC$ such that $R\cup -R$ is the set of real roots $R^a$ in an object $a\in A$. 
Moreover $\rsC \re (\cC )$ is a finite irreducible root system of type $\cC$.
\item Let $\cC$ be a \coscorf of rank $r$ and $a\in A$. Then the Dynkin diagram
$\Gamma^a$ is one of the diagrams listed in Figure \ref{fig_dynkin}, p.\ \pageref{fig_dynkin}.
\end{enumerate}
\end{theor}

\subsection{The idea}
The classification of rank three has been achieved in \cite{p-CH09c}. Thus
here we assume that the rank $r$ is greater or equal to $4$.
Let $<$ be the lexicographic ordering on $\ZZ ^r$ such
that $\al_r<\al_{r-1}<\ldots<\al_1$.
Then $\al >0$ for any $\al \in \NN _0^r\setminus \{0\}$.

The following theorem (\cite[Thm.\ 2.10]{p-CH09c}) is crucial for the algorithm.
\begin{theor}\label{root_is_sum}
Let $\cC $ be a Cartan scheme. Assume that $\rsC \re (\cC )$ is a finite
root system of type $\cC $. Let $a\in A$ and $\al \in R^a_+$.
Then either $\al $ is simple, or it is the sum of two positive roots.
\end{theor}

By Theorem \ref{root_is_sum} we may construct $R^a_+$ inductively
by starting with $R^a_+=\{\al _r,\al _{r-1},\ldots,\al _1\}$, and
appending in each step a sum of a pair of positive roots which is greater
than all roots in $R^a_+$ we already have.
During this process, we keep track of all \rootsubsets containing at least two
positive roots, and their positive roots.

This is an overview of the algorithm without any details:

\algo{EnumerateRootSetsOverview}{$R$}
{ Enumerate all root systems containing the roots $R$}
{a set of positive roots $R$.}
{all root sets containing $R$.}
{
\item If $R\cup -R$ is a root set, output $R\cup -R$ and continue.
\item For all subspaces $U$ generated by elements of $R$, check that $R\cap U$
could be extended to a root set.
\item Set $Y := \{ \alpha+\beta \mid \alpha,\beta\in R,\:\alpha\ne\beta \}\backslash R$.
\item For all $\alpha\in Y$ with $\alpha>\max R$, call\\
{\bf EnumerateRootSetsOverview}($R\cup\alpha$).
}

But this first approach is completely impracticable. We will need many
improvements to reach our goal.

\subsection{Some technical remarks}
In fact, depending on the rank it is not always useful to compute all
\rootsubsets of all ranks because for
instance a \coscorf of rank 7 can have up to $139251$ \rootsubsets of
rank $4$, in which case we spend more time organizing the \rootsubsets than
we spare using the restrictions they give. Thus in the following, $\rho<r$ will
be the rank up to which we compute all \rootsubsetsn.

Remark that to obtain a finite number of root systems as output,
we have to ensure that we compute only irreducible systems since there are
infinitely many inequivalent reducible root systems of rank two.
Hence starting with $\{\al _r,\al _{r-1},\ldots,\al _1\}$ will not work.
Instead, for each irreducible \coscorf we take a root system $R'$ of rank $r-1$
and start with the sets
\[ R_j:=\Big\{ \sum_{i=1}^{r-1} \beta_i\alpha_{i+1} \mid
\sum_{i=1}^{r-1} \beta_i\alpha_i \in R' \Big\} \cup
\{\alpha_1, \alpha_1+\alpha_j \}, \quad j=2,\ldots,r. \]

Before starting the algorithm, we collect all irreducible \rootsubsets of rank
up to $\rho$ of all irreducible \coscorfs of rank $r-1$ in a list $\Xi$
(including all \rootsubsets with permuted coordinates).
During the algorithm, if a fragment of a \rootsubset is found to be irreducible, then it is
part of an irreducible \rootsubset of rank $r-1$, and hence it lies in $\Xi$.
We also store the list $\Upsilon$ of all roots for all ranks $2,\ldots,\rho$ that
appear in $\Xi$. This way we never need to fill the memory with coordinates but
only with labels pointing to the root in $\Upsilon$.

\begin{defin}
Let $R_+$ be the set of positive roots of a fragment of a root subset (see \ref{subsect:rsf})
of rank three and $M$ the set of planes containing at least two elements of $R_+$.
We call the number
\[ \varepsilon_{R_+}:=3+\sum_{V\in M} \left(|V\cap R_+|-3\right) \]
the {\it Euler invariant} of $R_+$.
\end{defin}
If $R_+$ is a root subset of rank three, then $\varepsilon_{R_+}=0$ by
\cite[Thm.\ 3.17]{p-CH09c}.

\subsection{The rsf}\label{subsect:rsf}
In this section, we will call {\it root system fragment} (or {\it rsf}) the
following set of data associated to a set of positive roots $R$
in construction:
\begin{itemize}
\item An ordered set of positive vectors $R$.
\item For each rank $2,\ldots,\rho$, the sequence of fragments of \rootsubsetsn.
      Each such fragment consists of:
\begin{enumerate}
\item A subspace $U$ of $\QQ^r$, a matrix used for a membership test
      for this subspace, and a matrix needed to compute the coordinates
      of a given element with respect to the basis.
\item A hash value allowing us to perform a fast equality test for the subspaces.
\item Labels from $\Upsilon$ for the roots of $R$ in $U$ with respect
      to the basis of $U$.
\item Positions of the roots of $R$ in $U$ in the lexicographically ordered set $R$.
\item The adjacency matrix of the Dynkin diagram (so far) of $U$ and
      a flag whether it is connected.
\end{enumerate}
\item For the fragments of \rootsubsets of rank two: all entries of the Cartan matrices
      and flags indicating if the \rootsubset is ``finished''.
\item The Euler invariants of all fragments of rank three \rootsubsetsn.
\item Adjacency matrices of all fragments of parabolic subgroupoids.
\item A flag ``isvalid'' telling if all the above data are consistent.
\end{itemize}
These data are continuously updated during the algorithm.

\subsection{More remarks}
Although the algorithm looks similar to the algorithm enumerating the
root systems of rank three in \cite{p-CH09c}, this version is much more work to
implement for several reasons: The main reason is that
we need linear algebra for the subspaces generated by the roots
of a \rootsubsetn. This includes an implementation of small rational numbers,
Gau\ss{} algorithm, a fast membership test and
hash-values for the subspaces. Further we need a good memory management
for these subspaces to avoid duplicate versions of them.
But there are even more functions needed, for example a test to decide if the
Dynkin diagrams are connected.

Of course all these functions exist in computer algebra systems, but
unfortunately they do not reach the desired performance mainly for two reasons:
The first reason is that all these systems use arbitrary-precision
arithmetic and in our situation the coefficients of the roots never get bigger
than $14$. The second reason is that these systems spend much time
interpreting the code and dynamically determining the types of variables.
For instance the computation of all root systems of rank $7$ takes several weeks
using a computer algebra prototype and takes only $12$ minutes with the {\sc C++}
version. Since the {\sc C++} version needs approximatively $3$ hours for rank $8$,
we guess that a version on any computer algebra system would take little less than
a year. Besides, the computer algebra prototype uses a huge amount of memory.

\subsection{The algorithm}
This is the main recursion of the algorithm:

\algo{EnumerateRootSets}{$D$[,$\gamma$]}
{ Enumerate all root sets containing the roots of $D$ }
{an rsf $D$, possibly a required root $\gamma$.}
{all root sets containing the roots of $D$.}
{
\item If the Euler invariants of all fragments of rank three
\rootsubsets in $D$ are $0$, then check if $D$ yields a root system. If yes,
output $D$ and continue.
\item If no required root $\gamma$ is known, then:
For all fragments $U$ of irreducible \rootsubsets in $D$, search for possible completions
in $\Xi$. If $U$ may not be completed, then return $\emptyset$. Otherwise try
to determine a smallest root $\gamma$ which is missing and which will be included in any case,
call {\bf EnumerateRootSets}($D$,$\gamma$) if successful and return.
\item Denote $R$ the positive roots of $D$. Set\\
$Y \leftarrow \{ \alpha+\beta \mid \alpha,\beta\in R,\:\alpha\ne\beta \}\backslash R$.
\item\label{enum_step4} For all $\alpha\in Y$ with $[\gamma\ge] \alpha>\max R$, call
$\tilde D:=${\bf AppendRoot}($D$,$\alpha$);
if $\tilde D$ is valid, then call {\bf EnumerateRootSets}($D$[,$\gamma$]).
}

In practice we use a global list $\Omega$ in which we note which $R$ of an rsf has
already been treated. The first step in ``EnumerateRootSets'' is to check if $R$
is in $\Omega$. The following proposition gives more details:

\begin{propo}\label{omega}
If $R$ contains a root $\beta=\sum_i a_i\alpha_i$ with $a_1>1$, then we can
include the images of $R$ under $\sigma_2,\ldots,\sigma_r$ into $\Omega$,
and by the way we check if there is a contradiction (for example if these images
contain roots with positive and negative coefficients).
\end{propo}
\begin{proof}
Assume that $R$ contains a root $\beta=\sum_i a_i\alpha_i$ with $a_1>1$ and that $\beta$
is the greatest root in $R$.
Then all roots of the form $m\alpha_i+\alpha_j$ with $m\in\NN$ and $i,j>1$ are smaller than
$\beta$. Hence if $R$ is to become a root set $R_+^a$ some day (after including roots which
are greater than $\beta$), then its Cartan entries $c_{ij}$ with $i,j>1$ are already
known:
\[ c_{ij} = -\max\{m\in\NN \mid m\alpha_i+\alpha_j\in R_+^a \}. \]
The same holds for the entries $c_{i1}$, $i>1$. Therefore, the reflections
$\sigma_2,\ldots,\sigma_r$ are known.

Now let $i>1$ and consider $R':=\sigma_i(R)$. If we include $R'$ into $\Omega$, then
we have to ensure that all root sets constructed upon $R'$ have or will be handled
at some point. Thus assume $R''=R'\cup E$ where $E$ consists of roots greater than all
roots of $R'$. The roots of $\sigma_i(E)$ which are greater than $\beta$ do not pose
a problem, because they will be considered in future. So let
$\delta:=k\alpha_1+\gamma\in E$ with $\gamma\in\langle\alpha_2,\ldots,\alpha_r\rangle$
and $\sigma_i(\delta)< \beta$. But
\[ \sigma_i(\delta) = k\alpha_1-kc_{i1}\alpha_i+\sigma_i(\gamma) > k\alpha_1, \]
so $\sigma_i(\delta)$ is not a root from the starting set of roots.
If $\sigma_i(\delta)\notin R$, then $R\cup\{\sigma_i(\delta)\}$ is a set
of roots which has already been considered in an earlier stage of the algorithm.
\end{proof}
\begin{remar}
If we use Prop.\ \ref{omega}, then it is essential to append new roots
from $Y$ in lexicographical order in step \ref{enum_step4} of Algo.\ 4.2.
\end{remar}

The time consuming part of the algorithm is the function ``AppendRoot'':

\algo{AppendRoot}{$D$, $\alpha$}
{Append a root to an rsf}
{an rsf $D$, a root $\alpha$.}
{an rsf $\tilde D$ consisting of $D$ with $\alpha$ included.}
{
\item Copy the data of $D$ to a new rsf $\tilde D$.
\item The non-zero coordinates of $\alpha$ define a parabolic
subgroupoid $P$ of rank $s$ which will be irreducible in $\tilde D$.
If $s=2$, then update the adjacency matrix for all parabolic subgroupoids
containing $\alpha$.
Otherwise, if the Dynkin diagram of $P$ is not connected, then set
isvalid$:=${\bf false} and return.
\item For each positive root $\beta$ in $R$ compute the \rootsubset
$U:=\langle\alpha,\beta\rangle$. If $U$ is new, then include it to
the rsf $\tilde D$.
\item For each \rootsubset $U$ of rank $2,\ldots,\rho$ in $D$, test if $\alpha$ is
in $U$. If it is, then include it into $U$ in $\tilde D$, update the adjacency matrix
and test if the Dynkin diagram is connected;
compute its coordinates with respect to the basis: if they are not all
non-negative integers, then return $\tilde D$ with isvalid$:=${\bf false}.

If $U$ has rank $2$, update the Cartan entries: here we can test if the sequence
of Cartan entries is valid and return $\tilde D$ with isvalid$:=${\bf false} if it is not.

If $U$ has rank $3$, then update its Euler invariant.

If $\alpha\notin U$, then remember $U$.
\item For all \rootsubsets $U$ of rank $e$ we have remembered,
create a \rootsubset $U'$ of rank $e+1$ by including $\alpha$. Test if it is new by using
its hash value. If $e=2$, then initialize the Euler invariant of $U'$.
\item Return $\tilde D$ with isvalid$:=${\bf true}.
}

Finally, we still need a function to check which of the rsf is indeed a root set
(see \cite[Algo.\ 4.5]{p-CH09c}):

\algo{RootSetsForAllObjects}{$R$}
{Returns the roots for all objects if $R=R^a_+$ determines a Cartan
scheme $\cC $ such that $\rsC \re (\cC )$ is an irreducible root system}
{$R$ the set of positive roots at one object.}
{the set of roots at all objects, or $\emptyset$ if $R$ does
not yield a Cartan scheme as desired.}
{\label{A4}
\item $N \leftarrow [R]$, $M \leftarrow \emptyset$.
\item While $|N| > 0$, do steps \ref{begwhile} to \ref{endwhile}.
\item Let $F$ be the last element of $N$. Remove $F$ from $N$ and include it to $M$.\label{begwhile}
\item Compute the $r$ simple reflections given by $C_F$.
\item For each simple reflection $s$, do:\label{endwhile}
\begin{itemize}
\item Compute $G:=\{s(v)\mid v\in F\}$. If an element of $G$ has positive and negative coefficients,
then return $\emptyset$. Otherwise multiply the negative roots of $G$ by $-1$.
\item If $G\notin M$, then append $G$ to $N$.
\end{itemize}
\item Return $M$.}

\begin{appendix}
\section{Sporadic \coscorfs}

\subsection{Summary}
We will call {\it sporadic} the irreducible \coscorfs of rank $\ge 3$
not included in the series described in Section \ref{sec:rank_gt_eight}
because they do not seem to fit into a pattern.
Among them are those of type $F_4$, $E_6$, $E_7$, $E_8$.
In this section we summarize invariants of the sporadic \coscorfsn.

\begin{figure}[h]
\begin{tabular}{| l r r r r r r r | r |}
\hline
Rank & $2$ & $3$ & $4$ & $5$ & $6$ & $7$ & $8$ & $r > 8$ \\
Number & $\infty$ & $55$ & $18$ & $14$ & $13$ & $12$ & $12$ & $r+3$ \\
\hline
\end{tabular}
\caption{Number of irreducible \coscorfsn}
\label{number_table}
\end{figure}

Fig.\ \ref{number_table} shows an overview of the output of the above algorithms and
Section \ref{sec:rank_gt_eight}.
We thus have $50+11+6+4+2+1 = 74$ sporadic \coscorfs (in rank three only $5$ \coscorfs
are not sporadic because $A_3=D_3$).

On an {\it i7} with $2,8$ GHz, our C++ implementation of the algorithm
(including the final check whether the sets are root sets and the computation of ``canonical'' objects)
% needs $380.30$ sec, $96.54$ sec, $952.80$ sec, $708.18$ sec, $10900.88$ sec
% needs $6$ min., $2$ min., $16$ min., $12$ min., $182$ min.\ 
needs $6$ min., $2$ min., $16$ min., $12$ min., $170$ min.\ 
for the ranks $4$, $5$, $6$, $7$, $8$ respectively.
Remark that using Prop.\ \ref{omega} and the set $\Omega$ reduces the
runtime in all cases except for the case of rank $8$ where the runtime is increased
by $12$ minutes.
% Without Prop.\ \ref{omega}:
% Time for rank 4: 736.260000
% Time for rank 5: 161.290000
% Time for rank 6: 949.870000
% Time for rank 7: 1049.720000
% Time for rank 8: 10204.240000

\subsection{Dynkin diagrams}

Figure \ref{fig_dynkin} displays all Dynkin diagrams of irreducible
\coscorfs of arbitrary rank. They are obtained from the data in Section \ref{ap:rs}
and \cite{p-CH09c} by Lemma \ref{co:cij}.
\begin{figure}
\setlength{\unitlength}{3300sp}
\begingroup\makeatletter\ifx\SetFigFont\undefined
\gdef\SetFigFont#1#2#3#4#5{\reset@font\fontsize{#1}{#2pt}
  \fontfamily{#3}\fontseries{#4}\fontshape{#5}\selectfont}
\fi\endgroup
\begin{picture}(8773,11505)(2300,-11176)
\thicklines
{\put(6751, 89){\line( 1, 0){600}}}
{\put(7351, 89){\vector( 1, 0){600}}}
{\put(6751,-361){\line( 1, 0){600}}}
{\put(7951,-361){\vector(-1, 0){600}}}
{\put(6751,-811){\line( 1, 0){600}}}
{\put(7951,-811){\vector(-1, 0){600}}}
{\put(6751,-1261){\line( 1, 0){600}}}
{\put(7351,-1261){\vector( 1, 0){600}}}
{\put(6751,-1711){\vector( 1, 0){600}}}
{\put(7351,-1711){\vector( 1, 0){600}}}
{\put(6751,-2161){\vector( 1, 0){600}}}
{\put(7951,-2161){\vector(-1, 0){600}}}
{\put(6751,-2611){\vector( 1, 0){600}}}
{\put(7951,-2611){\vector(-1, 0){600}}}
{\put(6751,-3061){\vector( 1, 0){600}}}
{\put(7951,-3061){\vector(-1, 0){600}}}
{\put(6751,-3511){\vector( 1, 0){600}}}
{\put(7951,-3511){\vector(-1, 0){600}}}
{\put(6751,-3961){\vector( 1, 0){600}}}
{\put(7951,-3961){\vector(-1, 0){600}}}
{\put(6751,-4411){\vector( 1, 0){600}}}
{\put(7951,-4486){\vector(-1, 0){600}}}
{\put(7351,-4336){\vector( 1, 0){600}}}
{\put(6751,-4936){\vector( 1, 0){600}}}
{\put(7951,-5011){\vector(-1, 0){600}}}
{\put(7351,-4861){\vector( 1, 0){600}}}
{\put(6751,-5461){\vector( 1, 0){600}}}
{\put(7951,-5536){\vector(-1, 0){600}}}
{\put(7351,-5386){\vector( 1, 0){600}}}
{\put(6751,-5986){\vector( 1, 0){600}}}
{\put(7951,-6061){\vector(-1, 0){600}}}
{\put(7351,-5911){\vector( 1, 0){600}}}
{\put(9826,-9136){\line( 0,-1){600}}}
{\put(9826,-9736){\line( 0,-1){600}}}
{\put(9826,-10336){\line( 2,-5){237.931}}}
{\put(9826,-10336){\line(-2,-5){237.931}}}
{\put(9001,-9136){\line( 0,-1){600}}}
{\put(9001,-9736){\line( 0,-1){600}}}
{\put(9226,-10936){\line(-1, 0){450}}}
{\put(9001,-10336){\line( 2,-5){237.931}}}
{\put(9001,-10336){\vector( -1, -3){207}}}
{\put(9601,-10936){\vector( 1, 0){450}}}
{\put(1801, 89){\line( 1, 0){600}}
\put(2401, 89){\line( 1, 0){600}}}
{\multiput(3001, 89)(200,0.00000){6}{\makebox(1.6667,11.6667){\SetFigFont{5}{6}{\rmdefault}{\mddefault}{\updefault}$\bullet$}}}
{\put(1801,-436){\line( 1, 0){600}}
\put(2401,-436){\line( 1, 0){600}}}
{\multiput(3001,-436)(200,0.00000){6}{\makebox(1.6667,11.6667){\SetFigFont{5}{6}{\rmdefault}{\mddefault}{\updefault}$\bullet$}}}
{\put(4201,-436){\line( 1, 0){600}}}
{\put(4801,-436){\vector( 1, 0){600}}}
{\put(1801,-961){\line( 1, 0){600}}
\put(2401,-961){\line( 1, 0){600}}}
{\multiput(3001,-961)(200,0.00000){6}{\makebox(1.6667,11.6667){\SetFigFont{5}{6}{\rmdefault}{\mddefault}{\updefault}$\bullet$}}}
{\put(4201,-961){\line( 1, 0){600}}}
{\put(5401,-961){\vector(-1, 0){600}}}
{\put(1801,-1561){\line( 1, 0){600}}
\put(2401,-1561){\line( 1, 0){600}}}
{\multiput(3001,-1561)(200,0.00000){6}{\makebox(1.6667,11.6667){\SetFigFont{5}{6}{\rmdefault}{\mddefault}{\updefault}$\bullet$}}}
{\put(4201,-1561){\line( 1, 0){600}}}
{\put(4801,-1561){\line( 5, 2){594.828}}}
{\put(4801,-1561){\line( 5,-2){594.828}}}
{\put(5401,-1336){\line( 0,-1){450}}}
{\put(1801,-2311){\line( 1, 0){600}}
\put(2401,-2311){\line( 1, 0){600}}}
{\put(4201,-2311){\line( 1, 0){600}}}
{\put(4801,-2311){\line( 5, 2){594.828}}}
{\put(4801,-2311){\line( 5,-2){594.828}}}
{\multiput(3001,-2311)(200,0.00000){6}{\makebox(1.6667,11.6667){\SetFigFont{5}{6}{\rmdefault}{\mddefault}{\updefault}$\bullet$}}}
{\put(1801,-3211){\line( 1, 0){2400}}}
{\put(3001,-2836){\line( 0,-1){375}}}
{\put(1801,-3961){\line( 1, 0){3000}}}
{\put(3001,-3586){\line( 0,-1){375}}}
{\put(1801,-4711){\line( 1, 0){3600}}}
{\put(3001,-4336){\line( 0,-1){375}}}
{\put(2701,-5086){\line( 4,-5){300}}}
{\put(1801,-5461){\line( 1, 0){1800}}}
{\put(2701,-5086){\line(-4,-5){300}}}
{\put(1801,-6211){\line( 1, 0){2400}}}
{\put(2701,-5836){\line( 4,-5){300}}}
{\put(2701,-5836){\line(-4,-5){300}}}
{\put(1801,-6961){\line( 1, 0){3000}}}
{\put(2701,-6586){\line( 4,-5){300}}}
{\put(2701,-6586){\line(-4,-5){300}}}

{\put(5776,-10036){\line( 1, 0){600}}}
{\put(6376,-10036){\line( 1, 0){600}}}
{\put(6976,-10036){\vector( 1, 0){600}}}
{\put(7576,-10036){\line( 1, 0){600}}}
{\put(5776,-10486){\line( 1, 0){600}}}
{\put(6376,-10486){\line( 1, 0){600}}}
{\put(7576,-10486){\line( 1, 0){600}}}
{\put(7576,-10486){\vector(-1, 0){600}}}
{\put(5776,-10936){\line( 1, 0){600}}}
{\put(6376,-10936){\line( 1, 0){600}}}
{\put(6976,-10936){\vector( 1, 0){600}}}
{\put(8176,-10936){\vector(-1, 0){600}}}
{\put(8776,-7486){\line( 1, 0){600}}}
{\put(9376,-7486){\line( 5, 2){594.828}}}
{\put(9976,-7261){\line( 0,-1){450}}}
{\put(10030,-7715){\vector(-3, 1){607.500}}}
{\put(8776,-8386){\line( 1, 0){600}}}
{\put(9376,-8386){\line( 5, 2){594.828}}}
{\put(9376,-8386){\line( 5,-2){594.828}}}
{\put(9976,-8611){\vector( 0, 1){450}}}
{\put(9526,239){\line(-1,-2){300}}}
{\put(9526,239){\line( 1,-2){300}}}
{\put(9226,-361){\vector( 1, 0){600}}}
{\put(9526,-586){\line(-1,-2){300}}}
{\put(9526,-586){\line( 1,-2){300}}}
{\put(9226,-1186){\vector( 1, 0){600}}}
{\put(9526,-1411){\line(-1,-2){300}}}
{\put(9526,-1411){\line( 1,-2){300}}}
{\put(9226,-2011){\vector( 1, 0){600}}}
{\put(9526,-3061){\line(-1,-2){300}}}
{\put(9226,-3661){\vector( 1, 0){600}}}
{\put(9526,-3886){\line(-1,-2){300}}}
{\put(9526,-3886){\line( 1,-2){300}}}
{\put(9226,-4561){\vector( 1, 0){600}}}
{\put(9826,-4411){\vector(-1, 0){600}}}
{\put(9526,-4936){\line(-1,-2){300}}}
{\put(9526,-4936){\line( 1,-2){300}}}
{\put(9226,-5611){\vector( 1, 0){600}}}
{\put(9826,-5461){\vector(-1, 0){600}}}
{\put(9526,-5986){\line(-1,-2){300}}}
{\put(9226,-6661){\vector( 1, 0){600}}}
{\put(9826,-6511){\vector(-1, 0){600}}}
{\put(9526,-2236){\line(-1,-2){300}}}
{\put(9226,-2836){\vector( 1, 0){600}}}
{\put(9526,-2236){\vector( 1,-2){300}}}
{\put(9826,-3661){\vector(-1, 2){300}}}
{\put(9526,-5986){\vector( 1,-2){300}}}
{\put(6151,-7111){\line( 1, 0){600}}}
{\put(7951,-7111){\vector(-1, 0){600}}}
{\put(6151,-6661){\line( 1, 0){600}}}
{\put(7351,-6661){\line( 1, 0){600}}}
{\put(6751,-6661){\vector( 1, 0){600}}}
{\put(6151,-9361){\line( 1, 0){600}}}
{\put(6151,-8911){\line( 1, 0){600}}}
{\put(7951,-8911){\vector(-1, 0){600}}}
{\put(6751,-8911){\vector( 1, 0){600}}}
{\put(6151,-8461){\line( 1, 0){600}}}
{\put(7351,-8461){\vector( 1, 0){600}}}
{\put(6751,-8461){\line( 1, 0){600}}}
{\put(6151,-7561){\line( 1, 0){600}}}
{\put(7951,-7561){\vector(-1, 0){600}}}
{\put(6751,-7561){\vector( 1, 0){600}}}
{\put(6151,-8011){\line( 1, 0){600}}}
{\put(7951,-8011){\vector(-1, 0){600}}}
{\put(6751,-8011){\line( 1, 0){600}}}
{\put(7351,-7111){\vector(-1, 0){600}}}
{\put(6751,-9361){\line( 1, 0){600}}}
{\put(7951,-9286){\vector(-1, 0){600}}}
{\put(7351,-9436){\vector( 1, 0){600}}}
\put(7576,164){\makebox(0,0)[lb]{\smash{{\SetFigFont{12}{14.4}{\rmdefault}{\mddefault}{\updefault}{3}}}}}
\put(6226, 14){\makebox(0,0)[lb]{\smash{{\SetFigFont{12}{14.4}{\rmdefault}{\mddefault}{\updefault}{$\Gamma_3^1$}}}}}
\put(7576,-286){\makebox(0,0)[lb]{\smash{{\SetFigFont{12}{14.4}{\rmdefault}{\mddefault}{\updefault}{3}}}}}
\put(6226,-436){\makebox(0,0)[lb]{\smash{{\SetFigFont{12}{14.4}{\rmdefault}{\mddefault}{\updefault}{$\Gamma_3^2$}}}}}
\put(7576,-736){\makebox(0,0)[lb]{\smash{{\SetFigFont{12}{14.4}{\rmdefault}{\mddefault}{\updefault}{4}}}}}
\put(6226,-886){\makebox(0,0)[lb]{\smash{{\SetFigFont{12}{14.4}{\rmdefault}{\mddefault}{\updefault}{$\Gamma_3^3$}}}}}
\put(7576,-1186){\makebox(0,0)[lb]{\smash{{\SetFigFont{12}{14.4}{\rmdefault}{\mddefault}{\updefault}{4}}}}}
\put(6226,-1336){\makebox(0,0)[lb]{\smash{{\SetFigFont{12}{14.4}{\rmdefault}{\mddefault}{\updefault}{$\Gamma_3^4$}}}}}
\put(6976,-1636){\makebox(0,0)[lb]{\smash{{\SetFigFont{12}{14.4}{\rmdefault}{\mddefault}{\updefault}{2}}}}}
\put(7576,-1636){\makebox(0,0)[lb]{\smash{{\SetFigFont{12}{14.4}{\rmdefault}{\mddefault}{\updefault}{2}}}}}
\put(6226,-1786){\makebox(0,0)[lb]{\smash{{\SetFigFont{12}{14.4}{\rmdefault}{\mddefault}{\updefault}{$\Gamma_3^5$}}}}}
\put(6976,-2086){\makebox(0,0)[lb]{\smash{{\SetFigFont{12}{14.4}{\rmdefault}{\mddefault}{\updefault}{2}}}}}
\put(7576,-2086){\makebox(0,0)[lb]{\smash{{\SetFigFont{12}{14.4}{\rmdefault}{\mddefault}{\updefault}{2}}}}}
\put(6226,-2236){\makebox(0,0)[lb]{\smash{{\SetFigFont{12}{14.4}{\rmdefault}{\mddefault}{\updefault}{$\Gamma_3^6$}}}}}
\put(6976,-2536){\makebox(0,0)[lb]{\smash{{\SetFigFont{12}{14.4}{\rmdefault}{\mddefault}{\updefault}{2}}}}}
\put(7576,-2536){\makebox(0,0)[lb]{\smash{{\SetFigFont{12}{14.4}{\rmdefault}{\mddefault}{\updefault}{3}}}}}
\put(6226,-2686){\makebox(0,0)[lb]{\smash{{\SetFigFont{12}{14.4}{\rmdefault}{\mddefault}{\updefault}{$\Gamma_3^7$}}}}}
\put(6976,-2986){\makebox(0,0)[lb]{\smash{{\SetFigFont{12}{14.4}{\rmdefault}{\mddefault}{\updefault}{2}}}}}
\put(7576,-2986){\makebox(0,0)[lb]{\smash{{\SetFigFont{12}{14.4}{\rmdefault}{\mddefault}{\updefault}{4}}}}}
\put(6226,-3136){\makebox(0,0)[lb]{\smash{{\SetFigFont{12}{14.4}{\rmdefault}{\mddefault}{\updefault}{$\Gamma_3^8$}}}}}
\put(6976,-3436){\makebox(0,0)[lb]{\smash{{\SetFigFont{12}{14.4}{\rmdefault}{\mddefault}{\updefault}{2}}}}}
\put(7576,-3436){\makebox(0,0)[lb]{\smash{{\SetFigFont{12}{14.4}{\rmdefault}{\mddefault}{\updefault}{6}}}}}
\put(6226,-3586){\makebox(0,0)[lb]{\smash{{\SetFigFont{12}{14.4}{\rmdefault}{\mddefault}{\updefault}{$\Gamma_3^9$}}}}}
\put(6976,-5386){\makebox(0,0)[lb]{\smash{{\SetFigFont{12}{14.4}{\rmdefault}{\mddefault}{\updefault}{2}}}}}
\put(7576,-5386){\makebox(0,0)[lb]{\smash{{\SetFigFont{12}{14.4}{\rmdefault}{\mddefault}{\updefault}{2}}}}}
\put(6976,-5911){\makebox(0,0)[lb]{\smash{{\SetFigFont{12}{14.4}{\rmdefault}{\mddefault}{\updefault}{2}}}}}
\put(7576,-6211){\makebox(0,0)[lb]{\smash{{\SetFigFont{12}{14.4}{\rmdefault}{\mddefault}{\updefault}{3}}}}}
\put(7576,-5911){\makebox(0,0)[lb]{\smash{{\SetFigFont{12}{14.4}{\rmdefault}{\mddefault}{\updefault}{2}}}}}
\put(7576,-4636){\makebox(0,0)[lb]{\smash{{\SetFigFont{12}{14.4}{\rmdefault}{\mddefault}{\updefault}{2}}}}}
\put(7576,-4336){\makebox(0,0)[lb]{\smash{{\SetFigFont{12}{14.4}{\rmdefault}{\mddefault}{\updefault}{2}}}}}
\put(7576,-4861){\makebox(0,0)[lb]{\smash{{\SetFigFont{12}{14.4}{\rmdefault}{\mddefault}{\updefault}{3}}}}}
\put(7576,-5161){\makebox(0,0)[lb]{\smash{{\SetFigFont{12}{14.4}{\rmdefault}{\mddefault}{\updefault}{2}}}}}
\put(7576,-5686){\makebox(0,0)[lb]{\smash{{\SetFigFont{12}{14.4}{\rmdefault}{\mddefault}{\updefault}{2}}}}}
\put(6976,-3886){\makebox(0,0)[lb]{\smash{{\SetFigFont{12}{14.4}{\rmdefault}{\mddefault}{\updefault}{3}}}}}
\put(7576,-3886){\makebox(0,0)[lb]{\smash{{\SetFigFont{12}{14.4}{\rmdefault}{\mddefault}{\updefault}{3}}}}}
\put(6226,-4036){\makebox(0,0)[lb]{\smash{{\SetFigFont{12}{14.4}{\rmdefault}{\mddefault}{\updefault}{$\Gamma_3^{10}$}}}}}
\put(6226,-4486){\makebox(0,0)[lb]{\smash{{\SetFigFont{12}{14.4}{\rmdefault}{\mddefault}{\updefault}{$\Gamma_3^{11}$}}}}}
\put(6226,-5011){\makebox(0,0)[lb]{\smash{{\SetFigFont{12}{14.4}{\rmdefault}{\mddefault}{\updefault}{$\Gamma_3^{12}$}}}}}
\put(6226,-5536){\makebox(0,0)[lb]{\smash{{\SetFigFont{12}{14.4}{\rmdefault}{\mddefault}{\updefault}{$\Gamma_3^{13}$}}}}}
\put(6226,-6061){\makebox(0,0)[lb]{\smash{{\SetFigFont{12}{14.4}{\rmdefault}{\mddefault}{\updefault}{$\Gamma_3^{14}$}}}}}
\put(8701,-10711){\makebox(0,0)[lb]{\smash{{\SetFigFont{12}{14.4}{\rmdefault}{\mddefault}{\updefault}{2}}}}}
\put(9751,-11161){\makebox(0,0)[lb]{\smash{{\SetFigFont{12}{14.4}{\rmdefault}{\mddefault}{\updefault}{2}}}}}
\put(9351,-9886){\makebox(0,0)[lb]{\smash{{\SetFigFont{12}{14.4}{\rmdefault}{\mddefault}{\updefault}{$\Gamma_5^5$}}}}}
\put(8526,-9886){\makebox(0,0)[lb]{\smash{{\SetFigFont{12}{14.4}{\rmdefault}{\mddefault}{\updefault}{$\Gamma_5^4$}}}}}
\put(1351, 14){\makebox(0,0)[lb]{\smash{{\SetFigFont{12}{14.4}{\rmdefault}{\mddefault}{\updefault}{$A$}}}}}
\put(5101,-361){\makebox(0,0)[lb]{\smash{{\SetFigFont{12}{14.4}{\rmdefault}{\mddefault}{\updefault}{2}}}}}
\put(1351,-511){\makebox(0,0)[lb]{\smash{{\SetFigFont{12}{14.4}{\rmdefault}{\mddefault}{\updefault}{$B$}}}}}
\put(5101,-886){\makebox(0,0)[lb]{\smash{{\SetFigFont{12}{14.4}{\rmdefault}{\mddefault}{\updefault}{2}}}}}
\put(1351,-1036){\makebox(0,0)[lb]{\smash{{\SetFigFont{12}{14.4}{\rmdefault}{\mddefault}{\updefault}{$C$}}}}}
\put(1351,-1636){\makebox(0,0)[lb]{\smash{{\SetFigFont{12}{14.4}{\rmdefault}{\mddefault}{\updefault}{$D'$}}}}}
\put(1351,-2386){\makebox(0,0)[lb]{\smash{{\SetFigFont{12}{14.4}{\rmdefault}{\mddefault}{\updefault}{$D$}}}}}
\put(1351,-3286){\makebox(0,0)[lb]{\smash{{\SetFigFont{12}{14.4}{\rmdefault}{\mddefault}{\updefault}{$\Gamma_6^3$}}}}}
\put(1351,-4036){\makebox(0,0)[lb]{\smash{{\SetFigFont{12}{14.4}{\rmdefault}{\mddefault}{\updefault}{$\Gamma_7^1$}}}}}
\put(1351,-4786){\makebox(0,0)[lb]{\smash{{\SetFigFont{12}{14.4}{\rmdefault}{\mddefault}{\updefault}{$\Gamma_8^1$}}}}}
\put(1351,-5536){\makebox(0,0)[lb]{\smash{{\SetFigFont{12}{14.4}{\rmdefault}{\mddefault}{\updefault}{$\Gamma_5^6$}}}}}
\put(1351,-6286){\makebox(0,0)[lb]{\smash{{\SetFigFont{12}{14.4}{\rmdefault}{\mddefault}{\updefault}{$\Gamma_6^4$}}}}}
\put(1351,-7036){\makebox(0,0)[lb]{\smash{{\SetFigFont{12}{14.4}{\rmdefault}{\mddefault}{\updefault}{$\Gamma_7^2$}}}}}
\put(7201,-9961){\makebox(0,0)[lb]{\smash{{\SetFigFont{12}{14.4}{\rmdefault}{\mddefault}{\updefault}{2}}}}}
\put(5251,-10111){\makebox(0,0)[lb]{\smash{{\SetFigFont{12}{14.4}{\rmdefault}{\mddefault}{\updefault}{$\Gamma_5^1$}}}}}
\put(7201,-10411){\makebox(0,0)[lb]{\smash{{\SetFigFont{12}{14.4}{\rmdefault}{\mddefault}{\updefault}{2}}}}}
\put(5251,-10561){\makebox(0,0)[lb]{\smash{{\SetFigFont{12}{14.4}{\rmdefault}{\mddefault}{\updefault}{$\Gamma_5^2$}}}}}
\put(7201,-10861){\makebox(0,0)[lb]{\smash{{\SetFigFont{12}{14.4}{\rmdefault}{\mddefault}{\updefault}{2}}}}}
\put(5251,-11011){\makebox(0,0)[lb]{\smash{{\SetFigFont{12}{14.4}{\rmdefault}{\mddefault}{\updefault}{$\Gamma_5^3$}}}}}
\put(7876,-10861){\makebox(0,0)[lb]{\smash{{\SetFigFont{12}{14.4}{\rmdefault}{\mddefault}{\updefault}{2}}}}}
\put(9601,-7861){\makebox(0,0)[lb]{\smash{{\SetFigFont{12}{14.4}{\rmdefault}{\mddefault}{\updefault}{2}}}}}
\put(10051,-8536){\makebox(0,0)[lb]{\smash{{\SetFigFont{12}{14.4}{\rmdefault}{\mddefault}{\updefault}{2}}}}}
\put(9451,-286){\makebox(0,0)[lb]{\smash{{\SetFigFont{12}{14.4}{\rmdefault}{\mddefault}{\updefault}{2}}}}}
\put(9451,-1111){\makebox(0,0)[lb]{\smash{{\SetFigFont{12}{14.4}{\rmdefault}{\mddefault}{\updefault}{3}}}}}
\put(9451,-4336){\makebox(0,0)[lb]{\smash{{\SetFigFont{12}{14.4}{\rmdefault}{\mddefault}{\updefault}{2}}}}}
\put(9451,-5386){\makebox(0,0)[lb]{\smash{{\SetFigFont{12}{14.4}{\rmdefault}{\mddefault}{\updefault}{2}}}}}
\put(9451,-5836){\makebox(0,0)[lb]{\smash{{\SetFigFont{12}{14.4}{\rmdefault}{\mddefault}{\updefault}{3}}}}}
\put(9451,-6436){\makebox(0,0)[lb]{\smash{{\SetFigFont{12}{14.4}{\rmdefault}{\mddefault}{\updefault}{2}}}}}
\put(9451,-1936){\makebox(0,0)[lb]{\smash{{\SetFigFont{12}{14.4}{\rmdefault}{\mddefault}{\updefault}{4}}}}}
\put(9451,-2761){\makebox(0,0)[lb]{\smash{{\SetFigFont{12}{14.4}{\rmdefault}{\mddefault}{\updefault}{2}}}}}
\put(9751,-2536){\makebox(0,0)[lb]{\smash{{\SetFigFont{12}{14.4}{\rmdefault}{\mddefault}{\updefault}{2}}}}}
\put(9451,-3586){\makebox(0,0)[lb]{\smash{{\SetFigFont{12}{14.4}{\rmdefault}{\mddefault}{\updefault}{2}}}}}
\put(9751,-3361){\makebox(0,0)[lb]{\smash{{\SetFigFont{12}{14.4}{\rmdefault}{\mddefault}{\updefault}{2}}}}}
\put(9451,-4786){\makebox(0,0)[lb]{\smash{{\SetFigFont{12}{14.4}{\rmdefault}{\mddefault}{\updefault}{2}}}}}
\put(9451,-6886){\makebox(0,0)[lb]{\smash{{\SetFigFont{12}{14.4}{\rmdefault}{\mddefault}{\updefault}{2}}}}}
\put(8626,-136){\makebox(0,0)[lb]{\smash{{\SetFigFont{12}{14.4}{\rmdefault}{\mddefault}{\updefault}{$\Gamma_3^{15}$}}}}}
\put(8626,-961){\makebox(0,0)[lb]{\smash{{\SetFigFont{12}{14.4}{\rmdefault}{\mddefault}{\updefault}{$\Gamma_3^{16}$}}}}}
\put(8626,-1786){\makebox(0,0)[lb]{\smash{{\SetFigFont{12}{14.4}{\rmdefault}{\mddefault}{\updefault}{$\Gamma_3^{17}$}}}}}
\put(8626,-2611){\makebox(0,0)[lb]{\smash{{\SetFigFont{12}{14.4}{\rmdefault}{\mddefault}{\updefault}{$\Gamma_3^{18}$}}}}}
\put(8626,-3436){\makebox(0,0)[lb]{\smash{{\SetFigFont{12}{14.4}{\rmdefault}{\mddefault}{\updefault}{$\Gamma_3^{19}$}}}}}
\put(8626,-4261){\makebox(0,0)[lb]{\smash{{\SetFigFont{12}{14.4}{\rmdefault}{\mddefault}{\updefault}{$\Gamma_3^{20}$}}}}}
\put(8626,-5311){\makebox(0,0)[lb]{\smash{{\SetFigFont{12}{14.4}{\rmdefault}{\mddefault}{\updefault}{$\Gamma_3^{21}$}}}}}
\put(8626,-6361){\makebox(0,0)[lb]{\smash{{\SetFigFont{12}{14.4}{\rmdefault}{\mddefault}{\updefault}{$\Gamma_3^{22}$}}}}}
\put(9730,-6261){\makebox(0,0)[lb]{\smash{{\SetFigFont{12}{14.4}{\rmdefault}{\mddefault}{\updefault}{2}}}}}
\put(8351,-7561){\makebox(0,0)[lb]{\smash{{\SetFigFont{12}{14.4}{\rmdefault}{\mddefault}{\updefault}{$\Gamma_4^8$}}}}}
\put(8351,-8461){\makebox(0,0)[lb]{\smash{{\SetFigFont{12}{14.4}{\rmdefault}{\mddefault}{\updefault}{$\Gamma_4^9$}}}}}
\put(6976,-7036){\makebox(0,0)[lb]{\smash{{\SetFigFont{12}{14.4}{\rmdefault}{\mddefault}{\updefault}{2}}}}}
\put(7576,-7036){\makebox(0,0)[lb]{\smash{{\SetFigFont{12}{14.4}{\rmdefault}{\mddefault}{\updefault}{2}}}}}
\put(5626,-7186){\makebox(0,0)[lb]{\smash{{\SetFigFont{12}{14.4}{\rmdefault}{\mddefault}{\updefault}{$\Gamma_4^2$}}}}}
\put(6976,-6586){\makebox(0,0)[lb]{\smash{{\SetFigFont{12}{14.4}{\rmdefault}{\mddefault}{\updefault}{2}}}}}
\put(5626,-6736){\makebox(0,0)[lb]{\smash{{\SetFigFont{12}{14.4}{\rmdefault}{\mddefault}{\updefault}{$\Gamma_4^1$}}}}}
\put(6976,-8836){\makebox(0,0)[lb]{\smash{{\SetFigFont{12}{14.4}{\rmdefault}{\mddefault}{\updefault}{2}}}}}
\put(7576,-8836){\makebox(0,0)[lb]{\smash{{\SetFigFont{12}{14.4}{\rmdefault}{\mddefault}{\updefault}{3}}}}}
\put(7576,-8386){\makebox(0,0)[lb]{\smash{{\SetFigFont{12}{14.4}{\rmdefault}{\mddefault}{\updefault}{3}}}}}
\put(6976,-7486){\makebox(0,0)[lb]{\smash{{\SetFigFont{12}{14.4}{\rmdefault}{\mddefault}{\updefault}{2}}}}}
\put(7576,-7486){\makebox(0,0)[lb]{\smash{{\SetFigFont{12}{14.4}{\rmdefault}{\mddefault}{\updefault}{2}}}}}
\put(7576,-7936){\makebox(0,0)[lb]{\smash{{\SetFigFont{12}{14.4}{\rmdefault}{\mddefault}{\updefault}{3}}}}}
\put(7576,-9211){\makebox(0,0)[lb]{\smash{{\SetFigFont{12}{14.4}{\rmdefault}{\mddefault}{\updefault}{2}}}}}
\put(7576,-9661){\makebox(0,0)[lb]{\smash{{\SetFigFont{12}{14.4}{\rmdefault}{\mddefault}{\updefault}{2}}}}}
\put(5626,-7636){\makebox(0,0)[lb]{\smash{{\SetFigFont{12}{14.4}{\rmdefault}{\mddefault}{\updefault}{$\Gamma_4^3$}}}}}
\put(5626,-8086){\makebox(0,0)[lb]{\smash{{\SetFigFont{12}{14.4}{\rmdefault}{\mddefault}{\updefault}{$\Gamma_4^4$}}}}}
\put(5626,-8536){\makebox(0,0)[lb]{\smash{{\SetFigFont{12}{14.4}{\rmdefault}{\mddefault}{\updefault}{$\Gamma_4^5$}}}}}
\put(5626,-8986){\makebox(0,0)[lb]{\smash{{\SetFigFont{12}{14.4}{\rmdefault}{\mddefault}{\updefault}{$\Gamma_4^6$}}}}}
\put(5626,-9436){\makebox(0,0)[lb]{\smash{{\SetFigFont{12}{14.4}{\rmdefault}{\mddefault}{\updefault}{$\Gamma_4^7$}}}}}
{\put(4201, 89){\line( 1, 0){600}}
\put(4801, 89){\line( 1, 0){600}}}
{\put(6751, 89){\circle*{100}}}
{\put(7351, 89){\circle*{100}}}
{\put(7951, 89){\circle*{100}}}
{\put(6751,-361){\circle*{100}}}
{\put(7351,-361){\circle*{100}}}
{\put(7951,-361){\circle*{100}}}
{\put(6751,-811){\circle*{100}}}
{\put(7351,-811){\circle*{100}}}
{\put(7951,-811){\circle*{100}}}
{\put(6751,-1261){\circle*{100}}}
{\put(7351,-1261){\circle*{100}}}
{\put(7951,-1261){\circle*{100}}}
{\put(6751,-1711){\circle*{100}}}
{\put(7351,-1711){\circle*{100}}}
{\put(7951,-1711){\circle*{100}}}
{\put(6751,-2161){\circle*{100}}}
{\put(7351,-2161){\circle*{100}}}
{\put(7951,-2161){\circle*{100}}}
{\put(6751,-2611){\circle*{100}}}
{\put(7351,-2611){\circle*{100}}}
{\put(7951,-2611){\circle*{100}}}
{\put(6751,-3061){\circle*{100}}}
{\put(7351,-3061){\circle*{100}}}
{\put(7951,-3061){\circle*{100}}}
{\put(6751,-3511){\circle*{100}}}
{\put(7351,-3511){\circle*{100}}}
{\put(7951,-3511){\circle*{100}}}
{\put(6751,-3961){\circle*{100}}}
{\put(7351,-3961){\circle*{100}}}
{\put(7951,-3961){\circle*{100}}}
{\put(6751,-4411){\circle*{100}}}
{\put(7351,-4411){\circle*{100}}}
{\put(7951,-4411){\circle*{100}}}
{\put(6751,-4936){\circle*{100}}}
{\put(7351,-4936){\circle*{100}}}
{\put(7951,-4936){\circle*{100}}}
{\put(6751,-5461){\circle*{100}}}
{\put(7351,-5461){\circle*{100}}}
{\put(7951,-5461){\circle*{100}}}
{\put(6751,-5986){\circle*{100}}}
{\put(7351,-5986){\circle*{100}}}
{\put(7951,-5986){\circle*{100}}}
{\put(9601,-10936){\circle*{100}}}
{\put(9826,-9736){\circle*{100}}}
{\put(9826,-9136){\circle*{100}}}
{\put(9826,-10336){\circle*{100}}}
{\put(10051,-10936){\circle*{100}}}
{\put(9001,-9736){\circle*{100}}}
{\put(9001,-9136){\circle*{100}}}
{\put(9001,-10336){\circle*{100}}}
{\put(9226,-10936){\circle*{100}}}
{\put(8776,-10936){\circle*{100}}}
{\put(4201, 89){\circle*{100}}}
{\put(4801, 89){\circle*{100}}}
{\put(5401, 89){\circle*{100}}}
{\put(3001, 89){\circle*{100}}}
{\put(2401, 89){\circle*{100}}}
{\put(1801, 89){\circle*{100}}}
{\put(4201,-436){\circle*{100}}}
{\put(4801,-436){\circle*{100}}}
{\put(3001,-436){\circle*{100}}}
{\put(2401,-436){\circle*{100}}}
{\put(1801,-436){\circle*{100}}}
{\put(5401,-436){\circle*{100}}}
{\put(4201,-961){\circle*{100}}}
{\put(4801,-961){\circle*{100}}}
{\put(5401,-961){\circle*{100}}}
{\put(3001,-961){\circle*{100}}}
{\put(2401,-961){\circle*{100}}}
{\put(1801,-961){\circle*{100}}}
{\put(1801,-1561){\circle*{100}}}
{\put(2401,-1561){\circle*{100}}}
{\put(3001,-1561){\circle*{100}}}
{\put(4201,-1561){\circle*{100}}}
{\put(4801,-1561){\circle*{100}}}
{\put(5401,-1336){\circle*{100}}}
{\put(5401,-1786){\circle*{100}}}
{\put(1801,-2311){\circle*{100}}}
{\put(2401,-2311){\circle*{100}}}
{\put(3001,-2311){\circle*{100}}}
{\put(4201,-2311){\circle*{100}}}
{\put(4801,-2311){\circle*{100}}}
{\put(5401,-2086){\circle*{100}}}
{\put(5401,-2536){\circle*{100}}}
{\put(1801,-3211){\circle*{100}}}
{\put(2401,-3211){\circle*{100}}}
{\put(3001,-3211){\circle*{100}}}
{\put(3601,-3211){\circle*{100}}}
{\put(4201,-3211){\circle*{100}}}
{\put(3001,-2836){\circle*{100}}}
{\put(2401,-3961){\circle*{100}}}
{\put(1801,-3961){\circle*{100}}}
{\put(3001,-3961){\circle*{100}}}
{\put(3601,-3961){\circle*{100}}}
{\put(4201,-3961){\circle*{100}}}
{\put(3001,-3586){\circle*{100}}}
{\put(4801,-3961){\circle*{100}}}
{\put(1801,-4711){\circle*{100}}}
{\put(2401,-4711){\circle*{100}}}
{\put(3001,-4711){\circle*{100}}}
{\put(3601,-4711){\circle*{100}}}
{\put(4201,-4711){\circle*{100}}}
{\put(4801,-4711){\circle*{100}}}
{\put(5401,-4711){\circle*{100}}}
{\put(3001,-4336){\circle*{100}}}
{\put(1801,-5461){\circle*{100}}}
{\put(2401,-5461){\circle*{100}}}
{\put(3001,-5461){\circle*{100}}}
{\put(3601,-5461){\circle*{100}}}
{\put(2701,-5086){\circle*{100}}}
{\put(1801,-6211){\circle*{100}}}
{\put(2401,-6211){\circle*{100}}}
{\put(3001,-6211){\circle*{100}}}
{\put(3601,-6211){\circle*{100}}}
{\put(4201,-6211){\circle*{100}}}
{\put(2701,-5836){\circle*{100}}}
{\put(1801,-6961){\circle*{100}}}
{\put(2401,-6961){\circle*{100}}}
{\put(3001,-6961){\circle*{100}}}
{\put(3601,-6961){\circle*{100}}}
{\put(4201,-6961){\circle*{100}}}
{\put(4801,-6961){\circle*{100}}}
{\put(2701,-6586){\circle*{100}}}
% (4,10)
{\put(2401,-7536){\vector( 0,-1){375}}}
{\put(1801,-7911){\line( 1, 0){1200}}}
{\put(1801,-7911){\circle*{100}}}
{\put(3001,-7911){\circle*{100}}}
{\put(2401,-7536){\circle*{100}}}
{\put(2401,-7911){\circle*{100}}}
\put(1351,-7986){\makebox(0,0)[lb]{\smash{{\SetFigFont{12}{14.4}{\rmdefault}{\mddefault}{\updefault}{$\Gamma_4^{10}$}}}}}
\put(2176,-7836){\makebox(0,0)[lb]{\smash{{\SetFigFont{12}{14.4}{\rmdefault}{\mddefault}{\updefault}{2}}}}}
% (5,7)
{\put(1801,-8761){\line( 1, 0){1800}}}
{\put(2401,-8386){\vector( 0,-1){375}}}
{\put(1801,-8761){\circle*{100}}}
{\put(3001,-8761){\circle*{100}}}
{\put(3601,-8761){\circle*{100}}}
{\put(2401,-8386){\circle*{100}}}
{\put(2401,-8761){\circle*{100}}}
\put(2176,-8686){\makebox(0,0)[lb]{\smash{{\SetFigFont{12}{14.4}{\rmdefault}{\mddefault}{\updefault}{2}}}}}
\put(1351,-8836){\makebox(0,0)[lb]{\smash{{\SetFigFont{12}{14.4}{\rmdefault}{\mddefault}{\updefault}{$\Gamma_5^7$}}}}}
% (6,5)
{\put(1801,-9611){\line( 1, 0){2400}}}
{\put(2401,-9236){\vector( 0,-1){375}}}
\put(2176,-9536){\makebox(0,0)[lb]{\smash{{\SetFigFont{12}{14.4}{\rmdefault}{\mddefault}{\updefault}{2}}}}}
\put(1351,-9686){\makebox(0,0)[lb]{\smash{{\SetFigFont{12}{14.4}{\rmdefault}{\mddefault}{\updefault}{$\Gamma_6^5$}}}}}
{\put(1801,-9611){\circle*{100}}}
{\put(2401,-9611){\circle*{100}}}
{\put(3001,-9611){\circle*{100}}}
{\put(3601,-9611){\circle*{100}}}
{\put(2401,-9236){\circle*{100}}}
{\put(4201,-9611){\circle*{100}}}
% (6,1)
\put(1351,-10361){\makebox(0,0)[lb]{\smash{{\SetFigFont{12}{14.4}{\rmdefault}{\mddefault}{\updefault}{$\Gamma_6^1$}}}}}
\put(3881,-10211){\makebox(0,0)[lb]{\smash{{\SetFigFont{12}{14.4}{\rmdefault}{\mddefault}{\updefault}{2}}}}}
{\put(1801,-10286){\line( 1, 0){600}}}
{\put(2401,-10286){\line( 1, 0){600}}}
{\put(3001,-10286){\line( 1, 0){600}}}
{\put(4201,-10286){\vector(-1, 0){600}}}
{\put(4201,-10286){\line( 1, 0){600}}}
{\put(1801,-10286){\circle*{100}}}
{\put(2401,-10286){\circle*{100}}}
{\put(3001,-10286){\circle*{100}}}
{\put(3601,-10286){\circle*{100}}}
{\put(4201,-10286){\circle*{100}}}
{\put(4801,-10286){\circle*{100}}}
% (6,2)
\put(1351,-11011){\makebox(0,0)[lb]{\smash{{\SetFigFont{12}{14.4}{\rmdefault}{\mddefault}{\updefault}{$\Gamma_6^2$}}}}}
\put(3881,-10861){\makebox(0,0)[lb]{\smash{{\SetFigFont{12}{14.4}{\rmdefault}{\mddefault}{\updefault}{2}}}}}
{\put(1801,-10936){\line( 1, 0){600}}}
{\put(2401,-10936){\line( 1, 0){600}}}
{\put(3001,-10936){\line( 1, 0){600}}}
{\put(4201,-10936){\line( 1, 0){600}}}
{\put(3601,-10936){\vector( 1, 0){600}}}
{\put(1801,-10936){\circle*{100}}}
{\put(2401,-10936){\circle*{100}}}
{\put(3001,-10936){\circle*{100}}}
{\put(3601,-10936){\circle*{100}}}
{\put(4201,-10936){\circle*{100}}}
{\put(4801,-10936){\circle*{100}}}
{\put(5776,-10036){\circle*{100}}}
{\put(6376,-10036){\circle*{100}}}
{\put(6976,-10036){\circle*{100}}}
{\put(7576,-10036){\circle*{100}}}
{\put(8176,-10036){\circle*{100}}}
{\put(5776,-10486){\circle*{100}}}
{\put(6376,-10486){\circle*{100}}}
{\put(6976,-10486){\circle*{100}}}
{\put(7576,-10486){\circle*{100}}}
{\put(8176,-10486){\circle*{100}}}
{\put(5776,-10936){\circle*{100}}}
{\put(6376,-10936){\circle*{100}}}
{\put(6976,-10936){\circle*{100}}}
{\put(7576,-10936){\circle*{100}}}
{\put(8176,-10936){\circle*{100}}}
{\put(8776,-7486){\circle*{100}}}
{\put(9376,-7486){\circle*{100}}}
{\put(9976,-7261){\circle*{100}}}
{\put(9976,-7711){\circle*{100}}}
{\put(8776,-8386){\circle*{100}}}
{\put(9376,-8386){\circle*{100}}}
{\put(9976,-8611){\circle*{100}}}
{\put(9976,-8161){\circle*{100}}}
{\put(9526,239){\circle*{100}}}
{\put(9226,-361){\circle*{100}}}
{\put(9826,-361){\circle*{100}}}
{\put(9526,-586){\circle*{100}}}
{\put(9226,-1186){\circle*{100}}}
{\put(9826,-1186){\circle*{100}}}
{\put(9526,-1411){\circle*{100}}}
{\put(9226,-2011){\circle*{100}}}
{\put(9826,-2011){\circle*{100}}}
{\put(9226,-2836){\circle*{100}}}
{\put(9526,-3061){\circle*{100}}}
{\put(9226,-3661){\circle*{100}}}
{\put(9826,-3661){\circle*{100}}}
{\put(9526,-3886){\circle*{100}}}
{\put(9226,-4486){\circle*{100}}}
{\put(9826,-4486){\circle*{100}}}
{\put(9526,-4936){\circle*{100}}}
{\put(9226,-5536){\circle*{100}}}
{\put(9826,-5536){\circle*{100}}}
{\put(9526,-5986){\circle*{100}}}
{\put(9226,-6586){\circle*{100}}}
{\put(9826,-6586){\circle*{100}}}
{\put(9526,-2236){\circle*{100}}}
{\put(9826,-2836){\circle*{100}}}
{\put(7951,-6661){\circle*{100}}}
{\put(6151,-7111){\circle*{100}}}
{\put(6751,-7111){\circle*{100}}}
{\put(7951,-7111){\circle*{100}}}
{\put(6151,-6661){\circle*{100}}}
{\put(6751,-6661){\circle*{100}}}
{\put(7351,-6661){\circle*{100}}}
{\put(7951,-9361){\circle*{100}}}
{\put(6151,-9361){\circle*{100}}}
{\put(6751,-9361){\circle*{100}}}
{\put(6151,-8911){\circle*{100}}}
{\put(6751,-8911){\circle*{100}}}
{\put(7351,-8911){\circle*{100}}}
{\put(7951,-8911){\circle*{100}}}
{\put(6151,-8461){\circle*{100}}}
{\put(6751,-8461){\circle*{100}}}
{\put(7351,-8461){\circle*{100}}}
{\put(7951,-8461){\circle*{100}}}
{\put(6151,-7561){\circle*{100}}}
{\put(6751,-7561){\circle*{100}}}
{\put(7351,-7561){\circle*{100}}}
{\put(7951,-7561){\circle*{100}}}
{\put(6151,-8011){\circle*{100}}}
{\put(6751,-8011){\circle*{100}}}
{\put(7351,-8011){\circle*{100}}}
{\put(7951,-8011){\circle*{100}}}
{\put(7351,-7111){\circle*{100}}}
{\put(7351,-9361){\circle*{100}}}
\end{picture}
\caption{The Dynkin diagrams in irreducible \coscorfs of rank $\ge 3$}
\label{fig_dynkin}
\end{figure}

\subsection{Automorphism groups and planes}

\begin{remar}\label{rem_table}
We collect the following invariants in Table 1.

Let $\rsC_+=\{R_+^a \mid a \in A\}$ denote the set of root sets in the
objects of the Cartan scheme.
By identifying objects with the same roots one obtains a quotient Cartan scheme
of the simply connected Cartan scheme of the classification
(see \cite[Def.\ 3.1]{p-CH09b} for the definition of coverings).
This quotient has the minimal number of objects with respect to all quotients of the Cartan scheme.

In the fifth column we give the automorphism group of one (equivalently, any) object of this quotient
(this is $\Aut(\cC)$).

The last column contains a list of all Dynkin diagrams appearing in the Cartan scheme:
the number $i$ stands for the diagram $\Gamma_r^i$ of Fig.\ \ref{fig_dynkin}
if the root system is of rank $r$.

Remark that (except for the last column) the data for the Cartan schemes of
rank three is also in \cite[Table 1]{p-CH09c}. But notice that in \cite{p-CH09c}
the standard Cartan schemes and the ones from the infinite series were not omitted,
thus the scheme number $n$ here corresponds to the one labeled $n+5$ in \cite{p-CH09c}.
\end{remar}

\begin{center}
\begin{longtable}{|l l l l l l p{4cm}|}
\hline
r & Nr. & $|R_+^a|$ & $|\rsC_+|$ & $|A|$ & $\Aut(\cC)$ & Dynkin diagrams \\
\hline
\endhead
\hline
\endfoot
$3$ & $1$ & $10$ & $5$ & $60$ & $A_1 \times A_2 $ & $A$, $C$, $D'$, $6$ \\
$3$ & $2$ & $10$ & $10$ & $60$ & $A_2 $ & $A$, $B$, $C$, $15$ \\
$3$ & $3$ & $11$ & $9$ & $72$ & $A_1 \times A_1 \times A_1 $ & $A$, $B$, $C$, $6$, $15$ \\
$3$ & $4$ & $12$ & $21$ & $84$ & $A_1 \times A_1 $ & $A$, $B$, $C$, $2$, $6$, $15$, $20$ \\
$3$ & $5$ & $12$ & $14$ & $84$ & $A_2 $ & $A$, $C$, $6$, $15$ \\
$3$ & $6$ & $13$ & $4$ & $96$ & $G_2 \times A_1 $ & $A$, $B$, $2$, $15$ \\
$3$ & $7$ & $13$ & $12$ & $96$ & $A_1 \times A_1 \times A_1 $ & $A$, $B$, $C$, $D'$, $1$, $7$, $11$, $15$ \\
$3$ & $8$ & $13$ & $2$ & $96$ & $B_3 $ & $C$, $6$ \\
$3$ & $9$ & $13$ & $2$ & $96$ & $B_3 $ & $B$, $5$ \\
$3$ & $10$ & $14$ & $56$ & $112$ & $A_1 $ & $A$, $B$, $C$, $D'$, $2$, $6$, $7$, $11$, $15$, $16$, $20$ \\
$3$ & $11$ & $15$ & $16$ & $128$ & $A_1 \times A_1 \times A_1 $ & $A$, $C$, $6$, $7$, $15$, $20$ \\
$3$ & $12$ & $16$ & $36$ & $144$ & $A_1 \times A_1 $ & $A$, $B$, $C$, $2$, $6$, $7$, $15$, $18$, $20$ \\
$3$ & $13$ & $16$ & $24$ & $144$ & $A_2 $ & $A$, $B$, $C$, $1$, $5$, $11$, $15$, $19$ \\
$3$ & $14$ & $17$ & $10$ & $160$ & $A_1 \times B_2 $ & $A$, $B$, $C$, $D'$, $6$, $8$, $11$, $15$ \\
$3$ & $15$ & $17$ & $10$ & $160$ & $B_2 \times A_1 $ & $A$, $B$, $C$, $2$, $5$, $7$, $15$ \\
$3$ & $16$ & $17$ & $10$ & $160$ & $C_2 \times A_1 $ & $A$, $B$, $C$, $1$, $2$, $5$, $6$, $13$, $15$ \\
$3$ & $17$ & $18$ & $30$ & $180$ & $A_2 $ & $A$, $C$, $2$, $7$, $15$, $20$ \\
$3$ & $18$ & $18$ & $90$ & $180$ & $A_1 $ & $A$, $B$, $C$, $1$, $2$, $5$, $6$, $11$, $13$, $15$, $16$, $19$, $20$ \\
$3$ & $19$ & $19$ & $25$ & $200$ & $A_1 \times A_1 \times A_1 $ & $A$, $B$, $C$, $1$, $2$, $3$, $5$, $6$, $11$, $13$, $15$, $16$, $19$, $20$ \\
$3$ & $20$ & $19$ & $8$ & $192$ & $G_2 \times A_1 $ & $B$, $C$, $2$, $4$, $11$, $15$ \\
$3$ & $21$ & $19$ & $50$ & $200$ & $A_1 \times A_1$ & $A$, $B$, $C$, $1$, $2$, $5$, $6$, $11$, $13$, $15$, $16$, $19$, $21$ \\
$3$ & $22$ & $19$ & $25$ & $200$ & $A_1 \times A_1 \times A_1 $ & $A$, $B$, $C$, $1$, $2$, $5$, $6$, $11$, $13$, $15$, $16$, $19$ \\
$3$ & $23$ & $19$ & $8$ & $192$ & $G_2 \times A_1 $ & $B$, $C$, $1$, $6$, $7$, $11$ \\
$3$ & $24$ & $20$ & $27$ & $216$ & $C_2 $ & $A$, $B$, $C$, $D'$, $2$, $3$, $10$, $11$, $15$, $16$ \\
$3$ & $25$ & $20$ & $110$ & $220$ & $A_1 $ & $A$, $B$, $C$, $1$, $2$, $3$, $5$, $6$, $11$, $13$, $15$, $16$, $19$, $21$ \\
$3$ & $26$ & $20$ & $110$ & $220$ & $A_1 $ & $A$, $B$, $C$, $1$, $2$, $5$, $6$, $7$, $11$, $13$, $15$, $16$, $20$, $22$ \\
$3$ & $27$ & $21$ & $15$ & $240$ & $A_1 \times C_2 $ & $A$, $B$, $C$, $1$, $2$, $3$, $11$, $15$, $16$, $19$ \\
$3$ & $28$ & $21$ & $30$ & $240$ & $A_1 \times A_1 \times A_1 $ & $A$, $B$, $C$, $1$, $2$, $5$, $6$, $7$, $13$, $15$, $20$ \\
$3$ & $29$ & $21$ & $5$ & $240$ & $C_3 $ & $A$, $C$, $2$, $7$, $15$ \\
$3$ & $30$ & $22$ & $44$ & $264$ & $A_2 $ & $A$, $B$, $C$, $2$, $5$, $6$, $7$, $13$, $15$, $16$ \\
$3$ & $31$ & $25$ & $42$ & $336$ & $A_1 \times A_1 \times A_1 $ & $A$, $B$, $C$, $2$, $3$, $4$, $5$, $6$, $8$, $10$, $11$, $15$, $16$, $18$, $19$, $20$ \\
$3$ & $32$ & $25$ & $14$ & $336$ & $A_1 \times G_2 $ & $A$, $B$, $C$, $1$, $2$, $6$, $8$, $13$, $15$, $20$ \\
$3$ & $33$ & $25$ & $28$ & $336$ & $A_1 \times A_2$ & $A$, $B$, $C$, $D'$, $1$, $2$, $5$, $6$, $7$, $11$, $12$, $13$, $14$, $15$, $16$ \\
$3$ & $34$ & $25$ & $7$ & $336$ & $B_3 $ & $A$, $C$, $2$, $7$, $15$ \\
$3$ & $35$ & $26$ & $182$ & $364$ & $A_1 $ & $A$, $B$, $C$, $1$, $2$, $3$, $5$, $6$, $8$, $10$, $11$, $13$, $15$, $16$, $17$, $18$, $19$, $20$, $22$ \\
$3$ & $36$ & $26$ & $182$ & $364$ & $A_1 $ & $A$, $B$, $C$, $1$, $2$, $6$, $8$, $11$, $13$, $15$, $16$, $20$, $21$, $22$ \\
$3$ & $37$ & $27$ & $49$ & $392$ & $A_1 \times A_1 \times A_1 $ & $A$, $B$, $C$, $1$, $2$, $3$, $6$, $8$, $11$, $13$, $15$, $16$, $20$, $22$ \\
$3$ & $38$ & $27$ & $98$ & $392$ & $A_1 \times A_1$ & $A$, $B$, $C$, $1$, $2$, $5$, $6$, $8$, $11$, $13$, $15$, $16$, $20$, $21$, $22$ \\
$3$ & $39$ & $27$ & $98$ & $392$ & $A_1 \times A_1$ & $A$, $B$, $C$, $1$, $2$, $6$, $7$, $8$, $11$, $13$, $15$, $16$, $20$, $21$, $22$ \\
$3$ & $40$ & $28$ & $420$ & $420$ & $1$ & $A$, $B$, $C$, $1$, $2$, $3$, $5$, $6$, $7$, $8$, $11$, $13$, $15$, $16$, $20$, $21$, $22$ \\
$3$ & $41$ & $28$ & $210$ & $420$ & $A_1 $ & $A$, $B$, $C$, $1$, $2$, $5$, $6$, $7$, $8$, $11$, $13$, $15$, $16$, $21$, $22$ \\
$3$ & $42$ & $28$ & $70$ & $420$ & $A_2 $ & $A$, $B$, $C$, $2$, $5$, $6$, $8$, $13$, $15$, $16$, $21$ \\
$3$ & $43$ & $29$ & $56$ & $448$ & $A_1 \times A_1 \times A_1 $ & $A$, $B$, $C$, $2$, $3$, $5$, $6$, $7$, $8$, $11$, $13$, $15$, $16$, $20$, $22$ \\
$3$ & $44$ & $29$ & $112$ & $448$ & $A_1 \times A_1$ & $A$, $B$, $C$, $1$, $2$, $3$, $5$, $6$, $7$, $8$, $11$, $13$, $15$, $16$, $21$, $22$ \\
$3$ & $45$ & $29$ & $112$ & $448$ & $A_1 \times A_1$ & $A$, $B$, $C$, $2$, $3$, $5$, $6$, $7$, $8$, $11$, $13$, $15$, $16$, $21$, $22$ \\
$3$ & $46$ & $30$ & $238$ & $476$ & $A_1 $ & $A$, $B$, $C$, $2$, $3$, $5$, $6$, $7$, $8$, $11$, $13$, $15$, $16$, $21$, $22$ \\
$3$ & $47$ & $31$ & $21$ & $504$ & $A_1 \times G_2 $ & $A$, $B$, $C$, $D'$, $1$, $2$, $6$, $9$, $11$, $13$, $15$ \\
$3$ & $48$ & $31$ & $21$ & $504$ & $A_1 \times G_2 $ & $A$, $B$, $C$, $2$, $3$, $5$, $6$, $7$, $8$, $13$, $15$, $16$ \\
$3$ & $49$ & $34$ & $102$ & $612$ & $A_2 $ & $A$, $B$, $C$, $2$, $3$, $6$, $7$, $8$, $11$, $15$, $16$, $20$, $22$ \\
$3$ & $50$ & $37$ & $15$ & $720$ & $B_3 $ & $A$, $C$, $2$, $7$, $8$, $15$, $20$ \\
\hline
$4$ & $1$ & $15$ & $10$ & $360$ & $A_2 \times A_2 $ & $A$, $D'$ \\
$4$ & $2$ & $17$ & $10$ & $480$ & $B_3 $ & $A$, $D$, $D'$, $1$, $10$ \\
$4$ & $3$ & $18$ & $6$ & $576$ & $B_3 \times A_1 $ & $A$, $C$, $D'$, $1$ \\
$4$ & $4$ & $21$ & $36$ & $864$ & $A_3 $ & $A$, $B$, $C$, $D$, $D'$, $1$, $8$, $9$, $10$ \\
$4$ & $5$ & $22$ & $10$ & $960$ & $C_3 \times A_1 $ & $A$, $C$, $D$, $D'$, $1$, $3$, $10$ \\
$4$ & $6$ & $24$ & $1$ & $1152$ & $F_4 $ & $1$ \\
$4$ & $7$ & $25$ & $12$ & $1440$ & $A_4 $ & $A$, $B$, $C$, $D'$, $9$ \\
$4$ & $8$ & $28$ & $20$ & $1920$ & $B_3 \times A_1 $ & $A$, $B$, $C$, $D$, $D'$, $1$, $2$, $3$, $8$, $9$, $10$ \\
$4$ & $9$ & $30$ & $16$ & $2304$ & $G_2 \times G_2 $ & $A$, $D'$, $1$, $4$, $8$ \\
$4$ & $10$ & $32$ & $28$ & $2688$ & $B_3 \times A_1 $ & $A$, $B$, $C$, $D$, $D'$, $1$, $5$, $6$, $7$, $8$, $9$, $10$ \\
$4$ & $11$ & $32$ & $7$ & $2688$ & $B_4 $ & $A$, $C$, $D'$, $1$, $3$ \\
\hline
$5$ & $1$ & $25$ & $6$ & $4320$ & $A_5 $ & $A$, $D$, $6$ \\
$5$ & $2$ & $30$ & $12$ & $8640$ & $A_5 $ & $A$, $D$, $D'$, $6$ \\
$5$ & $3$ & $33$ & $15$ & $11520$ & $B_4 \times A_1 $ & $A$, $D$, $D'$, $1$, $2$, $6$, $7$ \\
$5$ & $4$ & $41$ & $7$ & $26880$ & $B_5 $ & $A$, $C$, $D$, $D'$, $1$, $6$ \\
$5$ & $5$ & $46$ & $56$ & $40320$ & $A_5 $ & $A$, $B$, $C$, $D$, $D'$, $2$, $4$, $5$, $6$, $7$ \\
$5$ & $6$ & $49$ & $21$ & $48384$ & $F_4 \times A_1 $ & $A$, $C$, $D$, $D'$, $1$, $2$, $3$, $6$, $7$ \\
\hline
$6$ & $1$ & $36$ & $1$ & $51840$ & $E_6 $ & $3$ \\
$6$ & $2$ & $46$ & $7$ & $161280$ & $D_6 $ & $A$, $D$, $3$, $4$ \\
$6$ & $3$ & $63$ & $14$ & $725760$ & $E_6 $ & $A$, $D'$, $3$, $4$ \\
$6$ & $4$ & $68$ & $21$ & $967680$ & $B_6 $ & $A$, $D$, $D'$, $1$, $2$, $3$, $4$, $5$ \\
\hline
$7$ & $1$ & $63$ & $1$ & $2903040$ & $E_7 $ & $1$ \\
$7$ & $2$ & $91$ & $8$ & $23224320$ & $E_7 $ & $A$, $D$, $1$, $2$ \\
\hline
$8$ & $1$ & $120$ & $1$ & $696729600$ & $E_8 $ & $ 1 $ \\
\end{longtable}
{\small Table 1: Invariants of sporadic \coscorfsn, see Rem.\ \ref{rem_table}}
\end{center}

\section{Irreducible root systems}
\label{ap:rs}

We give the roots in a multiplicative notation\footnote{We use the lexicographical ordering induced by
$\alpha_1>\alpha_2>\ldots>\alpha_r$. This is convenient because it
is the usual ordering in computer algebra systems. The index ``$r+1-i$''
ensures that the lists of roots start with $1,2,3,\ldots$}
to save space: For instance the word $\prod_{i=1}^r i^{x_i}$ corresponds to
$\sum_{i=1}^r x_i\alpha_{r+1-i}$.

Notice that we have chosen a ``canonical'' object for each
groupoid. Write $\pi(R^a_+)$ for the set $R^a_+$ where the coordinates
are permuted via $\pi\in S_r$. Then the set listed below is the minimum
of $\{\pi(R^a_+)\mid a\in A,\:\: \pi\in S_r\}$ with respect
to the lexicographical ordering on the sorted sequences of roots.

\subsection{Rank 3}
\hspace{12pt}\\

\baselineskip=10pt
\begin{tiny}
\noindent
Nr. $1$ with $10$ positive roots:
$1$, $2$, $3$, $12$, $13$, $1^{2}2$, $1^{2}3$, $123$, $1^{2}23$, $1^{3}23$\\
Nr. $2$ with $10$ positive roots:
$1$, $2$, $3$, $12$, $13$, $23$, $1^{2}2$, $123$, $1^{2}23$, $1^{2}2^{2}3$\\
Nr. $3$ with $11$ positive roots:
$1$, $2$, $3$, $12$, $13$, $1^{2}2$, $1^{2}3$, $123$, $1^{2}23$, $1^{3}23$, $1^{3}2^{2}3$\\
Nr. $4$ with $12$ positive roots:
$1$, $2$, $3$, $12$, $13$, $1^{2}2$, $123$, $1^{3}2$, $1^{2}23$, $1^{3}23$, $1^{3}2^{2}3$, $1^{4}2^{2}3$\\
Nr. $5$ with $12$ positive roots:
$1$, $2$, $3$, $12$, $13$, $1^{2}2$, $1^{2}3$, $123$, $1^{2}23$, $1^{3}23$, $1^{2}2^{2}3$, $1^{3}2^{2}3$\\
Nr. $6$ with $13$ positive roots:
$1$, $2$, $3$, $12$, $13$, $1^{2}2$, $123$, $1^{3}2$, $1^{2}23$, $1^{3}2^{2}$, $1^{3}23$, $1^{3}2^{2}3$, $1^{4}2^{2}3$\\
Nr. $7$ with $13$ positive roots:
$1$, $2$, $3$, $12$, $13$, $1^{2}2$, $1^{2}3$, $123$, $1^{3}2$, $1^{2}23$, $1^{3}23$, $1^{4}23$, $1^{4}2^{2}3$\\
Nr. $8$ with $13$ positive roots:
$1$, $2$, $3$, $12$, $13$, $1^{2}2$, $1^{2}3$, $123$, $1^{2}23$, $1^{3}23$, $1^{2}2^{2}3$, $1^{3}2^{2}3$, $1^{4}2^{2}3$\\
Nr. $9$ with $13$ positive roots:
$1$, $2$, $3$, $12$, $13$, $1^{2}2$, $123$, $13^{2}$, $1^{2}23$, $123^{2}$, $1^{2}23^{2}$, $1^{3}23^{2}$, $1^{3}2^{2}3^{2}$\\
Nr. $10$ with $14$ positive roots:
$1$, $2$, $3$, $12$, $13$, $1^{2}2$, $1^{2}3$, $123$, $1^{3}2$, $1^{2}23$, $1^{3}23$, $1^{4}23$, $1^{3}2^{2}3$, $1^{4}2^{2}3$\\
Nr. $11$ with $15$ positive roots:\\
$1$, $2$, $3$, $12$, $13$, $1^{2}2$, $1^{2}3$, $123$, $1^{3}2$, $1^{2}23$, $1^{3}23$, $1^{4}23$, $1^{3}2^{2}3$, $1^{4}2^{2}3$, $1^{5}2^{2}3$\\
Nr. $12$ with $16$ positive roots:\\
$1$, $2$, $3$, $12$, $13$, $1^{2}2$, $1^{2}3$, $123$, $1^{3}2$, $1^{2}23$, $1^{3}2^{2}$, $1^{3}23$, $1^{4}23$, $1^{3}2^{2}3$, $1^{4}2^{2}3$, $1^{5}2^{2}3$\\
Nr. $13$ with $16$ positive roots:\\
$1$, $2$, $3$, $12$, $23$, $1^{2}2$, $123$, $1^{3}2$, $1^{2}23$, $12^{2}3$, $1^{3}23$, $1^{2}2^{2}3$, $1^{3}2^{2}3$, $1^{4}2^{2}3$, $1^{4}2^{3}3$, $1^{4}2^{3}3^{2}$\\
Nr. $14$ with $17$ positive roots:\\
$1$, $2$, $3$, $12$, $13$, $1^{2}2$, $1^{2}3$, $123$, $1^{3}2$, $1^{2}23$, $1^{4}2$, $1^{3}23$, $1^{4}23$, $1^{5}23$, $1^{4}2^{2}3$, $1^{5}2^{2}3$, $1^{6}2^{2}3$\\
Nr. $15$ with $17$ positive roots:\\
$1$, $2$, $3$, $12$, $13$, $1^{2}2$, $1^{2}3$, $123$, $1^{3}2$, $1^{2}23$, $1^{3}2^{2}$, $1^{3}23$, $1^{4}23$, $1^{3}2^{2}3$, $1^{4}2^{2}3$, $1^{5}2^{2}3$, $1^{5}2^{2}3^{2}$\\
Nr. $16$ with $17$ positive roots:\\
$1$, $2$, $3$, $12$, $13$, $1^{2}2$, $123$, $1^{3}2$, $1^{2}23$, $1^{3}23$, $1^{2}2^{2}3$, $1^{3}2^{2}3$, $1^{4}2^{2}3$, $1^{5}2^{2}3$, $1^{5}2^{3}3$, $1^{5}2^{3}3^{2}$, $1^{6}2^{3}3^{2}$\\
Nr. $17$ with $18$ positive roots:\\
$1$, $2$, $3$, $12$, $13$, $1^{2}2$, $1^{2}3$, $123$, $1^{3}2$, $1^{2}23$, $1^{3}2^{2}$, $1^{3}23$, $1^{4}23$, $1^{3}2^{2}3$, $1^{4}2^{2}3$, $1^{5}2^{2}3$, $1^{5}2^{3}3$, $1^{6}2^{3}3$\\
Nr. $18$ with $18$ positive roots:\\
$1$, $2$, $3$, $12$, $13$, $23$, $1^{2}2$, $123$, $1^{3}2$, $1^{2}23$, $12^{2}3$, $1^{3}23$, $1^{2}2^{2}3$, $1^{3}2^{2}3$, $1^{4}2^{2}3$, $1^{3}2^{3}3$, $1^{4}2^{3}3$, $1^{4}2^{3}3^{2}$\\
Nr. $19$ with $19$ positive roots:\\
$1$, $2$, $3$, $12$, $13$, $1^{2}2$, $123$, $1^{3}2$, $1^{2}23$, $1^{4}2$, $1^{3}23$, $1^{4}23$, $1^{3}2^{2}3$, $1^{4}2^{2}3$, $1^{5}2^{2}3$, $1^{6}2^{2}3$, $1^{6}2^{3}3$, $1^{7}2^{3}3$, $1^{7}2^{3}3^{2}$\\
Nr. $20$ with $19$ positive roots:\\
$1$, $2$, $3$, $12$, $23$, $1^{2}2$, $123$, $1^{3}2$, $1^{2}23$, $1^{4}2$, $1^{3}23$, $1^{2}2^{2}3$, $1^{4}23$, $1^{3}2^{2}3$, $1^{4}2^{2}3$, $1^{5}2^{2}3$, $1^{6}2^{2}3$, $1^{6}2^{3}3$, $1^{6}2^{3}3^{2}$\\
Nr. $21$ with $19$ positive roots:\\
$1$, $2$, $3$, $12$, $13$, $23$, $1^{2}2$, $12^{2}$, $123$, $1^{3}2$, $1^{2}23$, $12^{2}3$, $1^{3}23$, $1^{2}2^{2}3$, $1^{3}2^{2}3$, $1^{4}2^{2}3$, $1^{3}2^{3}3$, $1^{4}2^{3}3$, $1^{4}2^{3}3^{2}$\\
Nr. $22$ with $19$ positive roots:\\
$1$, $2$, $3$, $12$, $13$, $23$, $1^{2}2$, $123$, $1^{3}2$, $1^{2}23$, $12^{2}3$, $1^{3}2^{2}$, $1^{3}23$, $1^{2}2^{2}3$, $1^{3}2^{2}3$, $1^{4}2^{2}3$, $1^{3}2^{3}3$, $1^{4}2^{3}3$, $1^{4}2^{3}3^{2}$\\
Nr. $23$ with $19$ positive roots:\\
$1$, $2$, $3$, $12$, $23$, $1^{2}2$, $123$, $1^{3}2$, $1^{2}23$, $1^{3}2^{2}$, $1^{3}23$, $1^{2}2^{2}3$, $1^{3}2^{2}3$, $1^{4}2^{2}3$, $1^{3}2^{3}3$, $1^{4}2^{3}3$, $1^{5}2^{3}3$, $1^{6}2^{3}3$, $1^{6}2^{4}3$\\
Nr. $24$ with $20$ positive roots:\\
$1$, $2$, $3$, $12$, $13$, $1^{2}2$, $123$, $1^{3}2$, $1^{2}23$, $1^{4}2$, $1^{3}2^{2}$, $1^{3}23$, $1^{4}23$, $1^{3}2^{2}3$, $1^{5}2^{2}$, $1^{4}2^{2}3$, $1^{5}2^{2}3$, $1^{6}2^{2}3$, $1^{6}2^{3}3$, $1^{7}2^{3}3$\\
Nr. $25$ with $20$ positive roots:\\
$1$, $2$, $3$, $12$, $13$, $1^{2}2$, $123$, $1^{3}2$, $1^{2}23$, $1^{4}2$, $1^{3}2^{2}$, $1^{3}23$, $1^{4}23$, $1^{3}2^{2}3$, $1^{4}2^{2}3$, $1^{5}2^{2}3$, $1^{6}2^{2}3$, $1^{6}2^{3}3$, $1^{7}2^{3}3$, $1^{7}2^{3}3^{2}$\\
Nr. $26$ with $20$ positive roots:\\
$1$, $2$, $3$, $12$, $13$, $1^{2}2$, $1^{2}3$, $123$, $1^{3}2$, $1^{2}23$, $1^{3}2^{2}$, $1^{3}23$, $1^{2}2^{2}3$, $1^{4}23$, $1^{3}2^{2}3$, $1^{4}2^{2}3$, $1^{5}2^{2}3$, $1^{4}2^{3}3$, $1^{5}2^{3}3$, $1^{6}2^{3}3^{2}$\\
Nr. $27$ with $21$ positive roots:\\
$1$, $2$, $3$, $12$, $13$, $1^{2}2$, $123$, $1^{3}2$, $1^{2}23$, $1^{4}2$, $1^{3}2^{2}$, $1^{3}23$, $1^{4}23$, $1^{3}2^{2}3$, $1^{5}2^{2}$, $1^{4}2^{2}3$, $1^{5}2^{2}3$, $1^{6}2^{2}3$, $1^{6}2^{3}3$, $1^{7}2^{3}3$, $1^{7}2^{3}3^{2}$\\
Nr. $28$ with $21$ positive roots:\\
$1$, $2$, $3$, $12$, $13$, $1^{2}2$, $1^{2}3$, $123$, $1^{3}2$, $1^{2}23$, $1^{3}2^{2}$, $1^{3}23$, $1^{2}2^{2}3$, $1^{4}23$, $1^{3}2^{2}3$, $1^{4}2^{2}3$, $1^{5}2^{2}3$, $1^{4}2^{3}3$, $1^{5}2^{3}3$, $1^{6}2^{3}3$, $1^{6}2^{3}3^{2}$\\
Nr. $29$ with $21$ positive roots:\\
$1$, $2$, $3$, $12$, $13$, $1^{2}2$, $1^{2}3$, $123$, $1^{3}2$, $1^{2}23$, $1^{3}2^{2}$, $1^{3}23$, $1^{4}23$, $1^{3}2^{2}3$, $1^{4}2^{2}3$, $1^{5}2^{2}3$, $1^{5}2^{3}3$, $1^{5}2^{2}3^{2}$, $1^{6}2^{3}3$, $1^{6}2^{3}3^{2}$, $1^{7}2^{3}3^{2}$\\
Nr. $30$ with $22$ positive roots:\\
$1$, $2$, $3$, $12$, $13$, $1^{2}2$, $1^{2}3$, $123$, $1^{3}2$, $1^{2}23$, $1^{3}2^{2}$, $1^{3}23$, $1^{2}2^{2}3$, $1^{4}23$, $1^{3}2^{2}3$, $1^{4}2^{2}3$, $1^{5}2^{2}3$, $1^{4}2^{3}3$, $1^{5}2^{3}3$, $1^{5}2^{2}3^{2}$, $1^{5}2^{3}3^{2}$, 
$1^{6}2^{3}3^{2}$\\
Nr. $31$ with $25$ positive roots:\\
$1$, $2$, $3$, $12$, $13$, $1^{2}2$, $1^{2}3$, $123$, $1^{3}2$, $1^{2}23$, $1^{4}2$, $1^{3}2^{2}$, $1^{3}23$, $1^{4}23$, $1^{3}2^{2}3$, $1^{5}2^{2}$, $1^{5}23$, $1^{4}2^{2}3$, $1^{5}2^{2}3$, $1^{6}2^{2}3$, $1^{7}2^{2}3$, $1^{6}2^{3}3$, $1^{7}2^{3}3$, 
$1^{8}2^{3}3$, $1^{8}2^{3}3^{2}$\\
Nr. $32$ with $25$ positive roots:\\
$1$, $2$, $3$, $12$, $13$, $1^{2}2$, $1^{2}3$, $123$, $1^{3}2$, $1^{2}23$, $1^{4}2$, $1^{3}23$, $1^{4}23$, $1^{3}2^{2}3$, $1^{5}23$, $1^{4}2^{2}3$, $1^{5}2^{2}3$, $1^{6}2^{2}3$, $1^{7}2^{2}3$, $1^{6}2^{3}3$, $1^{7}2^{3}3$, $1^{8}2^{3}3$, 
$1^{7}2^{3}3^{2}$, $1^{8}2^{3}3^{2}$, $1^{9}2^{3}3^{2}$\\
Nr. $33$ with $25$ positive roots:\\
$1$, $2$, $3$, $12$, $13$, $1^{2}2$, $1^{2}3$, $12^{2}$, $123$, $1^{3}2$, $1^{2}23$, $12^{2}3$, $1^{3}23$, $1^{2}2^{2}3$, $1^{4}23$, $1^{3}2^{2}3$, $1^{4}2^{2}3$, $1^{3}2^{3}3$, $1^{3}2^{2}3^{2}$, $1^{4}2^{3}3$, $1^{5}2^{3}3$, $1^{4}2^{3}3^{2}$, 
$1^{5}2^{3}3^{2}$, $1^{6}2^{3}3^{2}$, $1^{7}2^{4}3^{2}$\\
Nr. $34$ with $25$ positive roots:\\
$1$, $2$, $3$, $12$, $13$, $1^{2}2$, $1^{2}3$, $123$, $1^{3}2$, $1^{2}23$, $1^{3}2^{2}$, $1^{3}23$, $1^{2}2^{2}3$, $1^{4}23$, $1^{3}2^{2}3$, $1^{4}2^{2}3$, $1^{5}2^{2}3$, $1^{4}2^{3}3$, $1^{5}2^{3}3$, $1^{5}2^{2}3^{2}$, $1^{6}2^{3}3$, $1^{5}2^{3}3^{2}$, 
$1^{6}2^{3}3^{2}$, $1^{7}2^{3}3^{2}$, $1^{7}2^{4}3^{2}$\\
Nr. $35$ with $26$ positive roots:\\
$1$, $2$, $3$, $12$, $13$, $1^{2}2$, $1^{2}3$, $123$, $1^{3}2$, $1^{2}23$, $1^{4}2$, $1^{3}2^{2}$, $1^{3}23$, $1^{4}23$, $1^{3}2^{2}3$, $1^{5}2^{2}$, $1^{5}23$, $1^{4}2^{2}3$, $1^{5}2^{2}3$, $1^{6}2^{2}3$, $1^{7}2^{2}3$, $1^{6}2^{3}3$, $1^{7}2^{3}3$, 
$1^{8}2^{3}3$, $1^{7}2^{3}3^{2}$, $1^{8}2^{3}3^{2}$\\
Nr. $36$ with $26$ positive roots:\\
$1$, $2$, $3$, $12$, $13$, $1^{2}2$, $1^{2}3$, $123$, $1^{3}2$, $1^{2}23$, $1^{4}2$, $1^{3}2^{2}$, $1^{3}23$, $1^{4}23$, $1^{3}2^{2}3$, $1^{5}23$, $1^{4}2^{2}3$, $1^{5}2^{2}3$, $1^{6}2^{2}3$, $1^{7}2^{2}3$, $1^{6}2^{3}3$, $1^{7}2^{3}3$, $1^{8}2^{3}3$, 
$1^{7}2^{3}3^{2}$, $1^{8}2^{3}3^{2}$, $1^{9}2^{3}3^{2}$\\
Nr. $37$ with $27$ positive roots:\\
$1$, $2$, $3$, $12$, $13$, $1^{2}2$, $1^{2}3$, $123$, $1^{3}2$, $1^{2}23$, $1^{4}2$, $1^{3}2^{2}$, $1^{3}23$, $1^{4}23$, $1^{3}2^{2}3$, $1^{5}2^{2}$, $1^{5}23$, $1^{4}2^{2}3$, $1^{5}2^{2}3$, $1^{6}2^{2}3$, $1^{7}2^{2}3$, $1^{6}2^{3}3$, $1^{7}2^{3}3$, 
$1^{8}2^{3}3$, $1^{7}2^{3}3^{2}$, $1^{8}2^{3}3^{2}$, $1^{9}2^{3}3^{2}$\\
Nr. $38$ with $27$ positive roots:\\
$1$, $2$, $3$, $12$, $13$, $1^{2}2$, $1^{2}3$, $123$, $1^{3}2$, $1^{2}23$, $1^{4}2$, $1^{3}2^{2}$, $1^{3}23$, $1^{4}23$, $1^{3}2^{2}3$, $1^{5}23$, $1^{4}2^{2}3$, $1^{5}2^{2}3$, $1^{6}2^{2}3$, $1^{5}2^{2}3^{2}$, $1^{7}2^{2}3$, $1^{6}2^{3}3$, $1^{7}2^{3}3$,
$1^{8}2^{3}3$, $1^{7}2^{3}3^{2}$, $1^{8}2^{3}3^{2}$, $1^{9}2^{3}3^{2}$\\
Nr. $39$ with $27$ positive roots:\\
$1$, $2$, $3$, $12$, $13$, $1^{2}2$, $1^{2}3$, $123$, $1^{3}2$, $1^{2}23$, $1^{4}2$, $1^{3}2^{2}$, $1^{3}23$, $1^{4}23$, $1^{3}2^{2}3$, $1^{5}23$, $1^{4}2^{2}3$, $1^{5}2^{2}3$, $1^{6}2^{2}3$, $1^{7}2^{2}3$, $1^{6}2^{3}3$, $1^{7}2^{3}3$, $1^{7}2^{2}3^{2}$,
$1^{8}2^{3}3$, $1^{7}2^{3}3^{2}$, $1^{8}2^{3}3^{2}$, $1^{9}2^{3}3^{2}$\\
Nr. $40$ with $28$ positive roots:\\
$1$, $2$, $3$, $12$, $13$, $1^{2}2$, $1^{2}3$, $123$, $1^{3}2$, $1^{2}23$, $1^{4}2$, $1^{3}2^{2}$, $1^{3}23$, $1^{4}23$, $1^{3}2^{2}3$, $1^{5}2^{2}$, $1^{5}23$, $1^{4}2^{2}3$, $1^{5}2^{2}3$, $1^{6}2^{2}3$, $1^{5}2^{2}3^{2}$, $1^{7}2^{2}3$, $1^{6}2^{3}3$, 
$1^{7}2^{3}3$, $1^{8}2^{3}3$, $1^{7}2^{3}3^{2}$, $1^{8}2^{3}3^{2}$, $1^{9}2^{3}3^{2}$\\
Nr. $41$ with $28$ positive roots:\\
$1$, $2$, $3$, $12$, $13$, $1^{2}2$, $1^{2}3$, $123$, $1^{3}2$, $1^{2}23$, $1^{4}2$, $1^{3}2^{2}$, $1^{3}23$, $1^{4}23$, $1^{3}2^{2}3$, $1^{5}23$, $1^{4}2^{2}3$, $1^{5}2^{2}3$, $1^{6}2^{2}3$, $1^{5}2^{2}3^{2}$, $1^{7}2^{2}3$, $1^{6}2^{3}3$, $1^{7}2^{3}3$,
$1^{8}2^{3}3$, $1^{7}2^{3}3^{2}$, $1^{8}2^{3}3^{2}$, $1^{9}2^{3}3^{2}$, $1^{9}2^{4}3^{2}$\\
Nr. $42$ with $28$ positive roots:\\
$1$, $2$, $3$, $12$, $13$, $1^{2}2$, $1^{2}3$, $123$, $1^{3}2$, $1^{2}23$, $1^{4}2$, $1^{3}2^{2}$, $1^{3}23$, $1^{4}23$, $1^{3}2^{2}3$, $1^{5}23$, $1^{4}2^{2}3$, $1^{5}2^{2}3$, $1^{6}2^{2}3$, $1^{5}2^{2}3^{2}$, $1^{7}2^{2}3$, $1^{6}2^{3}3$, $1^{7}2^{3}3$,
$1^{8}2^{3}3$, $1^{7}2^{3}3^{2}$, $1^{8}2^{3}3^{2}$, $1^{9}2^{3}3^{2}$, $1^{11}2^{4}3^{2}$\\
Nr. $43$ with $29$ positive roots:\\
$1$, $2$, $3$, $12$, $13$, $1^{2}2$, $1^{2}3$, $123$, $1^{3}2$, $1^{2}23$, $1^{4}2$, $1^{3}2^{2}$, $1^{3}23$, $1^{4}23$, $1^{3}2^{2}3$, $1^{5}2^{2}$, $1^{5}23$, $1^{4}2^{2}3$, $1^{5}2^{2}3$, $1^{6}2^{2}3$, $1^{5}2^{2}3^{2}$, $1^{7}2^{2}3$, $1^{6}2^{3}3$, 
$1^{7}2^{3}3$, $1^{7}2^{2}3^{2}$, $1^{8}2^{3}3$, $1^{7}2^{3}3^{2}$, $1^{8}2^{3}3^{2}$, $1^{9}2^{3}3^{2}$\\
Nr. $44$ with $29$ positive roots:\\
$1$, $2$, $3$, $12$, $13$, $1^{2}2$, $1^{2}3$, $123$, $1^{3}2$, $1^{2}23$, $1^{4}2$, $1^{3}2^{2}$, $1^{3}23$, $1^{4}23$, $1^{3}2^{2}3$, $1^{5}2^{2}$, $1^{5}23$, $1^{4}2^{2}3$, $1^{5}2^{2}3$, $1^{6}2^{2}3$, $1^{5}2^{2}3^{2}$, $1^{7}2^{2}3$, $1^{6}2^{3}3$, 
$1^{7}2^{3}3$, $1^{8}2^{3}3$, $1^{7}2^{3}3^{2}$, $1^{8}2^{3}3^{2}$, $1^{9}2^{3}3^{2}$, $1^{9}2^{4}3^{2}$\\
Nr. $45$ with $29$ positive roots:\\
$1$, $2$, $3$, $12$, $13$, $1^{2}2$, $1^{2}3$, $123$, $1^{3}2$, $1^{2}23$, $1^{4}2$, $1^{3}2^{2}$, $1^{3}23$, $1^{4}23$, $1^{3}2^{2}3$, $1^{5}2^{2}$, $1^{5}23$, $1^{4}2^{2}3$, $1^{5}2^{2}3$, $1^{6}2^{2}3$, $1^{5}2^{2}3^{2}$, $1^{7}2^{2}3$, $1^{6}2^{3}3$, 
$1^{7}2^{3}3$, $1^{8}2^{3}3$, $1^{7}2^{3}3^{2}$, $1^{8}2^{3}3^{2}$, $1^{9}2^{3}3^{2}$, $1^{11}2^{4}3^{2}$\\
Nr. $46$ with $30$ positive roots:\\
$1$, $2$, $3$, $12$, $13$, $1^{2}2$, $1^{2}3$, $123$, $1^{3}2$, $1^{2}23$, $1^{4}2$, $1^{3}2^{2}$, $1^{3}23$, $1^{4}23$, $1^{3}2^{2}3$, $1^{5}2^{2}$, $1^{5}23$, $1^{4}2^{2}3$, $1^{5}2^{2}3$, $1^{6}2^{2}3$, $1^{5}2^{2}3^{2}$, $1^{7}2^{2}3$, $1^{6}2^{3}3$, 
$1^{7}2^{3}3$, $1^{7}2^{2}3^{2}$, $1^{8}2^{3}3$, $1^{7}2^{3}3^{2}$, $1^{8}2^{3}3^{2}$, $1^{9}2^{3}3^{2}$, $1^{9}2^{4}3^{2}$\\
Nr. $47$ with $31$ positive roots:\\
$1$, $2$, $3$, $12$, $13$, $1^{2}2$, $1^{2}3$, $123$, $1^{3}2$, $1^{2}23$, $1^{4}2$, $1^{3}23$, $1^{5}2$, $1^{4}23$, $1^{6}2$, $1^{5}23$, $1^{4}2^{2}3$, $1^{6}23$, $1^{5}2^{2}3$, $1^{7}23$, $1^{6}2^{2}3$, $1^{7}2^{2}3$, $1^{8}2^{2}3$, $1^{9}2^{2}3$, 
$1^{10}2^{2}3$, $1^{9}2^{3}3$, $1^{10}2^{3}3$, $1^{11}2^{3}3$, $1^{10}2^{3}3^{2}$, $1^{11}2^{3}3^{2}$, $1^{12}2^{3}3^{2}$\\
Nr. $48$ with $31$ positive roots:\\
$1$, $2$, $3$, $12$, $13$, $1^{2}2$, $1^{2}3$, $123$, $1^{3}2$, $1^{2}23$, $1^{4}2$, $1^{3}2^{2}$, $1^{3}23$, $1^{4}23$, $1^{3}2^{2}3$, $1^{5}2^{2}$, $1^{5}23$, $1^{4}2^{2}3$, $1^{5}2^{2}3$, $1^{6}2^{2}3$, $1^{5}2^{2}3^{2}$, $1^{7}2^{2}3$, $1^{6}2^{3}3$, 
$1^{7}2^{3}3$, $1^{7}2^{2}3^{2}$, $1^{8}2^{3}3$, $1^{7}2^{3}3^{2}$, $1^{8}2^{3}3^{2}$, $1^{9}2^{3}3^{2}$, $1^{9}2^{4}3^{2}$, $1^{11}2^{4}3^{2}$\\
Nr. $49$ with $34$ positive roots:\\
$1$, $2$, $3$, $12$, $13$, $1^{2}2$, $1^{2}3$, $123$, $1^{3}2$, $1^{2}23$, $1^{4}2$, $1^{3}2^{2}$, $1^{3}23$, $1^{4}23$, $1^{3}2^{2}3$, $1^{5}2^{2}$, $1^{5}23$, $1^{4}2^{2}3$, $1^{5}2^{2}3$, $1^{6}2^{2}3$, $1^{5}2^{3}3$, $1^{7}2^{2}3$, $1^{6}2^{3}3$, 
$1^{7}2^{3}3$, $1^{8}2^{3}3$, $1^{7}2^{3}3^{2}$, $1^{8}2^{4}3$, $1^{8}2^{3}3^{2}$, $1^{9}2^{4}3$, $1^{9}2^{3}3^{2}$, $1^{9}2^{4}3^{2}$, $1^{11}2^{4}3^{2}$, $1^{11}2^{5}3^{2}$, $1^{12}2^{5}3^{2}$\\
Nr. $50$ with $37$ positive roots:\\
$1$, $2$, $3$, $12$, $13$, $1^{2}2$, $1^{2}3$, $123$, $1^{3}2$, $1^{2}23$, $1^{4}2$, $1^{3}2^{2}$, $1^{3}23$, $1^{4}23$, $1^{3}2^{2}3$, $1^{5}2^{2}$, $1^{5}23$, $1^{4}2^{2}3$, $1^{5}2^{2}3$, $1^{6}2^{2}3$, $1^{5}2^{3}3$, $1^{7}2^{2}3$, $1^{6}2^{3}3$, 
$1^{7}2^{3}3$, $1^{8}2^{3}3$, $1^{7}2^{3}3^{2}$, $1^{9}2^{3}3$, $1^{8}2^{4}3$, $1^{8}2^{3}3^{2}$, $1^{9}2^{4}3$, $1^{9}2^{3}3^{2}$, $1^{10}2^{4}3$, $1^{9}2^{4}3^{2}$, $1^{11}2^{4}3^{2}$, $1^{11}2^{5}3^{2}$, $1^{12}2^{5}3^{2}$, $1^{13}2^{5}3^{2}$
\end{tiny}

\subsection{Rank 4}
\hspace{12pt}\\

\baselineskip=10pt
\begin{tiny}
\noindent
Nr. $1$ with $15$ positive roots:\\
$1$, $2$, $3$, $4$, $12$, $13$, $14$, $23$, $123$, $124$, $134$, $1^{2}24$, $1234$, $1^{2}234$, $1^{2}2^{2}34$\\
Nr. $2$ with $17$ positive roots:\\
$1$, $2$, $3$, $4$, $12$, $13$, $14$, $1^{2}2$, $123$, $124$, $134$, $1^{2}23$, $1^{2}24$, $1234$, $1^{2}234$, $1^{3}234$, 
$1^{3}2^{2}34$\\
Nr. $3$ with $18$ positive roots:\\
$1$, $2$, $3$, $4$, $12$, $13$, $24$, $1^{2}2$, $123$, $124$, $1^{2}23$, $1^{2}24$, $1234$, $1^{2}2^{2}4$, $1^{2}234$, 
$1^{2}2^{2}34$, $1^{3}2^{2}34$, $1^{3}2^{2}3^{2}4$\\
Nr. $4$ with $21$ positive roots:\\
$1$, $2$, $3$, $4$, $12$, $13$, $14$, $23$, $1^{2}2$, $123$, $124$, $134$, $1^{2}23$, $1^{2}24$, $1234$, $1^{2}2^{2}3$, $1^{2}234$,
$1^{3}234$, $1^{2}2^{2}34$, $1^{3}2^{2}34$, $1^{3}2^{2}3^{2}4$\\
Nr. $5$ with $22$ positive roots:\\
$1$, $2$, $3$, $4$, $12$, $13$, $24$, $1^{2}2$, $1^{2}3$, $123$, $124$, $1^{2}23$, $1^{2}24$, $1234$, $1^{3}23$, $1^{2}2^{2}4$, 
$1^{2}234$, $1^{3}234$, $1^{2}2^{2}34$, $1^{3}2^{2}34$, $1^{4}2^{2}34$, $1^{4}2^{2}3^{2}4$\\
Nr. $6$ with $24$ positive roots (type $F_4$):\\
$1$, $2$, $3$, $4$, $12$, $13$, $24$, $1^{2}2$, $123$, $124$, $1^{2}23$, $1^{2}24$, $1234$, $1^{2}2^{2}4$, $1^{2}23^{2}$, 
$1^{2}234$, $1^{2}2^{2}34$, $1^{2}23^{2}4$, $1^{3}2^{2}34$, $1^{2}2^{2}3^{2}4$, $1^{3}2^{2}3^{2}4$, $1^{4}2^{2}3^{2}4$, 
$1^{4}2^{3}3^{2}4$, $1^{4}2^{3}3^{2}4^{2}$\\
Nr. $7$ with $25$ positive roots:\\
$1$, $2$, $3$, $4$, $12$, $13$, $23$, $34$, $1^{2}2$, $123$, $134$, $234$, $1^{2}23$, $1234$, $13^{2}4$, $1^{2}2^{2}3$, $1^{2}234$,
$123^{2}4$, $1^{2}2^{2}34$, $1^{2}23^{2}4$, $1^{3}23^{2}4$, $1^{2}2^{2}3^{2}4$, $1^{3}2^{2}3^{2}4$, $1^{3}2^{2}3^{3}4$, 
$1^{3}2^{2}3^{3}4^{2}$\\
Nr. $8$ with $28$ positive roots:\\
$1$, $2$, $3$, $4$, $12$, $13$, $34$, $1^{2}2$, $1^{2}3$, $123$, $134$, $1^{2}23$, $1^{2}34$, $1234$, $1^{3}23$, $1^{2}234$, 
$1^{2}3^{2}4$, $1^{3}2^{2}3$, $1^{3}234$, $1^{2}23^{2}4$, $1^{3}2^{2}34$, $1^{3}23^{2}4$, $1^{4}23^{2}4$, $1^{3}2^{2}3^{2}4$, 
$1^{4}2^{2}3^{2}4$, $1^{5}2^{2}3^{2}4$, $1^{5}2^{2}3^{3}4$, $1^{5}2^{2}3^{3}4^{2}$\\
Nr. $9$ with $30$ positive roots:\\
$1$, $2$, $3$, $4$, $12$, $13$, $34$, $1^{2}2$, $123$, $134$, $1^{3}2$, $1^{2}23$, $1234$, $1^{3}2^{2}$, $1^{3}23$, $1^{2}234$, 
$1^{3}2^{2}3$, $1^{3}234$, $1^{2}23^{2}4$, $1^{4}2^{2}3$, $1^{3}2^{2}34$, $1^{3}23^{2}4$, $1^{4}2^{2}34$, $1^{3}2^{2}3^{2}4$, 
$1^{4}2^{2}3^{2}4$, $1^{5}2^{2}3^{2}4$, $1^{5}2^{3}3^{2}4$, $1^{6}2^{3}3^{2}4$, $1^{6}2^{3}3^{3}4$, $1^{6}2^{3}3^{3}4^{2}$\\
Nr. $10$ with $32$ positive roots:\\
$1$, $2$, $3$, $4$, $12$, $13$, $34$, $1^{2}2$, $1^{2}3$, $123$, $134$, $1^{3}2$, $1^{2}23$, $1^{2}34$, $1234$, $1^{3}23$, 
$1^{2}234$, $1^{2}3^{2}4$, $1^{4}23$, $1^{3}234$, $1^{2}23^{2}4$, $1^{4}2^{2}3$, $1^{4}234$, $1^{3}23^{2}4$, $1^{4}2^{2}34$, 
$1^{4}23^{2}4$, $1^{5}23^{2}4$, $1^{4}2^{2}3^{2}4$, $1^{5}2^{2}3^{2}4$, $1^{6}2^{2}3^{2}4$, $1^{6}2^{2}3^{3}4$, 
$1^{6}2^{2}3^{3}4^{2}$\\
Nr. $11$ with $32$ positive roots:\\
$1$, $2$, $3$, $4$, $12$, $13$, $24$, $1^{2}2$, $1^{2}3$, $123$, $124$, $1^{2}23$, $1^{2}24$, $1234$, $1^{3}23$, $1^{2}2^{2}3$, 
$1^{2}2^{2}4$, $1^{2}234$, $1^{3}2^{2}3$, $1^{3}234$, $1^{2}2^{2}34$, $1^{4}2^{2}3$, $1^{3}2^{2}34$, $1^{4}2^{2}34$, 
$1^{3}2^{3}34$, $1^{4}2^{3}34$, $1^{4}2^{2}3^{2}4$, $1^{5}2^{3}34$, $1^{4}2^{3}3^{2}4$, $1^{5}2^{3}3^{2}4$, $1^{6}2^{3}3^{2}4$, 
$1^{6}2^{4}3^{2}4$
\end{tiny}

\subsection{Rank 5}
\hspace{12pt}\\

\baselineskip=10pt
\begin{tiny}
\noindent
Nr. $1$ with $25$ positive roots:\\
$1$, $2$, $3$, $4$, $5$, $12$, $13$, $14$, $23$, $25$, $123$, $124$, $125$, $134$, $235$, $1234$, $1235$, $1245$, $1^{2}234$, 
$12^{2}35$, $12345$, $1^{2}2345$, $12^{2}345$, $1^{2}2^{2}345$, $1^{2}2^{2}3^{2}45$\\
Nr. $2$ with $30$ positive roots:\\
$1$, $2$, $3$, $4$, $5$, $12$, $13$, $14$, $23$, $35$, $123$, $124$, $134$, $135$, $235$, $1^{2}24$, $1234$, $1235$, $1345$, 
$1^{2}234$, $123^{2}5$, $12345$, $1^{2}2^{2}34$, $1^{2}2345$, $123^{2}45$, $1^{2}2^{2}345$, $1^{2}23^{2}45$, $1^{2}2^{2}3^{2}45$, 
$1^{3}2^{2}3^{2}45$, $1^{3}2^{2}3^{2}4^{2}5$\\
Nr. $3$ with $33$ positive roots:\\
$1$, $2$, $3$, $4$, $5$, $12$, $13$, $14$, $35$, $1^{2}2$, $123$, $124$, $134$, $135$, $1^{2}23$, $1^{2}24$, $1234$, $1235$, 
$1345$, $1^{2}234$, $1^{2}235$, $12345$, $1^{3}234$, $1^{2}23^{2}5$, $1^{2}2345$, $1^{3}2^{2}34$, $1^{3}2345$, $1^{2}23^{2}45$, 
$1^{3}2^{2}345$, $1^{3}23^{2}45$, $1^{3}2^{2}3^{2}45$, $1^{4}2^{2}3^{2}45$, $1^{4}2^{2}3^{2}4^{2}5$\\
Nr. $4$ with $41$ positive roots:\\
$1$, $2$, $3$, $4$, $5$, $12$, $13$, $24$, $45$, $1^{2}2$, $123$, $124$, $245$, $1^{2}23$, $1^{2}24$, $1234$, $1245$, 
$1^{2}2^{2}4$, $1^{2}234$, $1^{2}245$, $12345$, $1^{2}2^{2}34$, $1^{2}2^{2}45$, $1^{2}2345$, $1^{3}2^{2}34$, $1^{2}2^{2}345$, 
$1^{2}2^{2}4^{2}5$, $1^{3}2^{2}3^{2}4$, $1^{3}2^{2}345$, $1^{2}2^{2}34^{2}5$, $1^{3}2^{2}3^{2}45$, $1^{3}2^{2}34^{2}5$, 
$1^{3}2^{3}34^{2}5$, $1^{3}2^{2}3^{2}4^{2}5$, $1^{4}2^{3}34^{2}5$, $1^{3}2^{3}3^{2}4^{2}5$, $1^{4}2^{3}3^{2}4^{2}5$, 
$1^{5}2^{3}3^{2}4^{2}5$, $1^{5}2^{4}3^{2}4^{2}5$, $1^{5}2^{4}3^{2}4^{3}5$, $1^{5}2^{4}3^{2}4^{3}5^{2}$\\
Nr. $5$ with $46$ positive roots:\\
$1$, $2$, $3$, $4$, $5$, $12$, $13$, $14$, $23$, $45$, $1^{2}2$, $123$, $124$, $134$, $145$, $1^{2}23$, $1^{2}24$, $1234$, $1245$, 
$1345$, $1^{2}2^{2}3$, $1^{2}234$, $1^{2}245$, $12345$, $1^{3}234$, $1^{2}2^{2}34$, $1^{2}2345$, $1^{2}24^{2}5$, $1^{3}2^{2}34$, 
$1^{3}2345$, $1^{2}2^{2}345$, $1^{2}234^{2}5$, $1^{3}2^{2}3^{2}4$, $1^{3}2^{2}345$, $1^{3}234^{2}5$, $1^{2}2^{2}34^{2}5$, 
$1^{3}2^{2}3^{2}45$, $1^{3}2^{2}34^{2}5$, $1^{4}2^{2}34^{2}5$, $1^{3}2^{2}3^{2}4^{2}5$, $1^{4}2^{3}34^{2}5$, 
$1^{4}2^{2}3^{2}4^{2}5$, $1^{4}2^{3}3^{2}4^{2}5$, $1^{5}2^{3}3^{2}4^{2}5$, $1^{5}2^{3}3^{2}4^{3}5$, $1^{5}2^{3}3^{2}4^{3}5^{2}$\\
Nr. $6$ with $49$ positive roots:\\
$1$, $2$, $3$, $4$, $5$, $12$, $13$, $24$, $45$, $1^{2}2$, $1^{2}3$, $123$, $124$, $245$, $1^{2}23$, $1^{2}24$, $1234$, $1245$, 
$1^{3}23$, $1^{2}2^{2}4$, $1^{2}234$, $1^{2}245$, $12345$, $1^{3}234$, $1^{2}2^{2}34$, $1^{2}2^{2}45$, $1^{2}2345$, $1^{3}2^{2}34$,
$1^{3}2345$, $1^{2}2^{2}345$, $1^{2}2^{2}4^{2}5$, $1^{4}2^{2}34$, $1^{3}2^{2}345$, $1^{2}2^{2}34^{2}5$, $1^{4}2^{2}3^{2}4$, 
$1^{4}2^{2}345$, $1^{3}2^{2}34^{2}5$, $1^{4}2^{2}3^{2}45$, $1^{4}2^{2}34^{2}5$, $1^{3}2^{3}34^{2}5$, $1^{4}2^{3}34^{2}5$, 
$1^{4}2^{2}3^{2}4^{2}5$, $1^{5}2^{3}34^{2}5$, $1^{4}2^{3}3^{2}4^{2}5$, $1^{5}2^{3}3^{2}4^{2}5$, $1^{6}2^{3}3^{2}4^{2}5$, 
$1^{6}2^{4}3^{2}4^{2}5$, $1^{6}2^{4}3^{2}4^{3}5$, $1^{6}2^{4}3^{2}4^{3}5^{2}$
\end{tiny}

\subsection{Rank 6}
\hspace{12pt}\\

\baselineskip=10pt
\begin{tiny}
\noindent
Nr. $1$ with $36$ positive roots (type $E_6$):\\
$1$, $2$, $3$, $4$, $5$, $6$, $12$, $13$, $14$, $25$, $36$, $123$, $124$, $125$, $134$, $136$, $1234$, $1235$, $1236$, $1245$, 
$1346$, $1^{2}234$, $12345$, $12346$, $12356$, $1^{2}2345$, $1^{2}2346$, $123456$, $1^{2}2^{2}345$, $1^{2}23^{2}46$, $1^{2}23456$, 
$1^{2}2^{2}3456$, $1^{2}23^{2}456$, $1^{2}2^{2}3^{2}456$, $1^{3}2^{2}3^{2}456$, $1^{3}2^{2}3^{2}4^{2}56$\\
Nr. $2$ with $46$ positive roots:\\
$1$, $2$, $3$, $4$, $5$, $6$, $12$, $13$, $14$, $23$, $25$, $46$, $123$, $124$, $125$, $134$, $146$, $235$, $1234$, $1235$, $1245$,
$1246$, $1346$, $1^{2}234$, $12^{2}35$, $12345$, $12346$, $12456$, $1^{2}2345$, $1^{2}2346$, $12^{2}345$, $123456$, 
$1^{2}2^{2}345$, $1^{2}234^{2}6$, $1^{2}23456$, $12^{2}3456$, $1^{2}2^{2}3^{2}45$, $1^{2}2^{2}3456$, $1^{2}234^{2}56$, 
$1^{2}2^{2}3^{2}456$, $1^{2}2^{2}34^{2}56$, $1^{3}2^{2}34^{2}56$, $1^{2}2^{2}3^{2}4^{2}56$, $1^{3}2^{2}3^{2}4^{2}56$, 
$1^{3}2^{3}3^{2}4^{2}56$, $1^{3}2^{3}3^{2}4^{2}5^{2}6$\\
Nr. $3$ with $63$ positive roots:\\
$1$, $2$, $3$, $4$, $5$, $6$, $12$, $13$, $14$, $23$, $35$, $56$, $123$, $124$, $134$, $135$, $235$, $356$, $1^{2}24$, $1234$, 
$1235$, $1345$, $1356$, $2356$, $1^{2}234$, $123^{2}5$, $12345$, $12356$, $13456$, $1^{2}2^{2}34$, $1^{2}2345$, $123^{2}45$, 
$123^{2}56$, $123456$, $1^{2}2^{2}345$, $1^{2}23^{2}45$, $1^{2}23456$, $123^{2}456$, $123^{2}5^{2}6$, $1^{2}2^{2}3^{2}45$, 
$1^{2}2^{2}3456$, $1^{2}23^{2}456$, $123^{2}45^{2}6$, $1^{3}2^{2}3^{2}45$, $1^{2}2^{2}3^{2}456$, $1^{2}23^{2}45^{2}6$, 
$1^{3}2^{2}3^{2}4^{2}5$, $1^{3}2^{2}3^{2}456$, $1^{2}2^{2}3^{2}45^{2}6$, $1^{2}23^{3}45^{2}6$, $1^{3}2^{2}3^{2}4^{2}56$, 
$1^{3}2^{2}3^{2}45^{2}6$, $1^{2}2^{2}3^{3}45^{2}6$, $1^{3}2^{2}3^{3}45^{2}6$, $1^{3}2^{2}3^{2}4^{2}5^{2}6$, 
$1^{3}2^{3}3^{3}45^{2}6$, $1^{3}2^{2}3^{3}4^{2}5^{2}6$, $1^{4}2^{2}3^{3}4^{2}5^{2}6$, $1^{3}2^{3}3^{3}4^{2}5^{2}6$, 
$1^{4}2^{3}3^{3}4^{2}5^{2}6$, $1^{4}2^{3}3^{4}4^{2}5^{2}6$, $1^{4}2^{3}3^{4}4^{2}5^{3}6$, $1^{4}2^{3}3^{4}4^{2}5^{3}6^{2}$\\
Nr. $4$ with $68$ positive roots:\\
$1$, $2$, $3$, $4$, $5$, $6$, $12$, $13$, $14$, $35$, $56$, $1^{2}2$, $123$, $124$, $134$, $135$, $356$, $1^{2}23$, $1^{2}24$, 
$1234$, $1235$, $1345$, $1356$, $1^{2}234$, $1^{2}235$, $12345$, $12356$, $13456$, $1^{3}234$, $1^{2}23^{2}5$, $1^{2}2345$, 
$1^{2}2356$, $123456$, $1^{3}2^{2}34$, $1^{3}2345$, $1^{2}23^{2}45$, $1^{2}23^{2}56$, $1^{2}23456$, $1^{3}2^{2}345$, 
$1^{3}23^{2}45$, $1^{3}23456$, $1^{2}23^{2}456$, $1^{2}23^{2}5^{2}6$, $1^{3}2^{2}3^{2}45$, $1^{3}2^{2}3456$, $1^{3}23^{2}456$, 
$1^{2}23^{2}45^{2}6$, $1^{4}2^{2}3^{2}45$, $1^{3}2^{2}3^{2}456$, $1^{3}23^{2}45^{2}6$, $1^{4}2^{2}3^{2}4^{2}5$, 
$1^{4}2^{2}3^{2}456$, $1^{3}2^{2}3^{2}45^{2}6$, $1^{3}23^{3}45^{2}6$, $1^{4}2^{2}3^{2}4^{2}56$, $1^{4}2^{2}3^{2}45^{2}6$, 
$1^{3}2^{2}3^{3}45^{2}6$, $1^{4}2^{2}3^{3}45^{2}6$, $1^{4}2^{2}3^{2}4^{2}5^{2}6$, $1^{5}2^{2}3^{3}45^{2}6$, 
$1^{4}2^{2}3^{3}4^{2}5^{2}6$, $1^{5}2^{3}3^{3}45^{2}6$, $1^{5}2^{2}3^{3}4^{2}5^{2}6$, $1^{5}2^{3}3^{3}4^{2}5^{2}6$, 
$1^{6}2^{3}3^{3}4^{2}5^{2}6$, $1^{6}2^{3}3^{4}4^{2}5^{2}6$, $1^{6}2^{3}3^{4}4^{2}5^{3}6$, $1^{6}2^{3}3^{4}4^{2}5^{3}6^{2}$
\end{tiny}

\subsection{Rank 7}
\hspace{12pt}\\

\baselineskip=10pt
\begin{tiny}
\noindent
Nr. $1$ with $63$ positive roots (type $E_7$):\\
$1$, $2$, $3$, $4$, $5$, $6$, $7$, $12$, $13$, $14$, $25$, $36$, $57$, $123$, $124$, $125$, $134$, $136$, $257$, $1234$, $1235$, 
$1236$, $1245$, $1257$, $1346$, $1^{2}234$, $12345$, $12346$, $12356$, $12357$, $12457$, $1^{2}2345$, $1^{2}2346$, $123456$, 
$123457$, $123567$, $1^{2}2^{2}345$, $1^{2}23^{2}46$, $1^{2}23456$, $1^{2}23457$, $1234567$, $1^{2}2^{2}3456$, $1^{2}2^{2}3457$, 
$1^{2}23^{2}456$, $1^{2}234567$, $1^{2}2^{2}3^{2}456$, $1^{2}2^{2}345^{2}7$, $1^{2}2^{2}34567$, $1^{2}23^{2}4567$, 
$1^{3}2^{2}3^{2}456$, $1^{2}2^{2}3^{2}4567$, $1^{2}2^{2}345^{2}67$, $1^{3}2^{2}3^{2}4^{2}56$, $1^{3}2^{2}3^{2}4567$, 
$1^{2}2^{2}3^{2}45^{2}67$, $1^{3}2^{2}3^{2}4^{2}567$, $1^{3}2^{2}3^{2}45^{2}67$, $1^{3}2^{3}3^{2}45^{2}67$, 
$1^{3}2^{2}3^{2}4^{2}5^{2}67$, $1^{3}2^{3}3^{2}4^{2}5^{2}67$, $1^{4}2^{3}3^{2}4^{2}5^{2}67$, $1^{4}2^{3}3^{3}4^{2}5^{2}67$, 
$1^{4}2^{3}3^{3}4^{2}5^{2}6^{2}7$\\
Nr. $2$ with $91$ positive roots:\\
$1$, $2$, $3$, $4$, $5$, $6$, $7$, $12$, $13$, $14$, $23$, $25$, $46$, $67$, $123$, $124$, $125$, $134$, $146$, $235$, $467$, 
$1234$, $1235$, $1245$, $1246$, $1346$, $1467$, $1^{2}234$, $12^{2}35$, $12345$, $12346$, $12456$, $12467$, $13467$, $1^{2}2345$, 
$1^{2}2346$, $12^{2}345$, $123456$, $123467$, $124567$, $1^{2}2^{2}345$, $1^{2}234^{2}6$, $1^{2}23456$, $1^{2}23467$, $12^{2}3456$,
$1234567$, $1^{2}2^{2}3^{2}45$, $1^{2}2^{2}3456$, $1^{2}234^{2}56$, $1^{2}234^{2}67$, $1^{2}234567$, $12^{2}34567$, 
$1^{2}2^{2}3^{2}456$, $1^{2}2^{2}34^{2}56$, $1^{2}2^{2}34567$, $1^{2}234^{2}567$, $1^{2}234^{2}6^{2}7$, $1^{3}2^{2}34^{2}56$, 
$1^{2}2^{2}3^{2}4^{2}56$, $1^{2}2^{2}3^{2}4567$, $1^{2}2^{2}34^{2}567$, $1^{2}234^{2}56^{2}7$, $1^{3}2^{2}3^{2}4^{2}56$, 
$1^{3}2^{2}34^{2}567$, $1^{2}2^{2}3^{2}4^{2}567$, $1^{2}2^{2}34^{2}56^{2}7$, $1^{3}2^{3}3^{2}4^{2}56$, $1^{3}2^{2}3^{2}4^{2}567$, 
$1^{3}2^{2}34^{2}56^{2}7$, $1^{2}2^{2}3^{2}4^{2}56^{2}7$, $1^{3}2^{3}3^{2}4^{2}5^{2}6$, $1^{3}2^{3}3^{2}4^{2}567$, 
$1^{3}2^{2}3^{2}4^{2}56^{2}7$, $1^{3}2^{2}34^{3}56^{2}7$, $1^{3}2^{3}3^{2}4^{2}5^{2}67$, $1^{3}2^{3}3^{2}4^{2}56^{2}7$, 
$1^{3}2^{2}3^{2}4^{3}56^{2}7$, $1^{4}2^{2}3^{2}4^{3}56^{2}7$, $1^{3}2^{3}3^{2}4^{3}56^{2}7$, $1^{3}2^{3}3^{2}4^{2}5^{2}6^{2}7$, 
$1^{4}2^{3}3^{2}4^{3}56^{2}7$, $1^{3}2^{3}3^{2}4^{3}5^{2}6^{2}7$, $1^{4}2^{3}3^{3}4^{3}56^{2}7$, $1^{4}2^{3}3^{2}4^{3}5^{2}6^{2}7$,
$1^{4}2^{4}3^{2}4^{3}5^{2}6^{2}7$, $1^{4}2^{3}3^{3}4^{3}5^{2}6^{2}7$, $1^{4}2^{4}3^{3}4^{3}5^{2}6^{2}7$, 
$1^{5}2^{4}3^{3}4^{3}5^{2}6^{2}7$, $1^{5}2^{4}3^{3}4^{4}5^{2}6^{2}7$, $1^{5}2^{4}3^{3}4^{4}5^{2}6^{3}7$, 
$1^{5}2^{4}3^{3}4^{4}5^{2}6^{3}7^{2}$
\end{tiny}

\subsection{Rank 8}
\hspace{12pt}\\

\baselineskip=10pt
\begin{tiny}
\noindent
Nr. $1$ with $120$ positive roots (type $E_8$):\\
$1$, $2$, $3$, $4$, $5$, $6$, $7$, $8$, $12$, $13$, $14$, $25$, $36$, $57$, $78$, $123$, $124$, $125$, $134$, $136$, $257$, $578$, 
$1234$, $1235$, $1236$, $1245$, $1257$, $1346$, $2578$, $1^{2}234$, $12345$, $12346$, $12356$, $12357$, $12457$, $12578$, 
$1^{2}2345$, $1^{2}2346$, $123456$, $123457$, $123567$, $123578$, $124578$, $1^{2}2^{2}345$, $1^{2}23^{2}46$, $1^{2}23456$, 
$1^{2}23457$, $1234567$, $1234578$, $1235678$, $1^{2}2^{2}3456$, $1^{2}2^{2}3457$, $1^{2}23^{2}456$, $1^{2}234567$, $1^{2}234578$, 
$12345678$, $1^{2}2^{2}3^{2}456$, $1^{2}2^{2}345^{2}7$, $1^{2}2^{2}34567$, $1^{2}2^{2}34578$, $1^{2}23^{2}4567$, $1^{2}2345678$, 
$1^{3}2^{2}3^{2}456$, $1^{2}2^{2}3^{2}4567$, $1^{2}2^{2}345^{2}67$, $1^{2}2^{2}345^{2}78$, $1^{2}2^{2}345678$, $1^{2}23^{2}45678$, 
$1^{3}2^{2}3^{2}4^{2}56$, $1^{3}2^{2}3^{2}4567$, $1^{2}2^{2}3^{2}45^{2}67$, $1^{2}2^{2}3^{2}45678$, $1^{2}2^{2}345^{2}678$, 
$1^{2}2^{2}345^{2}7^{2}8$, $1^{3}2^{2}3^{2}4^{2}567$, $1^{3}2^{2}3^{2}45^{2}67$, $1^{3}2^{2}3^{2}45678$, 
$1^{2}2^{2}3^{2}45^{2}678$, $1^{2}2^{2}345^{2}67^{2}8$, $1^{3}2^{3}3^{2}45^{2}67$, $1^{3}2^{2}3^{2}4^{2}5^{2}67$, 
$1^{3}2^{2}3^{2}4^{2}5678$, $1^{3}2^{2}3^{2}45^{2}678$, $1^{2}2^{2}3^{2}45^{2}67^{2}8$, $1^{3}2^{3}3^{2}4^{2}5^{2}67$, 
$1^{3}2^{3}3^{2}45^{2}678$, $1^{3}2^{2}3^{2}4^{2}5^{2}678$, $1^{3}2^{2}3^{2}45^{2}67^{2}8$, $1^{4}2^{3}3^{2}4^{2}5^{2}67$, 
$1^{3}2^{3}3^{2}4^{2}5^{2}678$, $1^{3}2^{3}3^{2}45^{2}67^{2}8$, $1^{3}2^{2}3^{2}4^{2}5^{2}67^{2}8$, $1^{4}2^{3}3^{3}4^{2}5^{2}67$, 
$1^{4}2^{3}3^{2}4^{2}5^{2}678$, $1^{3}2^{3}3^{2}4^{2}5^{2}67^{2}8$, $1^{3}2^{3}3^{2}45^{3}67^{2}8$, 
$1^{4}2^{3}3^{3}4^{2}5^{2}6^{2}7$, $1^{4}2^{3}3^{3}4^{2}5^{2}678$, $1^{4}2^{3}3^{2}4^{2}5^{2}67^{2}8$, 
$1^{3}2^{3}3^{2}4^{2}5^{3}67^{2}8$, $1^{4}2^{3}3^{3}4^{2}5^{2}6^{2}78$, $1^{4}2^{3}3^{3}4^{2}5^{2}67^{2}8$, 
$1^{4}2^{3}3^{2}4^{2}5^{3}67^{2}8$, $1^{4}2^{4}3^{2}4^{2}5^{3}67^{2}8$, $1^{4}2^{3}3^{3}4^{2}5^{3}67^{2}8$, 
$1^{4}2^{3}3^{3}4^{2}5^{2}6^{2}7^{2}8$, $1^{4}2^{4}3^{3}4^{2}5^{3}67^{2}8$, $1^{4}2^{3}3^{3}4^{2}5^{3}6^{2}7^{2}8$, 
$1^{5}2^{4}3^{3}4^{2}5^{3}67^{2}8$, $1^{4}2^{4}3^{3}4^{2}5^{3}6^{2}7^{2}8$, $1^{5}2^{4}3^{3}4^{3}5^{3}67^{2}8$, 
$1^{5}2^{4}3^{3}4^{2}5^{3}6^{2}7^{2}8$, $1^{5}2^{4}3^{4}4^{2}5^{3}6^{2}7^{2}8$, $1^{5}2^{4}3^{3}4^{3}5^{3}6^{2}7^{2}8$, 
$1^{5}2^{4}3^{4}4^{3}5^{3}6^{2}7^{2}8$, $1^{6}2^{4}3^{4}4^{3}5^{3}6^{2}7^{2}8$, $1^{6}2^{5}3^{4}4^{3}5^{3}6^{2}7^{2}8$, 
$1^{6}2^{5}3^{4}4^{3}5^{4}6^{2}7^{2}8$, $1^{6}2^{5}3^{4}4^{3}5^{4}6^{2}7^{3}8$, $1^{6}2^{5}3^{4}4^{3}5^{4}6^{2}7^{3}8^{2}$
\end{tiny}

\end{appendix}

% \bibliographystyle{amsalpha}
% \bibliography{refs}

\providecommand{\bysame}{\leavevmode\hbox to3em{\hrulefill}\thinspace}

\end{document}